\documentclass{amsart}
\usepackage[all]{xy}
\usepackage{pstricks}

\newtheorem{theorem}{Theorem}
\newtheorem{lemma}{Lemma}
\newtheorem{proposition}{Proposition}
\newtheorem{corollary}{Corollary}

\newtheorem{definition}{Definition}

\theoremstyle{remark}
\newtheorem{remark}{Remark}[section]
\numberwithin{equation}{section}

\newcommand{\C}{{\mathbb C}}       
\newcommand{\R}{{\mathbb R}}       

\newcommand{\internalcomment}[1]{}

\begin{document}

\title[A criterion for the reconstruction of a holomorphic function]
{A criterion for the explicit reconstruction of a holomorphic function from its restrictions on lines}

\author{Amadeo Irigoyen}

\address{Universitat de Barcelona, Gran Via de les Corts Catalanes, 585,
08007 Barcelona, Spain}

\email{
\begin{minipage}[t]{5cm}
amadeo.irigoyen@ub.edu
\end{minipage}
}

\begin{abstract}

We deal with a problem of the explicit reconstruction
of any holomorphic function $f$ on
$\mathbb{C}^2$ from its restricions on a union of
complex lines. The validity of such a reconstruction
essentially depends on the mutual repartition of these lines,
condition that can be analytically described.
The motivation of 
this problem comes also from possible applications 
in mathematical economics and medical imaging.

\end{abstract}

\maketitle

\tableofcontents

\section{Introduction}

\subsection{Presentation of the problem}

In this paper we deal with a problem of the reconstruction of
a holomorphic function from its restrictions on analytic
subvarieties. The general presentation is the following:
let $f$ be a holomorphic function on a domain
$\Omega\subset\C^n$ and
$\{Z_j\}_{j=1}^N$ a family of analytic subvarieties
of $\Omega$. We assume that we just know the data
$f_{|\{Z_j\}_{j=1}^N}:=\{f_{|Z_j}\}_{j=1}^N$
and we want to find $f$. One can
give interpolation functions
$f_N\in\mathcal{O}\left(\Omega\right)$ that satisfy
${f_N}_{|\{Z_j\}_{j=1}^N}=f_{|\{Z_j\}_{j=1}^N}$
(see~\cite{berndtsson}),
but generally $f_N\neq f$ (since there exist
nonzero holomorphic functions that vanish on a finite
union of analytic subvarieties). Then a natural way is
to consider an infinite family of subvarieties
$\{Z_j\}_{j=1}^{\infty}$ and construct the associate interpolating
$f_{\{Z_j\}_{j\geq1}}$ as
$\lim_{N\rightarrow\infty}f_N$. In this case the uniqueness
of the interpolating function will certainly be made sure
but now without any guaranty of the convergence of the sequence
$\left(f_N\right)_{N\geq1}$.
Moreover, in case of positive result,
we are motivated to give explicit reconstruction formula.

In this paper we will deal with the special case of
$\C^2$ and a family of distinct complex lines that cross the origin.
Such a family can be described as
\begin{eqnarray}\label{lines}
\left\{
\{z\in\C^2,\;z_1-\eta_jz_2=0\}
\right\}_{j\geq1}
\,,
\end{eqnarray}
with $\eta_j\in\C$ all different, that we will simply
denote by
$\eta=\{\eta_j\}_{j\geq1}$ (without loss of generality
we can forget the line $\{z_2=0\}$ that is associate to
$\eta_0=\infty$). On the other hand, we consider an entire
function $f\in\mathcal{O}\left(\C^2\right)$
and assume that we know all its restrictions
$f_j\in\mathcal{O}(\C),\,\forall\,j\geq1$, where
\begin{eqnarray}
\{f_j\}_{j\geq1}
& := & 
\left\{
f_{|\{z_1=\eta_jz_2\}}
\right\}_{j\geq1}
\,.
\end{eqnarray}
A way to give the interpolation function
$f_N$ is the one that uses
the following relation that is proved in~\cite{irigoyen3}
by using residues and principal values:
$\mathbb{B}_2$ (resp. $\mathbb{S}_2$) being the unit ball
$\{z\in\C^2,\,|z_1|^2+|z_2|^2<1\}$
(resp. unit sphere $\{\,|z_1|^2+|z_2|^2=1\}$), one has,
$\forall\,z\in\mathbb{B}_2$,
\begin{eqnarray*}\nonumber
f(z)
& = &
\lim_{\varepsilon\rightarrow0}
\frac{1}{(2i\pi)^2}
\int_{\zeta\in\mathbb{S}_2,
\left|
\prod_{j=1}^N
(\zeta_1-\eta_j\zeta_2)
\right|
=\varepsilon}
\frac{
f(\zeta)
\det{\left(\overline{\zeta},P_N(\zeta,z)\right)}
\,\omega(\zeta)
}
{\prod_{j=1}^N(\zeta_1-\eta_j\zeta_2)
(1-<\overline{\zeta},z>)}
\\
& &
-\;
\lim_{\varepsilon\rightarrow0}
\frac{\prod_{j=1}^N(z_1-\eta_j z_2)}{(2i\pi)^2}
\int_{\zeta\in\mathbb{S}_2,
\left|
\prod_{j=1}^N
(\zeta_1-\eta_j\zeta_2)
\right|
>\varepsilon}
\frac{f(\zeta)\,\omega'\left(\overline{\zeta}\right)\wedge\omega(\zeta)}
{\prod_{j=1}^N(z_1-\eta_j z_2)
(1-<\overline{\zeta},z>)^2}
\,,
\end{eqnarray*}
with
\begin{eqnarray*}
\begin{cases}
\omega'(\zeta)=
\zeta_1d\zeta_2-\zeta_2d\zeta_1\,,
\\
\omega(\zeta)=d\zeta_1\wedge d\zeta_2\,,
\end{cases}
\end{eqnarray*}
and
$P_N(\zeta,z)\in
\left(
\mathcal{O}\left(\C^2\right)
\right)^2$
satisfying, for all
$(\zeta,z)\in\C^2\times\C^2$, 
\begin{eqnarray*}
<P_N(\zeta),\zeta-z>
& = &
P_{N,1}(\zeta,z)(\zeta_1-z_1)
+P_{N,2}(\zeta,z)(\zeta_2-z_2)
\\
& = &
\prod_{j=1}^N
(\zeta_1-\eta_j\zeta_2)
-
\prod_{j=1}^N
(z_1-\eta_j z_2)
\,.
\end{eqnarray*}
Both integrals can be explicited and yield to
(Section~\ref{preliminar}, Lemma~\ref{prelimformula}):
$\forall\,z\in\C^2$,
\begin{eqnarray}\label{relation}
f(z) & = &
E_N(f;\eta)(z)
-R_N(f;\eta)(z)
+\sum_{k+l\geq N}
a_{k,l}z_1^kz_2^l
\,,
\end{eqnarray}
where
$\sum_{k,l\geq0}a_{k,l}z_1^kz_2^l$ is
the Taylor expansion of $f$,
\begin{eqnarray}\label{defEN}\nonumber
E_N(f;\eta)(z)
& := &
\sum_{p=1}^N
\left(
\prod_{j=p+1}^N
(z_1-\eta_j z_2)
\right)
\sum_{q=p}^N
\frac{1+\eta_p\overline{\eta_q}}{1+|\eta_q|^2}
\frac{1}
{\prod_{j=p,j\neq q}^N(\eta_q-\eta_j)}
\times
\\
& &
\times
\sum_{m\geq N-p}
\left(
\frac{z_2+\overline{\eta_q}z_1}{1+|\eta_q|^2}
\right)^{m-N+p}
\frac{1}{m!}
\frac{\partial^m}{\partial v^m}|_{v=0}
[f(\eta_qv,v)]
\,.
\end{eqnarray}
and
\begin{eqnarray}\label{defRN}
& &
R_N(f;\eta)(z):=
\sum_{p=1}^N
\left(
\prod_{j=1,j\neq p}^N\frac{z_1-\eta_jz_2}{\eta_p-\eta_j}
\right)
\sum_{k+l\geq N}a_{k,l}\eta_p^k
\left(
\frac{z_2+\overline{\eta_p}z_1}{1+|\eta_p|^2}
\right)^{k+l-N+1}.
\end{eqnarray}

The formula $E_N(f;\eta)$ is an explicit interpolation formula
constructed with the data
$(f_1,\ldots,f_N)$: it is an entire function that
coincides with $f$ on the first $N$ lines, ie
$\forall\,p=1,\ldots,N$,
\begin{eqnarray*}
\left(E_N(f;\eta)\right)_{|\{z_1=\eta_pz_2\}}
& = &
f_{|\{z_1=\eta_pz_2\}}
\,.
\end{eqnarray*}
As $N\rightarrow\infty$, the function
$f-E_N(f;\eta)$ will be a holomorphic function that will vanish
on an increasing number of lines. If $E_N(f;\eta)$ is uniformly bounded
on any compact subset $K\subset\C^2$ (in particular
if it converges to some function), then by the Stiltjes-Vitali-Montel
Theorem, there will be a subsequence of $f-E_N(f;\eta)$ that
will converge to a holomorphic function that will vanish on
an infinite number of lines. So this limit will be $0$. In fact,
any subsequence will have another subsequence that will also converge
to $0$. It will follow that the sequence
$f-E_N(f;\eta)$ will converge to $0$, ie
$E_N(f;\eta)$ will converge to $f$
(see \cite{alabut} and \cite{loganshepp} for examples of
positive results of reconstruction).

The problem is that we do not have any control
{\em a priori} of the function $E_N(f;\eta)$ and do not have any idea
if $E_N(f;\eta)$ is always uniformly bounded for any
$f\in\mathcal{O}\left(\C^2\right)$. In fact, this is false and
one can construct special subsets $\{\eta_j\}_{j\geq1}$
(see Section~\ref{ctrex}, Proposition~\ref{ctrexample})
for which $E_N(f;\eta)$ is not bounded for any $f$ (ie
there exists $f\in\mathcal{O}\left(\C^2\right)$ such that
$E_N(f;\eta)$ does not converge).
Conversely, if we choose for $f$ a polynomial function,
the formula $E_N(f;\eta)$ will converge for any subset
$\{\eta_j\}_{j\geq1}$. Finally, if we choose special subsets
for $\{\eta_j\}_{j\geq1}$
like real lines or circles of $\C$, the formula $E_N(f;\eta)$
will always converge to $f$ (Theorem~\ref{superbe}).
Therefore we want to reduce the problem of the convergence of
$E_N(f;\eta)$ to the one of classifying the {\em good} subsets
$\{\eta_j\}_{j\geq1}$ (ie the ones for which
$E_N(f;\eta)$ converges for every $f$) and the others
(the ones for which there is --at least-- a function
$f$ such that $E_N(f;\eta)$ does not converge).

On the other hand, since for any compact subset
$K\subset\C^2$,
\begin{eqnarray*}
\sup_{z\in K}
\left|
\sum_{k+l\geq N}a_{k,l}z_1^kz_2^l
\right|
& \xrightarrow[N\rightarrow\infty]{} &
0\,,
\end{eqnarray*}
the relation~(\ref{relation}) allows us
to reduce the problem of the convergence of $E_N(f;\eta)$
to the one of the formula $R_N(f;\eta)$. A preliminar
property about the Lagrange interpolation polynomials
(Section~\ref{preliminar}, Lemma~\ref{lagrangebis})
yields to
\begin{eqnarray}\nonumber
R_N(f;\eta)(z)
& = &
\sum_{p=0}^{N-1}
z_2^{N-1-p}
\prod_{j=1}^p(z_1-\eta_jz_2)
\;\times
\\\label{RNlagrange}
& &
\times\;
\Delta_{p,(\eta_p,\ldots,\eta_1)}
\left[\zeta\mapsto
\sum_{m\geq N}
\left(
\frac{z_2+\overline{\zeta}z_1}{1+|\zeta|^2}
\right)^{m-N+1}
\sum_{k+l=m}a_{k,l}\zeta^k
\right](\eta_{p+1})
\,,
\end{eqnarray}
where the function $\Delta_p$ is defined as follows:
\begin{eqnarray}\label{defDelta}
\Delta_{0,\emptyset}(g)(\eta_1)
& := &
g(\eta_1),\,
\\\nonumber
\Delta_{p+1,(\eta_{p+1},\eta_p\ldots,\eta_1)}
(g)(\eta_{p+2})
& := &
\dfrac{\Delta_{p,(\eta_{p},\ldots,\eta_1)}
(g)(\eta_{p+2})
-
\Delta_{p,(\eta_p,\ldots,\eta_1)}
(g)(\eta_{p+1})
}
{\eta_{p+2}-\eta_{p+1}}
\,.
\end{eqnarray}
$\Delta_p$ means the discrete derivative of $g$.
A control of these $\Delta_p$ will allow us to get
a control of $R_N(f;\eta)$. For example, the
$\Delta_p$ of a holomorphic function looks like
its usual iterative derivative
(Section~\ref{preliminar}, Lemma~\ref{Deltanal}).
Nevertheless, a non holomorphic function,
even $C^{\infty}$, can
have a $\Delta_p$ without any control
(Section~\ref{ctrex}, Lemma~\ref{constructctrex}).
It follows that, in the case of
$R_N(f;\eta)$, the non-validity of the convergence
for every $f$ is
linked to the fact that, in its above
expression~(\ref{RNlagrange}),
there is the $\Delta_p$ of a
non holomorphic function with respect to
$\zeta\in\C$ (although it is $C^{\infty}$, bounded
on $\C$ and uniformly converges to $0$ as
$N\rightarrow\infty$).
This should allow us to reduce the problem of the
control of $R_N(f;\eta)$ to the one
of the $\Delta_p$ of the function
$\overline{\zeta}$, that will be specified in the
following subsection.

\subsection{Results}

In this part we give the results of this paper. The first
gives an equivalent criterion about the convergence of
$E_N(f;\eta)$ for every $f$,
for the case when the subset
$\{\eta_j\}_{j\geq1}$ is bounded.

\begin{theorem}\label{equivbounded}

Let $\{\eta_j\}_{j\geq1}$ be bounded.
Then the interpolation formula
$E_N(f;\eta)$ converges to $f$
uniformly on any compact
$K\subset\C^2$ and for all
$f\in\mathcal{O}\left(\C^2\right)$,
if and only if
$\{\eta_j\}_{j\geq1}$ satisfies the following condition: 
$\exists\,R_{\eta}\geq1$ (that only depends on
$\{\eta_j\}_{j\geq1}$), 
$\forall\,p,q\geq0$,
\begin{eqnarray}\label{criter}
\left|
\Delta_{p,(\eta_p,\ldots,\eta_1)}
\left[
\left(
\frac{\overline{\zeta}}{1+|\zeta|^2}
\right)^q
\right]
\left(\eta_{p+1}\right)
\right|
& \leq &
R_{\eta}^{p+q}
\,.
\end{eqnarray}

\end{theorem}

First, we see that the convergence of $E_N(f;\eta)$ is linked
to the control of the $\Delta_p$ of a non holomorphic function.
Next, this is an analytic criterion that only depends on the
configuration of the points $\eta_j,\,j\geq1$, and does not involve
any function $f\in\mathcal{O}\left(\C^2\right)$.
\bigskip

We also give an extension of this result for the case when
$\{\eta_j\}_{j\geq1}$ can be not bounded but is not dense in $\C$. Let be
$\eta_{\infty}\in\C\setminus\overline{\{\eta_j\}_{j\geq1}}$ and
set, for all $j\geq1$,
\begin{eqnarray}\label{deftheta}
\theta_j & := &
\frac{1+\overline{\eta_{\infty}}\eta_j}
{\eta_j-\eta_{\infty}}
\,,
\end{eqnarray}
Then the set $\{\theta_j\}_{j\geq1}$ is bounded
and an application of Theorem~\ref{equivbounded}
with the $\theta_j$ yields to the following more general result.

\begin{theorem}\label{equivnodense}

Let assume that $\{\eta_j\}_{j\geq1}$ is not dense
in $\C$. Let be any
$\eta_{\infty}\notin\overline{\{\eta_j\}_{j\geq1}}$
and the associate set
$\{\theta_j\}_{j\geq1}$.
Then for all 
$f\in\mathcal{O}\left(\C^2\right)$
the interpolation formula
$E_N(f;\eta)$ converges to $f$
uniformly on any compact
$K\subset\C^2$,
if and only if
$\{\theta_j\}_{j\geq1}$ satisfies the condition~(\ref{criter})
of Theorem~\ref{equivbounded}, ie 
$\exists\,R_{\theta}\geq1$, 
$\forall\,p,q\geq0$,
\begin{eqnarray*}
\left|
\Delta_{p,(\theta_p,\ldots,\theta_1)}
\left[
\left(
\frac{\overline{\zeta}}{1+|\zeta|^2}
\right)^q
\right]
\left(\theta_{p+1}\right)
\right|
& \leq &
R_{\theta}^{p+q}
\,.
\end{eqnarray*}

\end{theorem}

\bigskip

Now we will give another condition for the set
$\{\eta_j\}_{j\geq1}$ to make converge any
$E_N(f;\eta)$, that is more natural to be expressed.
First, we give the following definition.

\begin{definition}\label{anal}

$\{\eta_j\}_{j\geq1}$ is
{\em real-analytically interpolated}
if, for all
$\zeta\in\overline{\{\eta_j\}_{j\geq1}}$,
there exist
a neighborhood $V$ of
$\zeta$ and
$g\in\mathcal{O}(V)$ such that,
$\forall\,\eta_j\in V$,
\begin{eqnarray*}
\overline{\eta_j}
& = &
g(\eta_j)
\,.
\end{eqnarray*}

\end{definition}

Although the conjugate function
$\overline{\zeta}$ is not holomorphic,
this condition means that the closure of
$\{\eta_j\}_{j\geq1}$ can be embedded in some subset on which
$\overline{\zeta}$ locally coincides with a certain holomorphic function.
It follows that any $\eta_j$ is locally in the zero set of some
real-analytic function: 
$(x,y)\in V'\subset\R^2\mapsto x-iy-g(x+iy)$.

Then we can give the following result that is a sufficient
condition for the set $\{\eta_j\}_{j\geq1}$
to make converge $E_N(f;\eta)$ for every function $f$.

\begin{theorem}\label{superbe}

If $\{\eta_j\}_{j\geq1}$ is
real-analytically interpolated,
then for all
$f\in\mathcal{O}\left(\C^2\right)$,
the interpolation formula
$E_N(f;\eta)$ converges to $f$
uniformly on any $K\subset\C^2$.

\end{theorem}

The condition
for $\{\eta_j\}_{j\geq1}$ to be real-analytically interpolated
is also linked to the uniform control of its
$\Delta_p\left(\overline{\zeta}\right)$
(Section~\ref{proofs}, Proposition~\ref{unifDelta}).
This is an analogous condition to~(\ref{criter}), but with
a uniform estimation on the subsequences
$(\eta_{j_k})_{k\geq1}$ and the same constant $R_{\eta}$
(Section~\ref{proofs}, Definition~\ref{defunifDelta}).

On the other hand, we
understand why the convergence of $E_N(f;\eta)$
is always true for subsets like lines. For example,
the subset $\R$ is the set of $z\in\C$ such that
$\overline{z}=z$. In the same way,
$i\R=\{\overline{z}=-z\}$, and more generally, any real line $D$
of $\C$ can be written as
$\{z\in\C,\;a\Re\,z+b\Im\,z+c=0\}$,
with $(a,b)\in\R^2\setminus\{(0,0)\}$,
then
\begin{eqnarray*}
D 
& = & 
\left\{
a\frac{z+\overline{z}}{2}+b\frac{z-\overline{z}}{2i}+c=0
\right\}
\;=\;
\left\{
\overline{z}=-\frac{\frac{a-ib}{2}\,z+c}{(a+ib)/2}
\right\}
\,.
\end{eqnarray*}
In the same way, any circle $C$ of $\C$ can be written as
$\{z\in\C,\;|z-z_0|=r\}$, with $z_0\in\C,\,r>0$, then
\begin{eqnarray*}
C & = &
\left\{
(z-z_0)\overline{(z-z_0)}=r^2
\right\}
\;=\;
\left\{
\overline{z}=\overline{z_0}+\frac{r^2}{z-z_0}
\right\}
\,.
\end{eqnarray*}

Another fact is that the real-analytical condition for
$\{\eta_j\}_{j\geq1}$ can also be reformulated as
a real-analytical dependence of the family
$\left(\{z_1-\eta_jz_2=0\}\right)_{j\geq1}$,
formulation that can be extended in the case of any
family of analytic subvarieties
$\{Z_j\}_{j\geq1}$ of any domain
$\Omega\subset\C^n$.

This also yields to another natural question: is this condition necessary?
We think that it should not be the case
since the formula $E_N(f;\eta)$ still converges
for a lots of subsets $\{\eta_j\}_{j\geq1}$ that are in
greater number than the ones that are
real-analytically interpolated.

Nevertheless, we will see that this condition
cannot be completely forgotten. Indeed, we will
construct in Section~\ref{ctrex}
(Proposition~\ref{ctrexample}) a bounded sequence
$(\eta_j)_{j\geq1}\subset\R\cup i\R$ 
that converges to $0$
without staying in $\R$ or $i\R$,
and does not satisfy the
condition~(\ref{criter}) of
Theorem~\ref{equivbounded} (then
$E_N(f;\eta)$ cannot converge for every
function $f$).
\bigskip

Another fact that we want to specify is that the
formula $E_N(f;\eta)$ is a canonical interpolation formula
in the meaning that it is essentially the most simple interpolation
formula that fixes the polynomials of degree at
most $N-1$ (Section~\ref{preliminar}, Lemma~\ref{lemmEN}), ie
\begin{eqnarray}
\forall\,P\in\C_{N-1}[z_1,z_2],\;
E_N(P;\eta)
& \equiv &
P
\,.
\end{eqnarray}
\bigskip

This problem of explicit reconstruction is also motivated by possible
applications in mathematical economics and medical imaging.
We want to reconstruct any given function
$F$ with compact support $K\subset\R^n$ from the knowledge of
its Radon transforms
$(RF)\left(\theta_j,s\right)$,
$\left(\theta_j,s\right)\in\mathbb{S}^{n-1}\times\R$,
on a finite number of directions
$j=1,\ldots,N$ (see~\cite{hensha1} and~\cite{hensha2}). 

\bigskip

I would like to thank G. Henkin for having introduced me this
interesting problem and J. Ortega-Cerda for all the
rewarding ideas and discussions about it.

\bigskip

\section{Some preliminar results}\label{preliminar}

\subsection{Equality of residues/principal values}

We remind from Introduction (\cite{irigoyen3}, Proposition~1) that,
for all 
$f\in\mathcal{O}\left(\overline{\mathbb{B}_2}\right)$,
one has
\begin{eqnarray}\label{relation1}\nonumber
f(z)
& = &
\lim_{\varepsilon\rightarrow0}
\frac{1}{(2i\pi)^2}
\int_{
\left|
\prod_{j=1}^N
(\zeta_1-\eta_j\zeta_2)
\right|
=\varepsilon}
\frac{
f(\zeta)
\det{\left(\overline{\zeta},P_N(\zeta,z)\right)}
\,\omega(\zeta)
}
{\prod_{j=1}^N(\zeta_1-\eta_j\zeta_2)
(1-<\overline{\zeta},z>)}
\\
& &
-\;
\lim_{\varepsilon\rightarrow0}
\frac{\prod_{j=1}^N(z_1-\eta_j z_2)}{(2i\pi)^2}
\int_{
\left|
\prod_{j=1}^N
(\zeta_1-\eta_j\zeta_2)
\right|
>\varepsilon}
\frac{f(\zeta)\,\omega'(\overline{\zeta})\wedge\omega^(\zeta)}
{\prod_{j=1}^N(z_1-\eta_j z_2)
(1-<\overline{\zeta},z>)^2}
\,,
\end{eqnarray}
the integration being on the unit sphere
$\mathbb{S}_2=\{\zeta\in\C^2,\;
|\zeta_1|^2+|\zeta_2|^2=1\}$,
with
\begin{eqnarray*}
\begin{cases}
\omega'(\zeta)=
\zeta_1d\zeta_2-\zeta_2d\zeta_1\,,
\\
\omega(\zeta)=d\zeta_1\wedge d\zeta_2\,,
\end{cases}
\end{eqnarray*}
and
$P_N(\zeta,z)\in
\left(
\mathcal{O}\left(\C^2\right)
\right)^2$
is such that, for all
$(\zeta,z)\in\C^2\times\C^2$, 
\begin{eqnarray*}
<P_N(\zeta,z),\zeta-z>
& = &
\prod_{j=1}^N
(\zeta_1-\eta_j\zeta_2)
-
\prod_{j=1}^N
(z_1-\eta_j z_2)
\end{eqnarray*}
(see \cite{irigoyen3}, Lemma 7 for an explicit expression
of $P_N$).

We also have the analogous following relation
(\cite{irigoyen3}, Proposition 2)
\begin{eqnarray}\label{relation2}
f(z)
& = &
\lim_{\varepsilon\rightarrow0}
\frac{1}{(2i\pi)^2}
\int_{
\partial\Sigma_{\varepsilon}
}
\frac{
f(\zeta)
\det{\left(\overline{\zeta},P_N(\zeta,z)\right)}
\,\omega(\zeta)
}
{\prod_{j=1}^N(\zeta_1-\eta_j\zeta_2)
(1-<\overline{\zeta},z>)}
\\\nonumber
& &
-\;
\lim_{\varepsilon\rightarrow0}
\frac{\prod_{j=1}^N(z_1-\eta_j z_2)}{(2i\pi)^2}
\int_{\Sigma_{\varepsilon}}
\frac{f(\zeta)\,\omega'(\overline{\zeta})\wedge\omega(\zeta)}
{\prod_{j=1}^N(z_1-\eta_j z_2)
(1-<\overline{\zeta},z>)^2}
\,,
\end{eqnarray}
where
\begin{eqnarray*}
\Sigma_{\varepsilon}
& := &
\bigcup_{j=0}^{N}
\left\{\zeta\in\mathbb{S}_2,\;
\alpha_j+\varepsilon<|\zeta_1|<\alpha_{j+1}-\varepsilon
\right\}
\end{eqnarray*}
and for all
$p=1,\ldots,N$
\begin{eqnarray}\label{defalpha}
\alpha_p & := &
\frac{|\eta_p|}{\sqrt{1+|\eta_p|^2}}
\,,
\end{eqnarray}
with
$\alpha_1\leq\alpha_2\leq\cdots\leq\alpha_N$ (that one can
assume without loss of generality) and 
the convention that
$\alpha_0=0$ and $\alpha_{N+1}=1$.
This allows us to explicit both above integrals
and get the following relation: for all
$f\in\mathcal{O}\left(\overline{\mathbb{B}_2}\right)$ and
$z\in\mathbb{B}_2$,
\begin{eqnarray}\label{relationexp}
f(z)
& = &
E_N(f;\eta)(z)
-R_N(f;\eta)(z)
+\sum_{k+l\geq N}
a_{k,l}z_1^kz_2^l
\,,
\end{eqnarray}
where
\begin{eqnarray}\label{dse}
f(z) & = &
\sum_{k,l\geq0}
a_{k,l}z_1^kz_2^l
\,,
\end{eqnarray}
is the Taylor expansion of $f$
(that absolutely converges on
$\mathbb{B}_2$),
\begin{eqnarray}\label{EN}\nonumber
E_N(f;\eta)(z)
& = &
\sum_{p=1}^N
\prod_{j=p+1}^N
(z_1-\eta_j z_2)
\sum_{q=p}^N
\frac{1+\eta_p\overline{\eta_q}}{1+|\eta_q|^2}
\frac{1}
{\prod_{j=p,j\neq q}^N(\eta_q-\eta_j)}
\times
\\
& &
\;\;\;\;\;\;\;\;\;\;\;\;\;\;\;\;
\times
\sum_{m\geq N-p}
\left(
\frac{z_2+\overline{\eta_q}z_1}{1+|\eta_q|^2}
\right)^{m-N+p}
\frac{1}{m!}
\frac{\partial^m}{\partial v^m}|_{v=0}
[f(\eta_qv,v)]
\\\nonumber
& = &
\sum_{p=1}^N
\prod_{j=p+1}^N
(z_1-\eta_j z_2)
\sum_{q=p}^N
\frac{1+\eta_p\overline{\eta_q}}{1+|\eta_q|^2}
\frac{1}
{\prod_{j=p,j\neq q}^N(\eta_q-\eta_j)}
\times
\\\nonumber
& &
\;\;\;\;\;\;\;\;\;\;\;\;\;\;\;\;\;\;\;\;\;\;\;\;\;\;\;\;\;\;\;\;
\times
\sum_{k+l\geq N-p}
a_{k,l}\eta_p^k
\left(
\frac{z_2+\overline{\eta_q}z_1}{1+|\eta_q|^2}
\right)^{k+l-N+p}
\,,
\end{eqnarray}
and
\begin{eqnarray}\label{defRN1}\nonumber
R_N(f;\eta)(z)
& = &
\sum_{p=1}^N
\prod_{j=1,j\neq p}^N\frac{z_1-\eta_jz_2}{\eta_p-\eta_j}
\sum_{m\geq N}
\left(
\frac{z_2+\overline{\eta_p}z_1}{1+|\eta_p|^2}
\right)^{m-N+1}
\times
\\
& &
\;\;\;\;\;\;\;\;\;\;\;\;\;\;\;\;\;\;\;\;\;\;\;\;\;\;\;\;\;\;\;\;\;\;\;
\;\;\;\;\;\;\;\;\;\;\;\;\;\;\
\times
\;
\frac{1}{m!}
\frac{\partial^m}{\partial v^m}|_{v=0}
[f(\eta_pv,v)]
\\\nonumber
\\\nonumber
& = &
\sum_{p=1}^N
\prod_{j=1,j\neq p}^N\frac{z_1-\eta_jz_2}{\eta_p-\eta_j}
\;
\sum_{k+l\geq N}a_{k,l}\eta_p^k
\left(
\frac{z_2+\overline{\eta_p}z_1}{1+|\eta_p|^2}
\right)^{k+l-N+1}
\,.
\end{eqnarray}

This equality holds if
$f\in\mathcal{O}\left(\C^2\right)$. Since
by Cauchy-Schwarz
\begin{eqnarray*}
\left|
\frac{z_2+\overline{\eta_p}z_1}{1+|\eta_p|^2}
\right|
& \leq &
\frac{
\sqrt{|z_1|^2+|z_2|^2}\sqrt{1+|\eta_p|^2}
}{1+|\eta_p|^2}
\;\leq\;
\|z\|
\,,
\end{eqnarray*}
for all $z\in\C^2$ and $p=1,\ldots,N$,
$E_N(f;\eta)$ and $R_N(f;\eta)$ are still
absolutely convergent series then are still
well-defined
for $z\in\C^2$. Then by analytic extension,
(\ref{relationexp}) holds for all
$z\in\C^2$.

Now we will prove the following preliminar result.

\begin{lemma}\label{prelimformula}

For all $N\geq1$ and $z\in\mathbb{B}_2$,
\begin{eqnarray*}
\lim_{\varepsilon\rightarrow0}
\frac{1}{(2i\pi)^2}
\int_{
\left|
\prod_{j=1}^N
(\zeta_1-\eta_j\zeta_2)
\right|
=\varepsilon}
\frac{
f(\zeta)
\det{\left(\overline{\zeta},P_N(\zeta,z)\right)}
\,\omega(\zeta)
}
{\prod_{j=1}^N(\zeta_1-\eta_j\zeta_2)
(1-<\overline{\zeta},z>)}
& = &
\;\;\;\;\;\;\;\;\;\;\;\;\;\;\;\;\;\;\;\;
\end{eqnarray*}
\begin{eqnarray*}
\;\;\;\;\;\;\;\;\;\;\;\;\;\;\;\;\;\;\;\;\;\;\;\;\;\;\;
& = &
\lim_{\varepsilon\rightarrow0}
\frac{1}{(2i\pi)^2}
\int_{
\partial\Sigma_{\varepsilon}
}
\frac{
f(\zeta)
\det{\left(\overline{\zeta},P_N(\zeta,z)\right)}
\,\omega(\zeta)
}
{\prod_{j=1}^N(\zeta_1-\eta_j\zeta_2)
(1-<\overline{\zeta},z>)}
\end{eqnarray*}
and
\begin{eqnarray*}
\lim_{\varepsilon\rightarrow0}
\frac{\prod_{j=1}^N(z_1-\eta_j z_2)}{(2i\pi)^2}
\int_{
\left|
\prod_{j=1}^N
(\zeta_1-\eta_j\zeta_2)
\right|
>\varepsilon}
\frac{f(\zeta)\,\omega'(\overline{\zeta})\wedge\omega^(\zeta)}
{\prod_{j=1}^N(z_1-\eta_j z_2)
(1-<\overline{\zeta},z>)^2}
& = &
\end{eqnarray*}
\begin{eqnarray*}
\;\;\;\;\;\;\;\;\;\;\;\;\;\;\;\;\;\;\;\;\;
& = &
\lim_{\varepsilon\rightarrow0}
\frac{\prod_{j=1}^N(z_1-\eta_j z_2)}{(2i\pi)^2}
\int_{\Sigma_{\varepsilon}}
\frac{f(\zeta)\,\omega'(\overline{\zeta})\wedge\omega(\zeta)}
{\prod_{j=1}^N(z_1-\eta_j z_2)
(1-<\overline{\zeta},z>)^2}
\,.
\end{eqnarray*}

\end{lemma}

Before giving the proof of the lemma, we begin with
this first result.

\begin{lemma}

We define for all $p=1,\ldots,N$
\begin{eqnarray}
C_p & = &
\prod_{j=1,j\neq p}^N
\frac{|\eta_j-\eta_p|}{\sqrt{1+|\eta_j|^2}}
\,.
\end{eqnarray}
Then
\begin{eqnarray*}
\lim_{\varepsilon\rightarrow0}
\frac{\prod_{j=1}^N(z_1-\eta_j z_2)}{(2i\pi)^2}
\int_{
\left|
\prod_{j=1}^N
(\zeta_1-\eta_j\zeta_2)
\right|
>\varepsilon}
\frac{f(\zeta)\,\omega'(\overline{\zeta})\wedge\omega^(\zeta)}
{\prod_{j=1}^N(z_1-\eta_j z_2)
(1-<\overline{\zeta},z>)^2}
& = &
\end{eqnarray*}
\begin{eqnarray*}
& = &
\lim_{\varepsilon\rightarrow0}
\frac{\prod_{j=1}^N(z_1-\eta_j z_2)}{(2i\pi)^2}
\int_{
\bigcap_{j=1}^N
\{|\zeta_1-\eta_j\zeta_2|
>2\varepsilon/C_j\}
}
\frac{f(\zeta)\,\omega'(\overline{\zeta})\wedge\omega^(\zeta)}
{\prod_{j=1}^N(z_1-\eta_j z_2)
(1-<\overline{\zeta},z>)^2}
\,.
\end{eqnarray*}

\end{lemma}

\begin{proof}

First, $z\in\mathbb{B}_2$ being fixed,
one has to integrate the following $3$-differential form
\begin{eqnarray*}
\frac{\varphi(\zeta)}
{\prod_{j=1}^N(\zeta_1-\eta_j\zeta_2)}
\,,
\end{eqnarray*}
with $\varphi$ a $3$-form that is
$C^{\infty}$ on $\mathbb{S}_2$
(since
$|<\overline{\zeta},z>|\leq\|\zeta\|\,\|z\|<1$).

On the other hand, one has
\begin{eqnarray}\label{prelim0}
\mathbb{S}_2\cap\{\zeta_1=\eta_j\zeta_2\}
& = &
\{\zeta_1=\eta_j\zeta_2,\;
|\zeta_1|=\alpha_p\}
\,.
\end{eqnarray}
For all $\varepsilon>0$ small enough, if
$\prod_{j=1}^N|\zeta_1-\eta_j\zeta_2|\leq\varepsilon$,
then $\exists\,p$,
$|\zeta_1-\eta_p\zeta_2|\leq\varepsilon^{1/N}$ and
\begin{eqnarray*}
\varepsilon\;\geq\;
\prod_{j=1}^N
|\zeta_1-\eta_j\zeta_2|
& \sim &
|\zeta_1-\eta_{p}\zeta_2|
\prod_{j\neq p}
|\eta_{p}-\eta_j||\zeta_2|
\;\sim\;
|\zeta_1-\eta_{p}\zeta_2|
\prod_{j\neq p}
\frac{|\eta_{p}-\eta_j|}
{\sqrt{1+|\eta_{p}|^2}}
\\
& \sim &
C_{p}|\zeta_1-\eta_{p}\zeta_2|
\,,
\end{eqnarray*}
thus
\begin{eqnarray*}
A_{\varepsilon}\;:=\;
\left\{
\prod_{j=1}^N
|\zeta_1-\eta_j\zeta_2|\leq\varepsilon
\right\}
& \subset &
\bigcup_{j=1}^N
\left\{
|\zeta_1-\eta_j\zeta_2|\leq2\varepsilon/C_j
\right\}
\;=:\;
B_{\varepsilon}
\end{eqnarray*}
(the union being disjoint for $\varepsilon$ small enough since
the $\eta_j$ are all differents).
Since
\begin{eqnarray*}
\int_{\mathbb{S}_2\setminus A_{\varepsilon}}
\frac{\varphi(\zeta)}{\prod_{j=1}^N(\zeta_1-\eta_j\zeta_2)}
& = &
\int_{\mathbb{S}_2\setminus B_{\varepsilon}}
\frac{\varphi(\zeta)}{\prod_{j=1}^N(\zeta_1-\eta_j\zeta_2)}
+
\int_{B_{\varepsilon}\setminus A_{\varepsilon}}
\frac{\varphi(\zeta)}{\prod_{j=1}^N(\zeta_1-\eta_j\zeta_2)}
\,,
\end{eqnarray*}
in order to prove the lemma it is sufficient to prove that
\begin{eqnarray}\label{prelim1}
\int_{B_{\varepsilon}\setminus A_{\varepsilon}}
\frac{\varphi(\zeta)}{\prod_{j=1}^N(\zeta_1-\eta_j\zeta_2)}
& \xrightarrow[\varepsilon\rightarrow0]{} &
0\,.
\end{eqnarray}
One has
\begin{eqnarray*}
\left|
\int_{B_{\varepsilon}\setminus A_{\varepsilon}}
\frac{\varphi(\zeta)}{\prod_{j=1}^N(\zeta_1-\eta_j\zeta_2)}
\right|
& \leq &
\sum_{j=1}^N
\int_{
\{|\zeta_1-\eta_j\zeta_2|\leq2\varepsilon/C_j,\,
\prod_{i=1}^N|\zeta_1-\eta_i\zeta_2|\geq\varepsilon\}
}
\left|
\frac{\varphi(\zeta)}{\prod_{j=1}^N(\zeta_1-\eta_j\zeta_2)}
\right|
\\
& \leq &
\frac{1}{\varepsilon}
\sum_{j=1}^N
\int_{
\{|\zeta_1-\eta_j\zeta_2|\leq2\varepsilon/C_j,\,
\prod_{i=1}^N|\zeta_1-\eta_i\zeta_2|\geq\varepsilon\}
}
|\varphi(\zeta)|
\\
& \leq &
\frac{1}{\varepsilon}
\sum_{j=1}^N
\int_{
\{|\zeta_1-\eta_j\zeta_2|\leq2\varepsilon/C_j\}}
|\varphi(\zeta)|
\\
& = &
\frac{1}{\varepsilon}
\sum_{j=1}^N
\int_{
\{|\zeta_1-\eta_j\zeta_2|\leq2\varepsilon/C_j\},
||\zeta_1|-\alpha_j|\leq b_j\varepsilon\}}
|\varphi(\zeta)|
\,,
\end{eqnarray*}
the last equality coming 
from~(\ref{prelim0}).
Since $\varphi$ is bounded on
$\mathbb{S}_2$, for $j=1,\ldots,N$ one has
\begin{eqnarray*}
\int_{
\{|\zeta_1-\eta_j\zeta_2|\leq2\varepsilon/C_j\},
||\zeta_1|-\alpha_j|\leq b_j\varepsilon\}}
|\varphi(\zeta)|
& \leq &
M
\int_{\alpha_j-b_j\varepsilon}^{\alpha_j+b_j\varepsilon}
2rdr
\int_{|\zeta_2|=\sqrt{1-r^2}}
d\theta_2
\int_{|\zeta_1|=r,|\zeta_1-\eta_j\zeta_2|\leq2\varepsilon/C_j}
d\theta_1
\,.
\end{eqnarray*}
For all $\varepsilon$ small enough and $\zeta_2$ fixed,
\begin{eqnarray*}
\theta_1=Arg(\zeta_1)
& = &
Arg(\eta_j\zeta_2)
+Arg\left(\frac{\zeta_1}{\eta_j\zeta_2}\right)
\;=\;
Arg(\eta_j\zeta_2)
+
Arg\left(1+\frac{\zeta_1-\eta_j\zeta_2}{\eta_j\zeta_2}\right)
\\
& = &
Arg(\eta_j\zeta_2)
+
\frac{1}{i}
\Im\,Log\left(1+\frac{O(\varepsilon)}{\eta_j\zeta_2}\right)
\;=\;
Arg(\eta_j\zeta_2)
+O(\varepsilon)
\,,
\end{eqnarray*}
then
\begin{eqnarray*}
\int_{\alpha_j-b_j\varepsilon}^{\alpha_j+b_j\varepsilon}
2rdr
\int_{|\zeta_2|=\sqrt{1-r^2}}
d\theta_2
\int_{|\zeta_1-\eta_j\zeta_2|\leq2\varepsilon/C_j,|\zeta_1|=r}
d\theta_1
& = &
\int_{\alpha_j-b_j\varepsilon}^{\alpha_j+b_j\varepsilon}
2rdr
\int_{|\zeta_2|=\sqrt{1-r^2}}
d\theta_2
O(\varepsilon)
\\
& = &
O\left(\varepsilon^2\right)
\,.
\end{eqnarray*}
Finally
\begin{eqnarray*}
\left|
\int_{B_{\varepsilon}\setminus A_{\varepsilon}}
\frac{\varphi(\zeta)}{\prod_{j=1}^N(\zeta_1-\eta_j\zeta_2)}
\right|
& = &
\frac{1}{\varepsilon}
N\,O\left(\varepsilon^2\right)
\;=\;
O(\varepsilon)
\;\xrightarrow[\varepsilon\rightarrow0]{}0\,.
\end{eqnarray*}

\end{proof}

\bigskip

Now we can prove Lemma~\ref{prelimformula}.

\begin{proof}

By relations~(\ref{relation1}) and~(\ref{relation2}),
in order to prove the lemma
it is sufficient to prove the second equality.
By the previous lemma it is sufficient to prove that
\begin{eqnarray*}
\lim_{\varepsilon\rightarrow0}
\frac{\prod_{j=1}^N(z_1-\eta_j z_2)}{(2i\pi)^2}
\int_{
\bigcap_{j=1}^N
\{|\zeta_1-\eta_j\zeta_2|
>2\varepsilon/C_j\}
}
\frac{f(\zeta)\,\omega'(\overline{\zeta})\wedge\omega^(\zeta)}
{\prod_{j=1}^N(z_1-\eta_j z_2)
(1-<\overline{\zeta},z>)^2}
& = &
\end{eqnarray*}
\begin{eqnarray}\label{prelimpv}
\;\;\;\;\;\;\;\;\;\;\;\;\;\;\;\;\;\;\;\;\;
& = &
\lim_{\varepsilon\rightarrow0}
\frac{\prod_{j=1}^N(z_1-\eta_j z_2)}{(2i\pi)^2}
\int_{\Sigma_{\varepsilon}}
\frac{f(\zeta)\,\omega'(\overline{\zeta})\wedge\omega(\zeta)}
{\prod_{j=1}^N(z_1-\eta_j z_2)
(1-<\overline{\zeta},z>)^2}
\,.
\end{eqnarray}

If $\exists\,\eta_{j_0}=0$ then
$\{|\zeta_1-\eta_{j_0}\zeta_2|>2\varepsilon/C_{j_0}\}
=\{|\zeta_1|>2\varepsilon/C_{j_0}\}$. Since
$1/\zeta_1$ is defined in the neighborhood
of any
$\{||\zeta_1|-\alpha_j|\leq b_j\varepsilon\}$,
one just has to consider the case of $\eta_j\neq0,\,j\neq j_0$
with $\varphi(\zeta)/\zeta_1$. So one can assume that
$\eta_j\neq0,\,\forall\,j=1,\ldots,N$ (then
$\alpha_j>0$).
Without loss of generality, one has
\begin{eqnarray*}
B_{\varepsilon}
& \subset &
\bigcup_{j=1}^N
\{||\zeta_1|-\alpha_j|\leq b_j\varepsilon\}
\;=\;
\bigsqcup_{k=1}^K
\{||\zeta_1|-\alpha_{j_k}|\leq b_{j_k}\varepsilon\}
\;=:\;
C_{\varepsilon}
\,,
\end{eqnarray*}
with
$\alpha_{1}=\cdots=\alpha_{j_1}
\neq\alpha_{j_1+1}=\cdots=\alpha_{j_2}\neq
\cdots\neq\alpha_{j_{K-1}+1}=\cdots=\alpha_{j_K}$.
Then for all $\varepsilon$ small enough,
\begin{eqnarray*}
C_{\varepsilon}\setminus B_{\varepsilon}
& = &
\bigsqcup_{k=1}^K
\left(
\left\{\left||\zeta_1|-\alpha_{j_k}\right|\leq c_{j_k}\varepsilon\right\}
\setminus B_{\varepsilon}
\right)
\\
& = &
\bigsqcup_{k=1}^K
\;
\bigcap_{\alpha_j=\alpha_{j_k}}
\left(
\left\{\left|\,|\zeta_1|-\alpha_{j_k}\right|\leq c_{j_k}\varepsilon\right\}
\cap
\{|\zeta_1-\eta_j\zeta_2|\geq2\varepsilon/C_j\}
\right)
\,,
\end{eqnarray*}
thus
\begin{eqnarray*}
\int_{C_{\varepsilon}\setminus B_{\varepsilon}}
\frac{\varphi(\zeta)}{\prod_{j=1}^N(\zeta_1-\eta_j\zeta_2)}
& = &
\sum_{k=1}^K
\int_{\{||\zeta_1|-\alpha_{j_k}|\leq c_{j_k}\varepsilon\}
\cap
\{|\zeta_1-\eta_j\zeta_2|\geq2\varepsilon/C_j\}}
\frac{\varphi(\zeta)}{\prod_{j=1}^N(\zeta_1-\eta_j\zeta_2)}
\,.
\end{eqnarray*}
Since for all $j=1,\ldots,N$
\begin{eqnarray*}
\frac{\varphi(\zeta)}{\prod_{j=1}^N(\zeta_1-\eta_j\zeta_2)}
& = &
\frac{\widetilde{\varphi_j}(\zeta)}{\zeta_1-\eta_j\zeta_2}
\\
& = &
\frac{\widetilde{\varphi_j}(\eta_j\zeta_2,\zeta_2)}{\zeta_1-\eta_j\zeta_2}
+
\frac{\widetilde{\varphi_j}(\zeta)-\widetilde{\varphi_j}(\eta_j\zeta_2,\zeta_2)}
{\zeta_1-\eta_j\zeta_2}
\\
& = &
\frac{\psi_{j,1}(\zeta_2)}{\zeta_1-\eta_j\zeta_2}
+
\psi_{j,2}(\zeta)
\,,
\end{eqnarray*}
with $\psi_{j,1},\,\psi_{j,2}$ locally
integrables and bounded in the neighborhood of
$\{\zeta_1-\eta_j\zeta_2=0\}$.
On the other hand, for all $k=1,\ldots,K$
\begin{eqnarray*}
\{||\zeta_1|-\alpha_{j_k}|\leq c_{j_k}\varepsilon\}
& = &
\bigsqcup_{j=j_{k-1}+1}^{j_k}
\left(U_j\cap
\{||\zeta_1|-\alpha_{j_k}|\leq c_{j_k}\varepsilon\}
\right)
\,,
\end{eqnarray*}
where 
$U_{j_{k-1}+1},\ldots,U_{j_k}$ is a partition such that
$\varphi_j,\,\psi_j$ are bounded
in $U_j$
(for example, 
$U_{j}=\{|\zeta_1-\eta_j\zeta_2|\leq\varepsilon_0\}$,
for all $j=j_{k-1}+1,\ldots,j_k$, with $\varepsilon_0$
small enough so that
$U_j\cap U_i\cap\mathbb{S}_2=\emptyset$
if $j\neq i$; finally one replaces
$U_{j_{k-1}+1}$ with
$U_{j_{k-1}+1}\bigcup
\left(
\mathbb{S}_2\setminus\bigcup_{j=j_{k-1}+2}^{j_k}U_j
\right)$ ).
Then
\begin{eqnarray*}
\int_{C_{\varepsilon}\setminus B_{\varepsilon}}
\frac{\varphi(\zeta)}{\prod_{j=1}^N(\zeta_1-\eta_j\zeta_2)}
& = &
\;\;\;\;\;\;\;\;\;\;\;\;\;\;\;\;\;\;\;\;\;\;\;\;\;\;\;\;\;\;\;
\;\;\;\;\;\;\;\;\;\;\;\;\;\;\;\;\;\;\;\;\;\;\;\;\;\;\;\;\;\;\;
\;\;\;\;\;\;\;\;\;\;\;\;\;\;\;\;\;\;\;\;
\end{eqnarray*}
\begin{eqnarray*}
\;\;\;\;\;\;\;\;\;\;
& = &
\sum_{k=1}^K
\sum_{j=j_{k-1}+1}^{j_k}
\int_{
U_j\cap\{||\zeta_1|-\alpha_{j_k}|\leq c_{j_k}\varepsilon\}
\cap
\{|\zeta_1-\eta_j\zeta_2|\geq2\varepsilon/C_j\}}
\left(
\frac{\psi_{j,1}(\zeta_2)}{\zeta_1-\eta_j\zeta_2}
+
\psi_{j,2}(\zeta)
\right)
\,.
\end{eqnarray*}
Since
\begin{eqnarray*}
\left|
\int_{U_j\cap
\{||\zeta_1|-\alpha_{j_k}|\leq c_{j_k}\varepsilon\}
\cap
\{|\zeta_1-\eta_j\zeta_2|\geq2\varepsilon/C_j\}}
\psi_{j,2}(\zeta)
\right|
\;\leq\;
M
\int_{\mathbb{S}_2\cap
\{||\zeta_1|-\alpha_{j_k}|\leq c_{j_k}\varepsilon\}
}
d\lambda
\xrightarrow[\varepsilon\rightarrow0]{}0\,,
\end{eqnarray*}
one has
\begin{eqnarray*}
\lim_{\varepsilon\rightarrow0}
\int_{C_{\varepsilon}\setminus B_{\varepsilon}}
\frac{\varphi(\zeta)}{\prod_{j=1}^N(\zeta_1-\eta_j\zeta_2)}
& = &
\;\;\;\;\;\;\;\;\;\;\;\;\;\;\;\;\;\;\;\;\;\;\;\;\;\;\;\;\;
\;\;\;\;\;\;\;\;\;\;\;\;\;\;\;\;\;\;\;\;\;\;\;\;\;\;\;\;\;
\end{eqnarray*}
\begin{eqnarray*}
& = &
\lim_{\varepsilon\rightarrow0}
\sum_{k=1}^K
\sum_{j=j_{k-1}+1}^{j_k}
\int_{
U_j\cap\{||\zeta_1|-\alpha_{j_k}|\leq c_{j_k}\varepsilon\}
\cap
\{|\zeta_1-\eta_j\zeta_2|\geq2\varepsilon/C_j\}}
\frac{\psi_{j,1}(\zeta_2)}{\zeta_1-\eta_j\zeta_2}
\\
& = &
\lim_{\varepsilon\rightarrow0}
\sum_{k=1}^K
\sum_{j=j_{k-1}+1}^{j_k}
\int_{\alpha_{j_k}-c_{j_k}\varepsilon}^{\alpha_{j_k}+c_{j_k}\varepsilon}
(-2rdr)
\int_{|\zeta_2|=\sqrt{1-r^2}}
\widetilde{\psi_j}(r,\zeta_2)
\frac{d\zeta_2}{\zeta_2}
\times
\\
& &
\;\;\;\;\;\;\;\;\;\;\;\;\;\;\;\;\;\;\;\;\;\;\;\;\;\;\;\;\;\;\;\;
\times
\int_{|\zeta_1|=r,\zeta\in U_j,|\zeta_1-\eta_j\zeta_2|\geq2\varepsilon/C_j}
\frac{d\zeta_1}{\zeta_1(\zeta_1-\eta_j\zeta_2)}
\,,
\end{eqnarray*}
because
\begin{eqnarray*}
\omega'(\overline{\zeta})\wedge\omega(\zeta)
& = &
\left(
\frac{r^2}{\zeta_1}
d\left(
\frac{1-r^2}{\zeta_2}
\right)
-\frac{1-r^2}{\zeta_2}
d\left(
\frac{r^2}{\zeta_1}
\right)
\right)
\wedge
d\zeta_1\wedge
d\zeta_2
\\
& = &
-2rdr\wedge
\frac{d\zeta_1}{\zeta_1}\wedge
\frac{d\zeta_2}{\zeta_2}
\,.
\end{eqnarray*}
Now for all
$r\neq\alpha_j$ and
$|\zeta_2|=\sqrt{1-r^2}$
($\neq\alpha_j/|\eta_j|$), one has
(since $\forall\,\varepsilon$ small enough,
$\{|\zeta_1-\eta_j\zeta_2|\leq2\varepsilon/C_j\}
\subset U_j$)
\begin{eqnarray*}
\int_{|\zeta_1|=r,\zeta\in U_j,|\zeta_1-\eta_j\zeta_2|\geq2\varepsilon/C_j}
\frac{d\zeta_1}{\zeta_1(\zeta_1-\eta_j\zeta_2)}
& = &
\;\;\;\;\;\;\;\;\;\;\;\;\;\;\;\;\;\;\;\;\;\;\;\;\;\;\;\;\;\;\;\;\;\;\;
\;\;\;\;\;\;\;\;\;\;\;\;
\end{eqnarray*}
\begin{eqnarray*}
& = &
\left[
\int_{|\zeta_1|=r}
-\int_{|\zeta_1|=r,\zeta\notin U_j}
-\int_{|\zeta_1|=r,|\zeta_1-\eta_j\zeta_2|\leq2\varepsilon/C_j}
\right]
\frac{d\zeta_1}{\zeta_1(\zeta_1-\eta_j\zeta_2)}
\,.
\end{eqnarray*}

First, since
$|\zeta_1|>|\eta_j\zeta_2|$ if and only if
$r>\alpha_j$, one has
\begin{eqnarray*}
\left|
\int_{|\zeta_1|=r}
\frac{d\zeta_1}{\zeta_1(\zeta_1-\eta_j\zeta_2)}
\right|
& = &
\left|
(2i\pi)
\left(
-\frac{1}{\eta_j\zeta_2}
+{\bf 1}_{r>\alpha_j}
\frac{1}{\eta_j\zeta_2}
\right)
\right|
\\
& = &
\frac{2\pi}{|\eta_j\zeta_2|}
\left(1-{\bf 1}_{r>\alpha_j}\right)
\;\leq\;
\frac{2\pi}{|\eta_j|\sqrt{1-r^2}}
\,.
\end{eqnarray*}

Next,
$\forall\,j=j_{k-1}+1,\ldots,j_k$,
$U_j\supset\{|\zeta_1-\eta_j\zeta_2|\leq\varepsilon_0\}$,
then
\begin{eqnarray*}
\left|
\int_{|\zeta_1|=r,\zeta\notin U_j}
\frac{d\zeta_1}{\zeta_1(\zeta_1-\eta_j\zeta_2)}
\right|
& \leq &
\int_{|\zeta_1|=r,\zeta\notin U_j}
\frac{|d\zeta_1|}{r\varepsilon_0}
\;\leq\;
\int_{|\zeta_1|=r}
\frac{|d\zeta_1|}{r\varepsilon_0}
\;=\;
\frac{2\pi}{\varepsilon_0}
\,.
\end{eqnarray*}

Thus
\begin{eqnarray*}
\left|
\sum_{k=1}^K
\sum_{j=j_{k-1}+1}^{j_k}
\int_{\alpha_{j_k}-c_{j_k}\varepsilon}^{\alpha_{j_k}+c_{j_k}\varepsilon}
2rdr
\int_{|\zeta_2|=\sqrt{1-r^2}}
\widetilde{\psi_j}(r,\zeta_2)
d\zeta_2
\left[
\int_{|\zeta_1|=r}
-\int_{|\zeta_1|=r,\zeta\notin U_j}
\right]
\frac{d\zeta_1}{\zeta_1(\zeta_1-\eta_j\zeta_2)}
\right|
& \leq &
\end{eqnarray*}
\begin{eqnarray*}
\;\leq\;
\sum_{j=1}^N
\int_{\alpha_{j}-c_{j}\varepsilon}^{\alpha_{j}+c_{j}\varepsilon}
2rdr
\,
2\pi\sqrt{1-r^2}\,
\max_{j_{k-1}+1\leq j\leq j_{k}}\|\widetilde{\psi_j}\|_{\infty,U_j}
\,
2\pi
\left(
\frac{1}{|\eta_j|\sqrt{1-r^2}}
+\frac{1}{\varepsilon_0}
\right)
& \xrightarrow[\varepsilon\rightarrow0]{}
0\,.
\end{eqnarray*}

Finally, we consider the following integral
\begin{eqnarray*}
\int_{|\zeta_1|=r,|\zeta_1-\eta_j\zeta_2|\leq2\varepsilon/C_j}
\frac{d\zeta_1}{\zeta_1(\zeta_1-\eta_j\zeta_2)}
\end{eqnarray*}
(in the case where the measure of
$\{|\zeta_1|=r,|\zeta_1-\eta_j\zeta_2|\leq2\varepsilon/C_j\}$
is not zero). We set
$\theta_j:=Arg(\eta_j\zeta_2)$, $\zeta_1=re^{i\theta}$ so
\begin{eqnarray*}
\frac{2\varepsilon}{C_j}
\;\geq\;
|\zeta_1-\eta_j\zeta_2|
& = &
\left|
re^{i\theta}
-|\eta_j\zeta_2|e^{i\theta_j}
\right|
\;=\;
\left|
re^{i(\theta-\theta_j)}
-|\eta_j\zeta_2|
\right|
\\
& \geq &
\left|
\Im\left(
re^{i(\theta-\theta_j)}
-|\eta_j\zeta_2|
\right)
\right|
\;=\;
\left|
r\sin(\theta-\theta_j)
\right|
\,.
\end{eqnarray*}
Since for all $\varepsilon$ small enough,
$r\sim\alpha_j>0$, one has
$|\sin(\theta-\theta_j)|\leq O(\varepsilon)$,
then
$|\theta-\theta_j|\leq O(\varepsilon)$, ie
$\theta\in[\theta_j-\theta_{\varepsilon},\theta_j+\theta'_{\varepsilon}]$
with
$\theta_{\varepsilon},\,\theta'_{\varepsilon}=O(\varepsilon)$.
Since
\begin{eqnarray*}
\left|
re^{i((\theta_j-\theta)-\theta_j)}
-|\eta_j\zeta_2|
\right|
& = &
\left|
re^{-i\theta}
-|\eta_j\zeta_2|
\right|
\;=\;
\left|
re^{i\theta}
-|\eta_j\zeta_2|
\right|
\\
& = &
\left|
re^{i((\theta_j+\theta)-\theta_j)}
-|\eta_j\zeta_2|
\right|
\,,
\end{eqnarray*}
then
$\theta'_{\varepsilon}=\theta_{\varepsilon}$.
So
\begin{eqnarray*}
\int_{|\zeta_1|=r,|\zeta_1-\eta_j\zeta_2|\leq2\varepsilon/C_j}
\frac{d\zeta_1}{\zeta_1(\zeta_1-\eta_j\zeta_2)}
& = &
\int_{\theta_j-\theta_{\varepsilon}}^{\theta_j+\theta_{\varepsilon}}
\frac{ire^{i\theta}\,d\theta}
{re^{i\theta}(re^{i\theta}-|\eta_j\zeta_2|e^{i\theta_j})}
\\
& = &
ie^{-i\theta_j}
\int_{-\theta_{\varepsilon}}^{\theta_{\varepsilon}}
\frac{d\theta}
{r\cos\theta-|\eta_j|\sqrt{1-r^2}+ir\sin\theta}
\,.
\end{eqnarray*}
With the change of variables
$t=\tan(\theta/2)$, one has
$t\in[-t_{\varepsilon},t_{\varepsilon}]$,
with
$t_{\varepsilon}=\tan(\theta_{\varepsilon}/2)=O(\varepsilon)$,
then
\begin{eqnarray*}
\int_{-\theta_{\varepsilon}}^{\theta_{\varepsilon}}
\frac{d\theta}
{r\cos\theta-|\eta_j|\sqrt{1-r^2}+ir\sin\theta}
& = &
\;\;\;\;\;\;\;\;\;\;\;\;\;\;\;\;\;\;\;\;\;\;\;\;\;\;\;\;\;\;\;\;\;\;\;\;\;\;\;\;\;\;\;\;
\;\;\;\;\;\;\;\;\;\;
\end{eqnarray*}
\begin{eqnarray*}
& = &
\int_{-t_{\varepsilon}}^{t_{\varepsilon}}
\frac{2\,dt/(1+t^2)}
{r\frac{1-t^2}{1+t^2}
-|\eta_j|\sqrt{1-r^2}
+ir\frac{2t}{1+t^2}}
\\
& = &
-\frac{2}{r+|\eta_j|\sqrt{1-r^2}}
\int_{-t_{\varepsilon}}^{t_{\varepsilon}}
\frac{dt}
{t^2-2i\frac{r}{r+|\eta_j|\sqrt{1-r^2}}\,t
-\frac{r-|\eta_j|\sqrt{1-r^2}}{r+|\eta_j|\sqrt{1-r^2}}}
\\
& = &
-\frac{2}{r+|\eta_j|\sqrt{1-r^2}}
\int_{-t_{\varepsilon}}^{t_{\varepsilon}}
\frac{dt}
{(t-t_+)(t-t_-)}
\,,
\end{eqnarray*}
with
\begin{eqnarray*}
\begin{cases}
t_+=\dfrac{ir+i|\eta_j|\sqrt{1-r^2}}{r+|\eta_j|\sqrt{1-r^2}}
=i\,,
\\
t_-=\dfrac{ir-i|\eta_j|\sqrt{1-r^2}}{r+|\eta_j|\sqrt{1-r^2}}
=i\,\dfrac{r-|\eta_j|\sqrt{1-r^2}}{r+|\eta_j|\sqrt{1-r^2}}
\,.
\end{cases}
\end{eqnarray*}
Since
\begin{eqnarray*}
\frac{1}{(t-t_+)(t-t_-)}
& = &
\frac{r+|\eta_j|\sqrt{1-r^2}}
{2i|\eta_j|\sqrt{1-r^2}}
\left(
\frac{1}{t-t_+}
-\frac{1}{t-t_-}
\right)
\,,
\end{eqnarray*}
one has
\begin{eqnarray*}
\int_{-\theta_{\varepsilon}}^{\theta_{\varepsilon}}
\frac{d\theta}
{r\cos\theta-|\eta_j|\sqrt{1-r^2}+ir\sin\theta}
& = &
\;\;\;\;\;\;\;\;\;\;\;\;\;\;\;\;\;\;\;\;\;\;\;\;\;\;\;\;\;\;\;\;\;\;\;\;\;\;\;\;\;\;\;\;
\;\;\;\;\;\;\;\;\;\;\;\;\;\;\;\;
\end{eqnarray*}
\begin{eqnarray*}
& = &
\frac{i}{|\eta_j|\sqrt{1-r^2}}
\int_{-t_{\varepsilon}}^{t_{\varepsilon}}
\left(
\frac{1}{t-t_+}
-\frac{1}{t-t_-}
\right)
dt
\\
& = &
\frac{i}{|\eta_j|\sqrt{1-r^2}}
\left(
Log(t_{\varepsilon}-t_+)
-Log(-t_{\varepsilon}-t_+)
-Log(t_{\varepsilon}-t_-)
+Log(-t_{\varepsilon}-t_-)
\right)
\,,
\end{eqnarray*}
$Log$ being the principal determination of the logarithm on
$\C\setminus\R^+$
(that is well-defined since
$t+,\,t_-\neq0$ for all $r\neq\alpha_j$).
$\Im(t_{\varepsilon}-t_+)=\Im(-t_{\varepsilon}-t_+)$
then
$Log(t_{\varepsilon}-t_+)
-Log(-t_{\varepsilon}-t_+)=
Log\left(
\dfrac{t_{\varepsilon}-t_+}{-t_{\varepsilon}-t_+}
\right)$, so
\begin{eqnarray*}
Log\left(
\frac{t_{\varepsilon}-t_+}{-t_{\varepsilon}-t_+}
\right)
& = &
Log
\left(
\frac{t_{\varepsilon}-i}{-t_{\varepsilon}-i}
\right)
\;=\;
Log(1+O(\varepsilon))
\;=\;
O(\varepsilon)
\,.
\end{eqnarray*}
On the other hand, one can write
$r=\alpha_j+r'$ with
$r'\in[-c_j\varepsilon,c_j\varepsilon]$
(and still keep the variable $r$)
so
\begin{eqnarray*}
t_- 
& = &
i\,\frac{
r+\alpha_j-|\eta_j|\sqrt{1-\alpha_j^2}
\sqrt{1-\frac{2\alpha_j}{1-\alpha_j^2}r+O(r^2)}
}
{r+\alpha_j+|\eta_j|\sqrt{1-\alpha_j^2}
\sqrt{1+O(r)}
}
\\
& = &
i\,\frac{
r+\alpha_j-\alpha_j
\left(1-\frac{\alpha_j}{1-\alpha_j^2}r+O(r^2)\right)
}
{r+\alpha_j+\alpha_j+O(r)}
\;=\;
i\,\frac{
\dfrac{r}{1-\alpha_j^2}+O(r^2)}
{2\alpha_j+O(r)}
\;=\;
\frac{i}{2\alpha_j(1-\alpha_j^2)}\,r+O\left(r^2\right)
\,.
\end{eqnarray*}
Then
\begin{eqnarray*}
|\pm t_{\varepsilon}-t_-|
& \geq &
|t_-|
\;=\;
\left|
\frac{i}{2\alpha_j(1-\alpha_j^2)}\,r+O\left(r^2\right)
\right|
\;\geq\;
\frac{|r|}{4\alpha_j(1-\alpha_j^2)}
\,,
\end{eqnarray*}
for all $\varepsilon$ small enough and
$r\in[-c_j\varepsilon,c_j\varepsilon]$,
$r\neq0$. Therefore
\begin{eqnarray*}
\left|
Log(\pm t_{\varepsilon}-t_-)
\right|
& \leq &
\left|
Log|\pm t_{\varepsilon}-t_-|
\,\right|
+|Arg(\pm t_{\varepsilon}-t_-)|
\\
& \leq &
\left|
Log
\left(
\frac{|r|}{4\alpha_j(1-\alpha_j^2)}
\right)
\right|
+\pi
\;\leq\;
|Log|r|\,|
+O(1)
\,,
\end{eqnarray*}
thus
\begin{eqnarray*}
\left|
\int_{-\theta_{\varepsilon}}^{\theta_{\varepsilon}}
\frac{d\theta}
{r\cos\theta-|\eta_j|\sqrt{1-r^2}+ir\sin\theta}
\right|
& \leq &
\frac{1}{|\eta_j|\sqrt{1-(r+\alpha_j)^2}}
\left(
O(\varepsilon)
+2|Log|r|\,|+O(1)
\right)
\\
& = &
\frac{2}{\alpha_j}|Log|r|\,|+O(1)
\;=\;
O(Log|r|)
\,.
\end{eqnarray*}

Finally, we get for all
$\varepsilon$ small enough
\begin{eqnarray*}
\sum_{k=1}^K
\sum_{j=j_{k-1}+1}^{j_k}
\left|
\int_{\alpha_{j_k}-c_{j_k}\varepsilon}^{\alpha_{j_k}+c_{j_k}\varepsilon}
(-2rdr)
\int_{|\zeta_2|=\sqrt{1-r^2}}
\widetilde{\psi_j}(r,\zeta_2)
d\zeta_2
\int_{|\zeta_1|=r,|\zeta_1-\eta_j\zeta_2|\leq2\varepsilon/C_j}
\frac{d\zeta_1}{\zeta_1(\zeta_1-\eta_j\zeta_2)}
\right|
& \leq &
\end{eqnarray*}
\begin{eqnarray*}
\;\;\;\;\;\;\;\;\;\;\;\;\;\;\;\;\;\;\;\;\;\;\;\;\;\;\;\;\;\;\;\;\;\;\;\;
\;\;\;\;\;\;\;\;\;\;\;\;\;\;\;\;\;\;\;\;\;\;\;\;\;\;\;\;\;\;\;\;\;\;\;\;
& \leq &
\sum_{j=1}^N
\int_{-c_{j}\varepsilon}^{c_{j}\varepsilon}
O(|r|\,Log|r|)dr
\;\xrightarrow[\varepsilon\rightarrow0]{}0
\,.
\end{eqnarray*}

\end{proof}

This result yields
to the following consequence that will be usefull
in Section~\ref{proofs}, Subsection~\ref{endproofs}
(Lemma~\ref{lemmextensio}).

\begin{corollary}\label{equivcv0}

One has, for all
$f\in\mathcal{O}\left(\C^2\right)$,
all $N\geq1$
and all $z\in\mathbb{B}_2$,
\begin{eqnarray*}
\lim_{\varepsilon\rightarrow0}
\frac{\prod_{j=1}^N(z_1-\eta_jz_2)}{(2i\pi)^2}
\int_{\prod_{j=1}^N
|\zeta_1-\eta_j\zeta_2|>\varepsilon}
\frac{f(\zeta)\omega'\left(\overline{\zeta}\right)
\wedge\omega(\zeta)}
{\prod_{j=1}^N(\zeta_1-\eta_j\zeta_2)
(1-<\overline{\zeta},z>)^2}
& = &
\;\;\;\;\;\;\;\;
\end{eqnarray*}
\begin{eqnarray*}
\;\;\;\;\;\;\;\;\;\;\;\;\;\;\;\;\;\;\;\;\;\;\;\;\;\;\;\;\;\;
\;\;\;\;\;\;\;\;\;\;\;\;\;\;\;\;\;\;\;\;\;\;\;\;\;\;\;\;\;\;
\;\;\;\;\;\;\;\;\;\;\;\;\;\;\;
& = &
R_N(f;\eta)(z)
-\sum_{k+l\geq N}
a_{k,l}z_1^kz_2^l
\,.
\end{eqnarray*}

\end{corollary}

\bigskip

\subsection{Some properties of $\Delta_{p}$}

In this part, $\{\eta_j\}_{j\geq1}$ will be any set
of points all differents, and $h$ will be any function
defined on the $\eta_j,\,j\geq1$.

We begin with this first result that follows from the definition
of $\Delta_p$.

\begin{lemma}

For all $p\geq0$ and $0\leq q\leq p$,
\begin{eqnarray}\nonumber
\Delta_{p,(\eta_p,\ldots,\eta_1)}(h)
\left(\eta_{p+1}\right)
& = &
\Delta_{1,\eta_p}
\left[
\zeta\mapsto
\Delta_{p-1,(\eta_{p-1},\ldots,\eta_1)}(h)(\zeta)
\right]
\left(\eta_{p+1}\right)
\\\nonumber
& = &
\Delta_{p-1,(\eta_{p},\ldots,\eta_2)}
\left[
\zeta\mapsto
\Delta_{1,\eta_1}(h)(\zeta)
\right]
\left(\eta_{p+1}\right)
\\\nonumber
& = &
\Delta_{p-q,(\eta_{p},\ldots,\eta_{q+1})}
\left[
\zeta\mapsto
\Delta_{q,(\eta_{q},\ldots,\eta_1)}
(h)(\zeta)
\right]
\left(\eta_{p+1}\right)
\\\nonumber
& = &
\Delta_{1,\eta_p}
\left[
\zeta_p\mapsto
\Delta_{1,\eta_{p-1}}
\left[
\cdots
\left[
\zeta_1\mapsto
\Delta_{1,\eta_1}
(h)\left(\zeta_1\right)
\right]
\cdots
\right]
\left(\zeta_p\right)
\right]
\left(\eta_{p+1}\right)
\,.
\end{eqnarray}

\end{lemma}

\bigskip

Now we will prove the two following results that will be usefull in
Section~\ref{proofequivbounded}.

\begin{lemma}\label{lagrangebis}

Let $h$ be a function that is defined
on every $\eta_j,\,j\geq1$. Then for all
$N\geq1$,
one has
\begin{eqnarray*}
\sum_{p=1}^N
\prod_{j=1,j\neq p}^N
\frac{X-\eta_j}{\eta_p-\eta_j}
\,
h(\eta_p)
& = &
\sum_{p=0}^{N-1}
\prod_{j=1}^p
(X-\eta_j)
\,
\Delta_{p,(\eta_p,\ldots,\eta_1)}
(h)\left(\eta_{p+1}\right)
\,.
\end{eqnarray*}

\end{lemma}

\begin{proof}

This relation can be proved by induction on $N\geq1$. It is obvious
for $N=1$. If it is true for $N\geq1$, then
\begin{eqnarray*}
\sum_{p=1}^{N+1}
\prod_{j=1,j\neq p}^{N+1}
\frac{X-\eta_j}{\eta_p-\eta_j}
h(\eta_p)
& = &
(X-\eta_{1})
\sum_{p=2}^{N+1}
\prod_{j=2,j\neq p}^{N+1}
\frac{X-\eta_j}{\eta_p-\eta_j}
\frac{h(\eta_p)\pm h(\eta_{1})}
{\eta_{p}-\eta_{1}}
+h(\eta_{1})
\prod_{j=2}^{N+1}
\frac{X-\eta_j}{\eta_{1}-\eta_j}
\\
& = &
\sum_{p=2}^{N+1}
\prod_{j=2,j\neq p}^{N+1}
\frac{X-\eta_j}{\eta_p-\eta_j}
h_{1,X}(\eta_p)
+h(\eta_{1})
\sum_{p=1}^{N+1}
\prod_{j=1,j\neq p}^{N+1}
\frac{X-\eta_j}{\eta_p-\eta_j}
\,,
\end{eqnarray*}
where
\begin{eqnarray*}
h_{1,X}(u)
& := &
(X-\eta_{1})
\frac{h(u)-h(\eta_{1})}
{u-\eta_{1}}
\,.
\end{eqnarray*}
By induction,
\begin{eqnarray*}
\sum_{p=2}^{N+1}
\prod_{j=2,j\neq p}^{N+1}
\frac{X-\eta_j}{\eta_p-\eta_j}
h_{1,X}(\eta_p)
& = &
\sum_{p=1}^{N}
\prod_{j=2}^p
(X-\eta_j)
\,
\Delta_{p-1,(\eta_{p},\ldots,\eta_2)}
\left(h_{1,X}\right)\left(\eta_{p+1}\right)
\\
& = &
\sum_{p=1}^{N}
(X-\eta_1)
\prod_{j=2}^p
(X-\eta_j)
\Delta_{p-1,(\eta_{p},\ldots,\eta_2)}
\left(u\mapsto
\Delta_{1,\eta_1}(h)(u)
\right)\left(\eta_{p+1}\right)
\\
& = &
\sum_{p=1}^N
\prod_{j=1}^p
(X-\eta_j)
\Delta_{p,(\eta_p,\ldots,\eta_1)}
(h)(\eta_{p+1})
\,.
\end{eqnarray*}
On the other hand,
\begin{eqnarray*}
\sum_{p=1}^{N+1}
\prod_{j=1,j\neq p}^{N+1}
\frac{X-\eta_j}{\eta_p-\eta_j}
& = &
1
\,,
\end{eqnarray*}
since $1$ is the unique polynomial of degree at most
$N$ that coincides with $1$ on $N+1$ points. Then
\begin{eqnarray*}
\sum_{p=1}^{N+1}
\prod_{j=1,j\neq p}^{N+1}
\frac{X-\eta_j}{\eta_p-\eta_j}
h(\eta_p)
& = &
\sum_{p=1}^N
\prod_{j=1}^p
(X-\eta_j)
\Delta_{p,(\eta_p,\ldots,\eta_1)}
(h)(\eta_{p+1})
+h(\eta_1)
\\
& = &
\sum_{p=0}^N
\prod_{j=1}^p
(X-\eta_j)
\Delta_{p,(\eta_p,\ldots,\eta_1)}
(h)(\eta_{p+1})
\,,
\end{eqnarray*}
and this achieves the induction.

\end{proof}

\begin{lemma}\label{Deltaprod}

For all $p\geq0$, $g,\,h$ functions defined on the
$\eta_j$,
\begin{eqnarray*}
\Delta_{p,(\eta_p,\ldots,\eta_1)}(gh)\left(\eta_{p+1}\right)
& = &
\sum_{q=0}^p
\Delta_{p-q,(\eta_p,\ldots,\eta_{q+1})}(g)\left(\eta_{p+1}\right)
\,
\Delta_{q,(\eta_{q},\ldots,\eta_{1})}(h)\left(\eta_{q+1}\right)
\,.
\end{eqnarray*}

\end{lemma}

\begin{proof}

We prove this result by induction on $p\geq0$.
It is obvious for $p=0$. If $p\geq0$ one has
\begin{eqnarray*}
\Delta_{p+1,(\eta_{p+1},\ldots,\eta_1)}(gh)
\left(\eta_{p+2}\right)
& = &
\Delta_{p,(\eta_{p+1},\ldots,\eta_2)}
\left[
\zeta\mapsto
\frac{(gh)(\zeta)-(gh)(\eta_1)\pm g(\zeta)h(\eta_1)}
{\zeta-\eta_1}
\right]
\left(\eta_{p+2}\right)
\\
& = &
\Delta_{p,(\eta_{p+1},\ldots,\eta_2)}
\left[
g(\zeta)
\Delta_{1,\eta_1}(h)(\zeta)
\right]
\left(\eta_{p+2}\right)
\\
& &
+\;h(\eta_1)
\Delta_{p,(\eta_{p+1},\ldots,\eta_2)}
\left[
\Delta_{1,\eta_1}(g)(\zeta)
\right]
\left(\eta_{p+2}\right)
\\
& = &
\sum_{q=0}^p
\Delta_{p-q,(\eta_{p+1},\ldots,\eta_{q+2})}(g)\left(\eta_{p+2}\right)
\Delta_{q,(\eta_{q+1},\ldots,\eta_{2})}
\left[
\Delta_{1,\eta_1}(h)(\zeta)
\right]
\left(\eta_{q+2}\right)
\\
& &
+\;
h(\eta_1)
\Delta_{p+1,(\eta_{p+1},\ldots,\eta_1)}
(g)
\left(\eta_{p+2}\right)
\\
& = &
\sum_{q=1}^{p+1}
\Delta_{p-q+1,(\eta_{p+1},\ldots,\eta_{q+1})}(g)\left(\eta_{p+2}\right)
\Delta_{q,(\eta_{q},\ldots,\eta_{1})}
(h)
\left(\eta_{q+1}\right)
\\
& &
+\;
h(\eta_1)
\Delta_{p+1,(\eta_{p+1},\ldots,\eta_1)}
(g)
\left(\eta_{p+2}\right)
\,,
\end{eqnarray*}
and this achieves the induction.

\end{proof}

The following result is an example of explicit calculation
of the $\Delta_p$ of an holomorhic function.

\begin{lemma}\label{Deltanal}

Let be 
$h\in\mathcal{O}(D(w_0,r))$ and
$h(w)=\sum_{n\geq0}a_n(w-w_0)^n$
its Taylor expansion for all
$|w-w_0|<r$. Assume that
$\forall\,j\geq1,\,\eta_j\in D(w_0,r)$. Then
for all $p\geq0$,
\begin{eqnarray}\nonumber
\Delta_{p,(\eta_p,\ldots,\eta_1)}
(h)\left(\eta_{p+1}\right)
& = &
\sum_{n\geq p}
a_n
\sum_{l_1=0}^{n-p}
(\eta_1-w_0)^{n-p-l_1}
\sum_{l_2=0}^{l_1}
(\eta_2-w_0)^{l_1-l_2}
\;\cdots
\\\label{Deltanal1}
& &
\cdots\;
\sum_{l_{p-1}=0}^{l_{p-2}}
(\eta_{p-1}-w_0)^{l_{p-2}-l_{p-1}}
\sum_{l_{p}=0}^{l_{p-1}}
(\eta_{p}-w_0)^{l_{p-1}-l_{p}}
(\eta_{p+1}-w_0)^{l_p}
\,.
\end{eqnarray}
In particular,
\begin{eqnarray*}
\lim_{\eta_1,\ldots,\eta_p,\eta_{p+1}\rightarrow w_0}
\Delta_{p,(\eta_p,\ldots,\eta_1)}
(h)\left(\eta_{p+1}\right)
& = &
a_p\;=\;
\frac{h^{(p)}(0)}{p!}
\,.
\end{eqnarray*}
On the other hand, if
$h\in\C[w]$, then for any subset
$\{\eta_j\}_{j\geq1}\subset\C$ and all
$p>\deg h$,
\begin{eqnarray*}
\Delta_{p,(\eta_p,\ldots,\eta_1)}
(h)\left(\eta_{p+1}\right)
& = & 0
\,.
\end{eqnarray*}

\end{lemma}

\begin{proof}

The second relation is a consequence
of~(\ref{Deltanal1}) and the third with the choice
of $r=+\infty$.

By translation we can assume that
$w_0=0$. The lemma will be proved by induction
on $p\geq0$. This is true for $p=0$ if we admit that
$l_0=n-p=n$ and
\begin{eqnarray*}
\prod_{j=1}^0
\sum_{l_{j}=0}^{l_{j-1}}
\eta_j^{l_{j-1}-l_j}
\eta_1^{l_0}
& = &
\eta_1^n
\,.
\end{eqnarray*}
Now if it is true for $p\geq0$ then
\begin{eqnarray*}
\Delta_{p+1,(\eta_{p+1},\ldots,\eta_1)}
(h)\left(\eta_{p+2}\right)
& = &
\sum_{n\geq p}
a_n
\sum_{l_1=0}^{n-p}
\eta_1^{n-p-l_1}
\cdots\;
\sum_{l_{p}=0}^{l_{p-1}}
\eta_{p}^{l_{p-1}-l_{p}}
\Delta_{1,\eta_{p+1}}
\left(\zeta^{l_p}\right)
\left(\eta_{p+2}\right)
\\
& = &
\sum_{n\geq p}
a_n
\sum_{l_1=0}^{n-p}
\eta_1^{n-p-l_1}
\cdots\;
\sum_{l_{p}=0}^{l_{p-1}}
\eta_{p}^{l_{p-1}-l_{p}}
{\bf 1}_{l_p\geq1}
\sum_{l_{p+1}=0}^{l_p-1}
\eta_{p+1}^{l_p-1-l_{p+1}}
\eta_{p+2}^{l_{p+1}}
\\
& = &
\sum_{n\geq p}
a_n
\sum_{l_1=0}^{n-p}
\eta_1^{n-p-l_1}
\cdots\;
{\bf l}_{l_{p-1}\geq1}
\sum_{l_{p}=0}^{l_{p-1}-1}
\eta_{p}^{l_{p-1}-l_{p}-1}
\sum_{l_{p+1}=0}^{l_p}
\eta_{p+1}^{l_p-l_{p+1}}
\eta_{p+2}^{l_{p+1}}
\\
& = &
\sum_{n\geq p}
a_n
\sum_{l_1=1}^{n-p}
\eta_1^{n-p-l_1}
\cdots\;
\sum_{l_{p}=0}^{l_{p-1}}
\eta_{p}^{l_{p-1}-l_{p}}
\sum_{l_{p+1}=0}^{l_p}
\eta_{p+1}^{l_p-l_{p+1}}
\eta_{p+2}^{l_{p+1}}
\\
& = &
\sum_{n\geq p+1}
a_n
\sum_{l_1=0}^{n-p-1}
\eta_1^{n-p-l_1-1}
\cdots\;
\sum_{l_{p+1}=0}^{l_p}
\eta_{p+1}^{l_p-l_{p+1}}
\eta_{p+2}^{l_{p+1}}
\,,
\end{eqnarray*}
and the induction is achieved.

\end{proof}

The following result will be usefull in
Section~\ref{proofs},
Subsection~\ref{proofsuperbounded} (Proposition~\ref{unifDelta}).

\begin{lemma}\label{Deltapermut}

For all $p\geq0$, $h$ function defined on the $\eta_j$,
\begin{eqnarray*}
\Delta_{p,(\eta_{\sigma(p)},\ldots,\eta_{\sigma(1)})}(h)
\left(\eta_{\sigma(p+1)}\right)
& = &
\Delta_{p,(\eta_p,\ldots,\eta_1)}(h)
\left(\eta_{p+1}\right)
\,.
\end{eqnarray*}

\end{lemma}

\begin{proof}

We prove the lemma by induction on $p\geq0$.
It is obvious for $p=0$ and $p=1$. Let be $p\geq2$ and
$\sigma\in\mathfrak{S}_{p+1}$. First,
assume that
$\sigma(1)=1$. Then by induction
\begin{eqnarray*}
\Delta_{p,(\eta_{\sigma(p)},\ldots,\eta_{\sigma(1)})}(h)
\left(\eta_{\sigma(p+1)}\right)
& = &
\Delta_{p,(\eta_{\sigma(p)},\ldots,\eta_{\sigma(2)},\eta_1)}(h)
\left(\eta_{\sigma(p+1)}\right)
\\
& = &
\Delta_{p-1,(\eta_{\sigma(p)},\ldots,\eta_{\sigma(2)})}
\left[
\zeta\mapsto
\Delta_{1,\eta_1}(h)(\zeta)
\right]
\left(\eta_{\sigma(p+1)}\right)
\\
& = &
\Delta_{p-1,(\eta_p,\ldots,\eta_2)}
\left[
\zeta\mapsto
\Delta_{1,\eta_1}(h)(\zeta)
\right]
\left(\eta_{p+1}\right)
\\
& = &
\Delta_{p,(\eta_p,\ldots,\eta_1)}(h)
\left(\eta_{p+1}\right)
\,.
\end{eqnarray*}

Now one can assume that $\sigma(1)\neq1$ and consider the transposition
$\tau=(1\;\sigma(1))$. Then
$(\tau\sigma)(1)=1$ and 
\begin{eqnarray*}
\Delta_{p,(\eta_{(\tau\sigma)(p)},\ldots,\eta_{(\tau\sigma)(1)})}(h)
\left(\eta_{(\tau\sigma)(p+1)}\right)
& = &
\Delta_{p,(\eta_p,\ldots,\eta_1)}(h)
\left(\eta_{p+1}\right)
\,.
\end{eqnarray*}
Now let assume that the lemma is also proved for all transposition
$(1\;j),\,2\leq j\leq p+1$. Then
\begin{eqnarray*}
\Delta_{p,(\eta_{\sigma(p)},\ldots,\eta_{\sigma(1)})}(h)
\left(\eta_{\sigma(p+1)}\right)
& = &
\Delta_{p,(\eta_{\tau[(\tau\sigma)(p)]},\ldots,\eta_{\tau[(\tau\sigma)(1)]})}(h)
\left(\eta_{\tau[(\tau\sigma)(p+1)]}\right)
\\
& = &
\Delta_{p,(\eta_{(\tau\sigma)(p)},\ldots,\eta_{(\tau\sigma)(1)})}(h)
\left(\eta_{(\tau\sigma)(p+1)}\right)
\\
& = &
\Delta_{p,(\eta_p,\ldots,\eta_1)}(h)
\left(\eta_{p+1}\right)
\end{eqnarray*}
and the lemma will be proved. So it is sufficient to
prove it for any $\tau=(1\;j),\,2\leq j\leq p+1$.
On the other hand, 
\begin{eqnarray*}
\Delta_{p,(\eta_p,\ldots,\eta_1)}(h)
\left(\eta_{p+1}\right)
& = &
\frac{
\Delta_{p-1,(\eta_{p-1},\ldots,\eta_{1})}(h)
\left(\eta_{p}\right)
-
\Delta_{p-1,(\eta_{p-1},\ldots,\eta_{1})}(h)
\left(\eta_{p+1}\right)
}{\eta_{p}-\eta_{p+1}}
\\
& = &
\Delta_{p,(\eta_{p+1},\eta_{p-1},\ldots,\eta_1)}(h)
\left(\eta_{p}\right)
\,,
\end{eqnarray*}
then it is also true for
$\tau_p:=(p\;p+1)$.

Let assume that it is also true for any
permutation that fixes $p+1$. Then it will be true
for all $(1\;j)$, $2\leq j\leq p$, also for
$(1\;p+1)=\tau_p(1\;p)\tau_p$
and the proof will be achieved.
\bigskip

Finally let be
$\sigma\in\mathfrak{S}_{p+1}$ such that
$\sigma(p+1)=p+1$. Then
$\prod_{j=1}^p
\left(X-\eta_{\sigma(j)}\right)
=
\prod_{j=1}^p(X-\eta_j)$
and
\begin{eqnarray*}
\sum_{q=0}^{p}
\prod_{j=1}^q\left(X-\eta_{\sigma(j)}\right)
\Delta_{q,(\eta_{\sigma(q)},\ldots,\eta_{\sigma(1)})}(h)
\left(\eta_{\sigma(q+1)}\right)
& = &
\;\;\;\;\;\;\;\;\;\;\;\;\;\;\;\;\;\;\;\;\;\;\;\;\;\;\;\;\;\;\;\;\;\;\;\;\;\;\;\;\;\;\;\;\;\;
\end{eqnarray*}
\begin{eqnarray*}
& = &
\sum_{q=0}^{p-1}
\left(
\prod_{j=1}^q\left(X-\eta_{\sigma(j)}\right)
\Delta_{q,(\eta_{\sigma(q)},\ldots,\eta_{\sigma(1)})}(h)
\left(\eta_{\sigma(q+1)}\right)
\right)
\\
& &
+
\prod_{j=1}^p(X-\eta_{j})
\Delta_{p,(\eta_{\sigma(p)},\ldots,\eta_{\sigma(1)})}(h)
\left(\eta_{\sigma(p+1)}\right)
\,.
\end{eqnarray*}
Since for all $q=0,\ldots,p-1$, the family
$\{1,X-\eta_1,\ldots,(X-\eta_1)\cdots(X-\eta_q)\}$
is a basis of
$\mathbb{C}_q[X]=\{P\in\C[X],\;\deg P\leq q\}$, one has
\begin{eqnarray*}
\prod_{j=1}^q(X-\eta_{\sigma(j)})
& = &
\sum_{l=0}^qc_{q,l}
\prod_{j=1}^l(X-\eta_j)
\,,
\end{eqnarray*}
then
\begin{eqnarray*}
\sum_{q=0}^{p}
\prod_{j=1}^q\left(X-\eta_{\sigma(j)}\right)
\Delta_{q,(\eta_{\sigma(q)},\ldots,\eta_{\sigma(1)})}(h)(\eta_{\sigma(q+1)})
& = &
\;\;\;\;\;\;\;\;\;\;\;\;\;\;\;\;\;\;\;\;\;\;\;\;\;\;\;\;\;\;\;\;\;\;\;
\;\;\;\;\;\;\;\;\;\;\;\;\;
\end{eqnarray*}
\begin{eqnarray*}
\;\;\;\;\;\;\;\;\;\;\;\;\;\;\;\;\;\;
& = &
\sum_{q=0}^{p-1}
C_q
\prod_{j=1}^q(X-\eta_j)
+
\prod_{j=1}^p(X-\eta_{j})
\Delta_{p,(\eta_{\sigma(p)},\ldots,\eta_{\sigma(1)})}(h)
\left(\eta_{\sigma(p+1)}\right)
\,.
\end{eqnarray*}
On the other hand, one has by Lemma~\ref{lagrangebis}
\begin{eqnarray*}
\sum_{q=0}^{p}
\prod_{j=1}^q\left(X-\eta_{\sigma(j)}\right)
\Delta_{q,(\eta_{\sigma(q)},\ldots,\eta_{\sigma(1)})}(h)
\left(\eta_{\sigma(q+1)}\right)
& = &
\;\;\;\;\;\;\;\;\;\;\;\;\;\;\;\;\;\;\;\;\;\;\;\;\;\;\;\;\;\;
\;\;\;\;\;\;\;\;\;\;\;\;\;\;\;
\end{eqnarray*}
\begin{eqnarray*}
\;\;\;\;\;\;
& = &
\sum_{q=1}^{p+1}
\left(
\prod_{j=1,j\neq q}^{p+1}
\frac{X-\eta_{\sigma(j)}}{\eta_{\sigma(q)}-\eta_{\sigma(j)}}
\right)
h\left(\eta_{\sigma(q)}\right)
\;=\;
\sum_{q=1}^{p+1}
\left(
\prod_{j=1,j\neq q}^{p+1}
\frac{X-\eta_j}{\eta_q-\eta_j}
\right)
h(\eta_q)
\\
& = &
\sum_{q=0}^{p-1}
\prod_{j=1}^q(X-\eta_j)
\Delta_{q,(\eta_q,\ldots,\eta_1)}(h)(\eta_{q+1})
+\prod_{j=1}^p(X-\eta_j)
\Delta_{p,(\eta_p,\ldots,\eta_1)}(h)(\eta_{p+1})
\,.
\end{eqnarray*}
Since the family
$\left\{1\,,X-\eta_1,\,(X-\eta_1)(X-\eta_2),\ldots,
(X-\eta_1)\cdots(X-\eta_{p})\right\}$
is a basis of
$\mathbb{C}_{p}[X]$, it follows that
\begin{eqnarray*}
\Delta_{p,(\eta_{\sigma(p)},\ldots,\eta_{\sigma(1)})}(h)
\left(\eta_{\sigma(p+1)}\right)
& = &
\Delta_{p,(\eta_p,\ldots,\eta_1)}(h)(\eta_{p+1})
\end{eqnarray*}
and the lemma is proved.

\end{proof}

\subsection{About the formula $E_N(\cdot,\eta)$}\label{uniquEN}

In this part we want to justify what we mean when we claim
that the formula $E_N(f;\eta)$ is the canonical interpolation formula
for any $f\in\mathcal{O}\left(\C^2\right)$.
We set, for all $p\geq1$,
\begin{eqnarray}
w_p(z) & := &
\frac{z_2+\overline{\eta_p}z_1}{1+|\eta_p|^2}
\,.
\end{eqnarray}
Then the formula $E_N(f;\eta)$ can be written as
\begin{eqnarray*}
E_N(f;\eta)(z)
& = &
\sum_{p=1}^N
\prod_{j=p+1}^N(z_1-\eta_jz_2)
\sum_{q=p}^N
\frac{1+\eta_p\overline{\eta_q}}{1+|\eta_q|^2}
\frac{1}{\prod_{j=p,j\neq q}^N(\eta_q-\eta_j)}
\frac{
\left[
f\left(
\eta_qw_q(z),w_q(z)
\right)
\right]_{N-p}
}{w_q(z)^{N-p}}
\,,
\end{eqnarray*}
where $[h]_{N-p}$ is the truncation of $h$
at order $N-p$, ie
$[h]_{N-p}(w)=\sum_{m\geq N-p}
\frac{w^m}{m!}h^{(m)}(0)$.

On the other hand, notice that the point
$\left(\eta_qw_q(z),w_q(z)\right)$ is the orthogonal projection
of $z$ on the line $\{z_1-\eta_qz_2=0\}$ with respect to
the hermitian scalar product on $\C^2$, since
$\left(\eta_qw_q(z),w_q(z)\right)=<z,\overline{u_q}>u_q$, with
$u_q:=(\eta_q,1)/\sqrt{1+|\eta_q|^2}$ being a normalized director vector
of $\{z_1-\eta_qz_2=0\}$. In particular, $z\in\C^2$ being given,
$f\left(\eta_qw_q(z),w_q(z)\right)$ is a naturel way to use
the restriction $f|_{\{z_1=\eta_qz_2\}}$.

Finally, on the expression of $E_N(f;\eta)$ appear
derivatives of the restrictions of $f$ of order at most
$\max(N-2,0)$ since $p\geq1$ and
\begin{eqnarray*}
\left[
f\left(
\eta_qw_q(z),w_q(z)
\right)
\right]_{N-p}
& = &
f\left(
\eta_qw_q(z),w_q(z)
\right)
-
\sum_{p=0}^{N-p-1}
\frac{w_q(z)^{m}}{m!}
\frac{\partial^m}{\partial v^m}|_{v=0}
(f(\eta_qv,v))
\,.
\end{eqnarray*}

Now we can give the following result about the fact that
the formula $E_N(f;\eta)$ is a canonical interpolation formula.

\begin{lemma}\label{lemmEN}

For all $N\geq1$, $E_N(f;\eta)$ is essentially the unique interpolation formula
betwen the interpolation formula that fix
$\C_{N-1}[z_1,z_2]$ and have the following expression
\begin{eqnarray*}
F=\sum_{q=1}^N
F_q(z)
\,,
\end{eqnarray*}
where
$F_q(z)=
\sum_{k+l\leq N-1}C_{q,k,l}z_1^kz_2^l$,
with $C_{q,k,l}$
being an operator on the space
$\mathcal{O}(\C)_q:=
\left(
\mathcal{O}\left(\C^2\right)
\right)|_{\{z_1=\eta_qz_2\}}$,
of order at most $\max(N-2,0)$.

\end{lemma}

\begin{proof}

First, in the expression of such an operator $F_q$ on
$\mathcal{O}(\C)_q$ and any 
$f\in\mathcal{O}\left(\C^2\right)$,
there must appear parts of
$f(\eta_qv_q(z),v_q(z))$ where 
$v_q(z)=a_qz_1+b_qz_2$ is such that
$\|z\|\geq\|(\eta_qv_q(z),v_q(z))\|
=|v_q(z)|\sqrt{1+|\eta_q|^2}$,
ie
$|v_q(z)|\leq\|z\|/\sqrt{1+|\eta_q|^2}$.
This yields to
\begin{eqnarray}\label{lemmEN1}
\sup_{\|z\|\leq1}|a_qz_1+b_qz_2|
\;=\;
\sqrt{|a_q|^2+|b_q|^2}
& \leq &
\frac{1}{\sqrt{1+|\eta_q|^2}}
\,.
\end{eqnarray}
On the other hand, the condition of interpolation must satisfy
$F(f)_{|\{z_1=\eta_qz_2\}}=f_{|\{z_1=\eta_qz_2\}}$,
then
\begin{eqnarray}\label{lemmEN2}
1\;=\;
|v_q(\eta_q,1)|
\;=\;
|\eta_qa_q+b_q|
& \leq &
\sqrt{|a_q|^2+|b_q|^2}
\sqrt{1+|\eta_q|^2}
\,.
\end{eqnarray}
It folows from~(\ref{lemmEN1}) and~(\ref{lemmEN2})
that
$\sqrt{|a_q|^2+|b_q|^2}=1/\sqrt{1+|\eta_q|^2}$ and
$(a_q,b_q)=\lambda(\overline{\eta_q},1)$,
$\lambda\in\C$. Then
$|\lambda|=1/(1+|\eta_q|^2)$ and
\begin{eqnarray}\label{lemmENvq}
v_q(z) & = &
\omega_q\frac{\overline{\eta_q}z_1+z_2}{1+|\eta_q|^2}
\,,
\;
|\omega_q|=1
\,.
\end{eqnarray}
Finally, the condition
$(\eta_qv_q(\eta_q,1),v_q(\eta_q,1))
=(\eta_q,1)$ yields to $\omega_q=1$, then
\begin{eqnarray}\label{lemmENvqwq}
v_q(z) & = &
\frac{\overline{\eta_q}z_1+z_2}{1+|\eta_q|^2}
\;=\;
w_q(z)
\,.
\end{eqnarray}

Next, for any $P\in\C_{N-1}[z]$, in particular,
for all $z_2\in\C\setminus\{0\}$, $P(\cdot,z_2)\in\C[z_1]$.
It follows by Lemma~\ref{lagrangebis} that
\begin{eqnarray}\nonumber
P(z) & = &
\sum_{p=1}^N
\prod_{j=1,j\neq p}^N
\frac{z_1-\eta_jz_2}{\eta_pz_2-\eta_jz_2}
\,P(\eta_pz_2,z_2)
\\\label{lemmENP1}
& = &
\sum_{p=1}^N
\prod_{j=p+1}^N(z_1-\eta_jz_2)
\Delta_{N-p,(\eta_{p+1}z_2,\ldots,\eta_Nz_2)}
\left(
\zeta\mapsto
P(\zeta,z_2)
\right)\left(\eta_pz_2\right)
\,.
\end{eqnarray}
On the other hand, by Lemma~\ref{Deltanal}, for all $p=1,\ldots,N$,
$\Delta_{N-p,(\eta_{p+1}z_2,\ldots,\eta_Nz_2)}
\left(P(\cdot,z_2)\right)\left(\eta_pz_2\right)
\in\C[z_2]$. Moreover,
\begin{eqnarray}\nonumber
\deg
\Delta_{N-p,(\eta_{p+1}z_2,\ldots,\eta_Nz_2)}
\left(P(\cdot,z_2)\right)\left(\eta_pz_2\right)
& \leq &
\deg_{z_1}P-(N-p)+\deg_{z_2}P
\\\label{lemmENdeg}
& \leq &
p-1\,.
\end{eqnarray}
It follows from~(\ref{lemmENvqwq}), (\ref{lemmENP1}) and~(\ref{lemmENdeg})
that one can write $F$ as
\begin{eqnarray}\nonumber
F & = &
\sum_{q=1}^N
\sum_{p=1}^N
\prod_{j=p+1}^N
(z_1-\eta_jz_2)
\Delta_{N-p,(\eta_{p+1}z_2,\ldots,\eta_Nz_2)}
\left(F_q(\cdot,z_2)\right)\left(\eta_pz_2\right)
\\\label{lemmENF}
& = &
\sum_{q,p=1}^N
Q_{q,p}(z_2)
\left(
\prod_{j=p+1}^N
(z_1-\eta_jz_2)
\right)
F_{q,p}
\,,
\end{eqnarray}
where $Q_{q,p}\in\C[z_2]$ and,
for all $f\in\mathcal{O}\left(\C^2\right)$,
\begin{eqnarray}
& &
F_{q,p}(f)\;=\;
\sum_{l=0}^{d_{q,p}}
c_{N,q,p,l}w_q(z)^l
\sum_{m\geq r_{q,p}}
\frac{w_q(z)^{m-r_{q,p}}}{m!}
\frac{\partial^m}{\partial v^m}|_{v=0}
\left[
f(\eta_qv,v)
\right]
\,,
\end{eqnarray}
with
$r_{q,p}\leq N-1$, $\forall\,l=0,\ldots,d_{q,p}$
(since $F_{q,p}$ is of order at most $N-2$).
\bigskip

Now, if $N\geq1$ and $F_N$ is such an operator,
the operator $F_N-E_N$ can still be written as
the expression~(\ref{lemmENF}). Moreover,
$(F_N-E_N)_{|\C_{N-1}[z]}=0$ since
$E_N$ and $F_N$ fix $\C_{N-1}[z]$.
On the other hand,
let consider the Lagrange monomials
\begin{eqnarray}\label{Lq}
L_q(z) & := &
\prod_{j=1,j\neq q}^N
\frac{z_1-\eta_jz_2}{\eta_q-\eta_j}
\;,\;
1\leq q\leq N\,.
\end{eqnarray}
In particular, $\deg L_q\leq N-1$ and
${L_q}_{|\{z_1-\eta_pz_2\}}={\bf 1}_{p=q}z_2^{N-1}$.
Then one has,
for all $N\geq1$ and $q=1,\ldots,N$,
\begin{eqnarray}\nonumber
0
\;=\;
(F_N-E_N)(L_q)
& = &
\sum_{p=1}^N
Q_{q,p}(z_2)
\prod_{j=p+1}^N
(z_1-\eta_jz_2)
\times
\\\nonumber
& &
\times
\sum_{l=0}^{d_{q,p}}
c_{N,q,p,l}w_q(z)^l
\sum_{m\geq r_{q,p}}
\frac{w_q(z)^{m-r_{q,p}}}{m!}
\frac{\partial^m}{\partial v^m}|_{v=0}
\left[v^{N-1}\right]
\\\label{lemmENFN}
& = &
\sum_{p=1}^N
Q_{q,p}(z_2)
\prod_{j=p+1}^N
(z_1-\eta_jz_2)
\sum_{l=0}^{d_{q,p}}
c_{N,q,p,l}w_q(z)^{l+N-1-r_{q,p}}
\,.
\end{eqnarray}

Now we will prove the lemma
by induction on $N\geq1$.
If $N=1$, (\ref{lemmENFN}) becomes
\begin{eqnarray*}
0
& = &
Q_{1,1}(z_2)
\sum_{l=0}^{d_{1,1}}
c_{1,1,1,l}
w_1(z)^l
\,,\;
\forall\,z\in\C^2
\,,
\end{eqnarray*}
then
$Q_{1,1}=0$ (in this case, $F_1=E_1$) or
$\sum_{l=0}^{d_{1,1}}
c_{1,1,1,l}
w_1(z)^l=0$,
$\forall\,z\in\C^2$ (in this case,
$\sum_{l=0}^{d_{1,1}}
c_{1,1,1,l}
w^l=0$,
$\forall\,w\in\C$, then
$c_{1,1,1,l}=0$,
$\forall\,l=0,\ldots,d_{1,1}$ and
$F_1=E_1$).

Now if the lemma is proved for $N\geq1$ and
we consider $N+1$, then~(\ref{lemmENFN}) becomes,
for all $q=1,\ldots,N+1$,
\begin{eqnarray}\nonumber
0 & = &
(z_1-\eta_{N+1}z_2)
\sum_{p=1}^N
Q_{q,p}(z_2)
\prod_{j=p+1}^N
(z_1-\eta_jz_2)
\sum_{l=0}^{d_{q,p}}
c_{N+1,q,p,l}w_q(z)^{l+N-r_{q,p}}
\\\label{lemmENFN+1}
& &
+\;
Q_{q,N+1}(z_2)
\sum_{l=0}^{d_{q,N+1}}
c_{N+1,q,N+1,l}
w_q(z)^{l+N-r_{q,N+1}}
\,.
\end{eqnarray}
In particular,
$(z_1-\eta_{N+1}z_2)$ divides
$Q_{q,N+1}(z_2)
\sum_{l=0}^{d_{q,N+1}}
c_{N+1,q,N+1,l}
w_q(z)^{l+N-r_{q,N+1}}$. Then one (and only one) of these
different cases can happen:

\begin{itemize}

\item

$(z_1-\eta_{N+1}z_2)$ divides
$Q_{q,N+1}(z_2)$, then
$Q_{q,N+1}(z_2)=0$ and~(\ref{lemmENFN+1})
becomes
\begin{eqnarray*}
0 & = &
\sum_{p=1}^N
Q_{q,p}(z_2)
\prod_{j=p+1}^N
(z_1-\eta_jz_2)
\sum_{l=0}^{d_{q,p}}
c_{N+1,q,p,l}w_q(z)^{l+N-r_{q,p}}
\,,
\end{eqnarray*}
that yields by induction to
$F_{N+1}=E_{N+1}$.

\item

Otherwise (since $z_1-\eta_{N+1}z_1$
is irreducible and
$\C[z_1,z_2]$ is factorial),
$(z_1-\eta_{N+1}z_2)$ divides
$\sum_{l=0}^{d_{q,p}}
c_{N+1,q,p,l}w_q(z)^{l+N-r_{q,p}}$.
If $\sum_{l=0}^{d_{q,p}}
c_{N+1,q,p,l}X^{l}=0$,
then~(\ref{lemmENFN+1}) and the induction yield to
$F_{N+1}=E_{N+1}$.

\item

Else
$\sum_{l=0}^{d_{q,p}}
c_{N+1,q,p,l}w_q(z)^{l+N-r_{q,p}}
=w_q(z)^{m'}
\sum_{l=0}^{d'}
c'_lw_q(z)^l$, with $c'_0\neq0$,
and
$(z_1-\eta_{N+1}z_2)$ divides
$\sum_{l=0}^{d'}
c'_lw_q(z)^l$. Then
$\exists\,\widetilde{Q}\in\C[z_1,z_2]$ such that
\begin{eqnarray*}
\sum_{l=0}^{d'}
c'_lw_q(z)^l
& = &
(z_1-\eta_{N+1}z_2)
\widetilde{Q}(z)
\end{eqnarray*}
and $z=0$ yields to
$c'_0=0$, which is impossible.

\item

Finally, $(z_1-\eta_{N+1}z_2)$ must divide
$w_q(z)^{m'}$ then divide
$w_q(z)=\frac{\overline{\eta_q}z_1+z_2}{1+|\eta_q|^2}$
that is irreducible too.
It follows that they are proportional then
$1+\overline{\eta_q}\eta_{N+1}=0$.
In this case, the part
\begin{eqnarray*}
\frac{1+\eta_{N+1}\overline{\eta_q}}{1+|\eta_q|^2}
\frac{1}
{\prod_{j=N+1,j\neq q}^{N+1}(\eta_q-\eta_j)}
\sum_{m\geq0}
\left(
\frac{z_2+\overline{\eta_q}z_1}{1+|\eta_q|^2}
\right)^{m}
\frac{1}{m!}
\frac{\partial^m}{\partial v^m}|_{v=0}
[f(\eta_qv,v)]
\end{eqnarray*}
will disappear in the expression~(\ref{defEN})
of $E_N$. It follows that $E_N$ will be the most natural choice
for the interpolation formula and this achieves the induction.

\end{itemize}

\end{proof}

\section{Proof of Theorem \ref{equivbounded}}\label{proofequivbounded}

In this part we will assume that the set
$\{\eta_j\}_{j\geq1}$ is bounded, what we will
write as
\begin{eqnarray}
\|\eta\|_{\infty}
\;:=\;
\sup_{j\geq1}|\eta_j|
& < &
+\infty
\,.
\end{eqnarray}

\begin{remark}\label{criterbis}

We will see that the condition (\ref{criter})
is equivalent to the existence of
$R_{\eta}$ such that, for all
$p,q,s\geq0$ with $s\leq q$,
\begin{eqnarray}
\left|
\Delta_{p,(\eta_p,\ldots,\eta_1)}
\left(
\frac{\overline{\zeta}^s}{(1+|\zeta|^2)^q}
\right)
\left(\eta_{p+1}\right)
\right|
& \leq &
R_{\eta}^{p+q}
\,.
\end{eqnarray}

\end{remark}

In all the following, we will mean the
Taylor expansion of any function
$f\in\mathcal{O}\left(\C^2\right)$
(that absolutely converges in any compact
subset $K\subset\C^2$) by
\begin{eqnarray}
f(z) & = &
\sum_{k,l\geq0}a_{k,l}z_1^kz_2^l
\,.
\end{eqnarray}

\bigskip

\subsection{Condition~(\ref{criter}) is necessary}

We begin with this result.

\begin{lemma}\label{condnec1}

For all
$f\in\mathcal{O}\left(\mathbb{C}^2\right)$,
$N\geq1$ and $k_1\geq N$,
\begin{eqnarray*}
\frac{1}{k_1!}
\frac{\partial^{k_1}}{\partial z_1^{k_1}}|_{z=0}
[R_N(f;\eta)(z)]
& = &
\Delta_{N-1,(\eta_{N-1},\ldots,\eta_1)}
\left(
\left(
\frac{\overline{\zeta}}{1+|\zeta|^2}
\right)^{k_1-N+1}
\sum_{k+l=k_1}a_{k,l}\zeta^k
\right)(\eta_N)
\,.
\end{eqnarray*}

\end{lemma}

\begin{proof}

First, we claim that
\begin{eqnarray}\label{lagrangebisRN}
& &
R_N(f;\eta)(z)
=
\sum_{p=0}^{N-1}
z_2^{N-1-p}
\prod_{j=1}^p
(z_1-\eta_jz_2)
\Delta_{p,(\eta_p,\ldots,\eta_1)}
\left(
\zeta\mapsto
r_N(\zeta,z)
\right)
\left(\eta_{p+1}\right)
,
\end{eqnarray}
with
\begin{eqnarray}\label{rN}
r_N(\zeta,z) & := &
\sum_{k+l\geq N}
a_{k,l}\zeta^k
\left(
\frac{z_2+\overline{\zeta}z_1}{1+|\zeta|^2}
\right)^{k+l-N+1}
\\\nonumber
& = &
\sum_{m\geq N}
\left(
\frac{z_2+\overline{\zeta}z_1}{1+|\zeta|^2}
\right)^{m-N+1}
\sum_{k+l=m}a_{k,l}\zeta^k
\,.
\end{eqnarray}
Indeed, by Lemma~\ref{lagrangebis},
\begin{eqnarray*}
R_N(f;z)(z) & = &
z_2^{N-1}
\sum_{p=1}^N
\prod_{j=1,j\neq p}^N
\frac{z_1/z_2-\eta_j}{\eta_p-\eta_j}
\sum_{k+l\geq N}
a_{k,l}\eta_p^k
\left(
\frac{z_2+\overline{\eta_p}z_1}{1+|\eta_p|^2}
\right)^{k+l-N+1}
\\
& = &
z_2^{N-1}
\sum_{p=0}^{N-1}
\prod_{j=1}^p
(z_1/z_2-\eta_j)
\Delta_{p,(\eta_p,\ldots,\eta_1)}
\left[
\sum_{k+l\geq N}
a_{k,l}\zeta^k
\left(
\frac{z_2+\overline{\zeta}z_1}{1+|\zeta|^2}
\right)^{k+l-N+1}
\right]
\left(\eta_{p+1}\right)
.
\end{eqnarray*}
It follows that
\begin{eqnarray*}
R_N(f;\eta)(z)
& = &
\sum_{p=0}^{N-1}
z_2^{N-1-p}
\prod_{j=1}^p
(z_1-\eta_jz_2)
\Delta_{p,(\eta_p,\ldots,\eta_1)}
(\zeta\mapsto r_N(\zeta,z))
(\eta_{p+1})
\\
& = &
\sum_{p=0}^{N-1}
z_2^{N-1-p}
\sum_{r=0}^p
z_1^r(-1)^{p-r}z_2^{p-r}
\sigma_{p-r}(\eta_1,\ldots,\eta_p)
\Delta_{p}
(r_N(\zeta,z))
\\
& = &
\sum_{r=0}^{N-1}
z_1^rz_2^{N-1-r}
\sum_{p=r}^{N-1}
(-1)^{p-r}
\sigma_{p-r}(\eta_1,\ldots,\eta_p)
\Delta_{p}
(r_N(\zeta,z))
\,,
\end{eqnarray*}
with
$\sigma_r(\eta_1,\ldots,\eta_p)
=
\sum_{1\leq j_1<\cdots<j_r\leq p}
\eta_{j_1}\cdots\eta_{j_r}$.
Then
\begin{eqnarray*}
\frac{1}{k_1!}
\frac{\partial^{k_1}}{\partial z_1^{k_1}}|_{z=0}
[R_N(f;\eta)(z)]
& = &
\;\;\;\;\;\;\;\;\;\;\;\;\;\;\;\;\;\;\;\;\;\;\;\;\;\;\;\;\;\;\;\;\;\;\;\;\;\;\;
\;\;\;\;\;\;\;\;\;\;\;\;\;\;\;\;\;\;\;\;\;\;\;\;\;\;\;\;\;\;\;\;\;\;\;\;\;\;\;
\end{eqnarray*}
\begin{eqnarray*}
& = &
\sum_{r=0}^{N-1}
0^{N-1-r}
\sum_{p=r}^{N-1}
(-1)^{p-r}
\sigma_{p-r}(\eta_1,\ldots,\eta_p)
\frac{1}{k_1!}
\frac{\partial^{k_1}}{\partial z_1^{k_1}}|_{z=0}
\left[
z_1^r
\Delta_{p}
(r_N(\zeta,z))
\right]
\\
& = &
\frac{1}{k_1!}
\frac{\partial^{k_1}}{\partial z_1^{k_1}}|_{z=0}
\left[
z_1^{N-1}
\Delta_{N-1}
(r_N(\zeta,z))
\right]
\\
& = &
\Delta_{N-1,(\eta_{N-1},\ldots,\eta_1)}
\left(
\zeta\mapsto
\frac{1}{k_1!}
\frac{\partial^{k_1}}{\partial z_1^{k_1}}|_{z=0}
\left[
z_1^{N-1}r_N(\zeta,z)
\right]
\right)
(\eta_N)
\,.
\end{eqnarray*}
Since
$k_1\geq N$, one has for all
$\zeta\in\mathbb{C}$
\begin{eqnarray*}
\frac{1}{k_1!}
\frac{\partial^{k_1}}{\partial z_1^{k_1}}|_{z=0}
\left[
z_1^{N-1}r_N(\zeta,z)
\right]
& = &
\;\;\;\;\;\;\;\;\;\;\;\;\;\;\;\;\;\;\;\;\;\;\;\;\;\;\;\;\;\;\;\;\;\;\;\;\;\;\;\;\;
\;\;\;\;\;\;\;\;\;\;\;\;\;\;\;\;\;\;\;\;\;\;\;\;\;\;\;\;\;\;\;\;\;\;\;\;\;\;\;\;\;
\end{eqnarray*}
\begin{eqnarray*}
& = &
\sum_{s=0}^{k_1}
\frac{1}{s!}
\frac{\partial^{s}}{\partial z_1^{s}}|_{z_1=0}
\left(z_1^{N-1}\right)
\frac{1}{(k_1-s)!}
\frac{\partial^{k_1-s}}{\partial z_1^{k_1-s}}|_{z=0}
[r_N(\zeta,z)]
\\
& = &
\frac{1}{(k_1-N+1)!}
\frac{\partial^{k_1-N+1}}{\partial z_1^{k_1-N+1}}|_{z=0}
[r_N(\zeta,z)]
\\
& = &
\sum_{m\geq N}
\sum_{k+l=m}a_{k,l}\zeta^k
\frac{1}{(k_1-N+1)!}
\frac{\partial^{k_1-N+1}}{\partial z_1^{k_1-N+1}}|_{z=0}
\left[
\left(
\frac{z_2+\overline{\zeta}z_1}{1+|\zeta|^2}
\right)^{m-N+1}
\right]
\\
& = &
\sum_{k+l=k_1}a_{k,l}\zeta^k
\left(
\frac{\overline{\zeta}}{1+|\zeta|^2}
\right)^{k_1-N+1}
\,,
\end{eqnarray*}
and the proof is achieved.

\end{proof}

Now we can prove the first sense of Theorem~\ref{equivbounded}.

\begin{proof}

We assume that, for every
$f\in\mathcal{O}(\mathbb{C}^2)$,
$R_N(f;\eta)$ is uniformly bounded on any
compact subset $K\subset\mathbb{C}^2$, ie
\begin{eqnarray*}
\sup_{N\geq1}\;
\sup_{z\in K}
\left|
R_N(f;\eta)(z)
\right|
& \leq &
M(f,K)
\;<\;+\infty\,.
\end{eqnarray*}
In particular,
$\forall\,p\geq0$,
\begin{eqnarray*}
\sup_{z\in \overline{D_2(0,(1,1))}}
\left|
R_{p+1}(f;\eta)(z)
\right|
& \leq &
M(f)
\,.
\end{eqnarray*}
Then for all
$p\geq0$ and $q\geq1$,
\begin{eqnarray*}
\left|
\frac{1}{(p+q)!}
\frac{\partial^{p+q}}{\partial z_1^{p+q}}|_{z=0}
[R_{p+1}(f;\eta)(z)]
\right|
& = &
\left|
\frac{1}{(2i\pi)^2}
\int_{|\zeta_1|=|\zeta_2|=1}
\frac{R_{p+1}(f;\eta)(\zeta_1,\zeta_2)\,d\zeta_1\wedge d\zeta_2}
{\zeta_1^{p+q+1}\zeta_2}
\right|
\\
& \leq &
\sup_{z\in \overline{D_2(0,(1,1))}}
\left|
R_{p+1}(f;\eta)(z)
\right|
\;\leq\;
M(f)\,.
\end{eqnarray*}
Since $p+q\geq p+1$, one can deduce from the above lemma that
\begin{eqnarray*}
\left|
\Delta_{p,(\eta_p,\ldots,\eta_1)}
\left[
\left(
\frac{\overline{\zeta}}{1+|\zeta|^2}
\right)^q
\sum_{k+l=p+q}a_{k,l}\zeta^k
\right](\eta_{p+1})
\right|
& \leq &
M(f)
\,.
\end{eqnarray*}

\bigskip

Now consider any function entire on $\C$,
$h(w)=\sum_{n\geq0}a_nw^n$
and set
$f_h(z):=h(z_2)$. Then
$f_h\in\mathcal{O}
\left(\mathbb{C}^2\right)$ and
for all $p\geq0$, $q\geq1$,
\begin{eqnarray*}
\sum_{k+l=p+q}
a_{k,l}(f_h)\zeta^k
& = &
a_{0,p+q}(f_h)
\;=\;
a_{p+q}
\,.
\end{eqnarray*}
It follows that for all $p\geq0$, $q\geq1$,
\begin{eqnarray*}
\left|
\Delta_{p,(\eta_p,\ldots,\eta_1)}
\left[
\left(
\frac{\overline{\zeta}}{1+|\zeta|^2}
\right)^q
\sum_{k+l=p+q}a_{k,l}(f_h)\zeta^k
\right](\eta_{p+1})
\right|
& = &
\;\;\;\;\;\;\;\;\;\;\;\;\;\;\;\;\;\;\;\;\;\;\;\;\;\;\;\;\;\;
\end{eqnarray*}
\begin{eqnarray*}
& = &
|a_{p+q}|
\left|
\Delta_{p,(\eta_p,\ldots,\eta_1)}
\left[
\left(
\frac{\overline{\zeta}}{1+|\zeta|^2}
\right)^q
\right](\eta_{p+1})
\right|
\;\leq\;
M(h)\;:=M\;(f_h)
\,,
\end{eqnarray*}
then
\begin{eqnarray*}
\sup_{p\geq0,\,q\geq1}
\left\{
|a_{p+q}|^{\frac{1}{p+q}}
\left|
\Delta_{p,(\eta_p,\ldots,\eta_1)}
\left[
\left(
\frac{\overline{\zeta}}{1+|\zeta|^2}
\right)^q
\right](\eta_{p+1})
\right|^{\frac{1}{p+q}}
\right\}
& < &
+\infty
\,.
\end{eqnarray*}
Since $h\in\mathcal{O}(\mathbb{C})$,
$\limsup_{n\rightarrow\infty}
|a_n|^{1/n}=0$. Conversely, if
$(\varepsilon_n)_{n\geq1}$ is any sequence that
converges to $0$, the function
$h_{\varepsilon}(w):=
\sum_{n\geq1}\varepsilon_n^nw^n$
is entire on $\C$ and
\begin{eqnarray}\label{prop1}
\sup_{p\geq0,\,q\geq1}
\left\{
|\varepsilon_{p+q}|
\left|
\Delta_{p,(\eta_p,\ldots,\eta_1)}
\left[
\left(
\frac{\overline{\zeta}}{1+|\zeta|^2}
\right)^q
\right](\eta_{p+1})
\right|^{\frac{1}{p+q}}
\right\}
& < &
+\infty
\,.
\end{eqnarray}

Now we claim that
\begin{eqnarray}\label{Reta}
R_{\eta}
& := &
\sup_{p\geq0,\,q\geq1}
\left\{
\left|
\Delta_{p,(\eta_p,\ldots,\eta_1)}
\left[
\left(
\frac{\overline{\zeta}}{1+|\zeta|^2}
\right)^q
\right](\eta_{p+1})
\right|^{\frac{1}{p+q}}
\right\}
\;<\;
+\infty
\,.
\end{eqnarray}
If it is true, we will have proved the lemma for
$p\geq0,\,q\geq1$. On the other hand, since for all
$q=0,\,p\geq1$, one has
\begin{eqnarray*}
\left|
\Delta_{p,(\eta_p,\ldots,\eta_1)}
(\zeta\mapsto1)(\eta_{p+1})
\right|
& = & 0
\,,
\end{eqnarray*}
and
\begin{eqnarray*}
\left|
\Delta_{0}
(\zeta\mapsto1)(\eta_{1})
\right|
& = &
1
\;=\;
R_{\eta}^0
\,,
\end{eqnarray*}
the proof will be achieved.

Assume that (\ref{Reta}) is not true. For all
$n\geq1$, there exist
$p_n$ and $q_n$ such that
\begin{eqnarray*}
\left|
\Delta_{p_n,(\eta_{p_n},\ldots,\eta_1)}
\left[
\left(
\frac{\overline{\zeta}}{1+|\zeta|^2}
\right)^{q_n}
\right](\eta_{p_n+1})
\right|^{\frac{1}{p_n+q_n}}
& \geq &
n\,.
\end{eqnarray*}
One can choose
$(p_n)_{n\geq1}$ and $(q_n)_{n\geq1}$ such that
the sequence 
$(p_n+q_n)_{n\geq1}$ is strictly increasing. Indeed, let be
$p_1$ and $q_1$ for $n=1$ and assume that there are
$p_1,\ldots,p_n$ and $q_1,\ldots,q_n$ such that
$\forall\,j=1,\ldots,n-1$,
$p_j+q_j<p_{j+1}+q_{j+1}$.
Since
\begin{eqnarray*}
+\infty
& = &
\sup_{p\geq0,\,q\geq1}
\left\{
\left|
\Delta_{p}
\left[
\left(
\frac{\overline{\zeta}}{1+|\zeta|^2}
\right)^q
\right]
\right|^{\frac{1}{p+q}}
\right\}
\\
& = &
\max
\left\{
\sup_{p+q\leq p_n+q_n}
\left|
\Delta_{p}
\left[
\left(
\frac{\overline{\zeta}}{1+|\zeta|^2}
\right)^q
\right]
\right|^{\frac{1}{p+q}}
,
\sup_{p+q>p_n+q_n}
\left|
\Delta_{p}
\left[
\left(
\frac{\overline{\zeta}}{1+|\zeta|^2}
\right)^q
\right]
\right|^{\frac{1}{p+q}}
\right\}
\end{eqnarray*}
and the set
$\{p\geq0,\,q\geq1,\;p+q\leq p_n+q_n\}$ is finite,
it follows that
\begin{eqnarray*}
\sup_{p+q>p_n+q_n}
\left|
\Delta_{p}
\left[
\left(
\frac{\overline{\zeta}}{1+|\zeta|^2}
\right)^q
\right]
\right|^{\frac{1}{p+q}}
& = &
+\infty
\,.
\end{eqnarray*}
In particular, there are
$p_{n+1},\,q_{n+1}$ such that
$p_{n+1}+q_{n+1}>p_n+q_n$ and
\begin{eqnarray*}
\left|
\Delta_{p_{n+1},(\eta_{p_{n+1}},\ldots,\eta_1)}
\left[
\left(
\frac{\overline{\zeta}}{1+|\zeta|^2}
\right)^{q_{n+1}}
\right]
(\eta_{p_{n+1}+1})
\right|^{\frac{1}{p_{n+1}+q_{n+1}}}
& \geq &
n+1
\,.
\end{eqnarray*}
This allows us to construct by induction on $n\geq1$
the sequences
$(p_n)_{n\geq1}$ and $(q_n)_{n\geq1}$.

Now we define $(\varepsilon_n)_{n\geq1}$ by
\begin{eqnarray}
\varepsilon_n
& := &
\begin{cases}
1/\sqrt{j}\,,\mbox{ if }
\exists\,j\geq1,\,p_j+q_j=n\,,
\\
0\mbox{ otherwise}\,.
\end{cases}
\end{eqnarray}
Since $(p_n+q_n)_{n\geq1}$ is strictly increasing, if such a $j$
exists, it is unique then
$(\varepsilon_n)_{n\geq1}$ is well-defined.

On the other hand,
$\lim_{n\rightarrow\infty}\varepsilon_n=0$.
Indeed,
$\forall\,\varepsilon>0$, $\exists\,J\geq1$,
$\forall\,j\geq J$,
$1/\sqrt{j}\leq1/\sqrt{J}<\varepsilon$.
We set 
$N_{\varepsilon}:=p_J+q_J$ and let be any
$n\geq N_{\varepsilon}$.
If there is no $j$ such that $p_j+q_j=n$, then
$\varepsilon_n=0<\varepsilon$; otherwise
$\exists\,j\geq1$, $n=p_j+q_j$ and one has
$p_j+q_j=n\geq N_{\varepsilon}=p_J+q_J$. In particular,
$j\geq J$ (else
$n=p_j+q_j<p_J+q_J=N_{\varepsilon}$), then
$\varepsilon_n=1/\sqrt{j}\leq1/\sqrt{J}<\varepsilon$.

It follows that for all $j\geq1$,
\begin{eqnarray*}
\varepsilon_{p_j+q_j}
\left|
\Delta_{p_j,(\eta_{p_j},\ldots,\eta_1)}
\left[
\left(
\frac{\overline{\zeta}}{1+|\zeta|^2}
\right)^{q_j}
\right](\eta_{p_j+1})
\right|^{\frac{1}{p_j+q_j}}
& \geq &
\frac{1}{\sqrt{j}}\,j
\;=\;
\sqrt{j}
\,,
\end{eqnarray*}
then
\begin{eqnarray*}
\sup_{p\geq0,\,q\geq1}
\left\{
\varepsilon_{p+q}
\left|
\Delta_{p,(\eta_p,\ldots,\eta_1)}
\left[
\left(
\frac{\overline{\zeta}}{1+|\zeta|^2}
\right)^q
\right](\eta_{p+1})
\right|^{\frac{1}{p+q}}
\right\}
& \geq &
\sup_{j\geq1}
\left\{\sqrt{j}\right\}
\;=\;
+\infty
\,,
\end{eqnarray*}
which is in contradiction with~(\ref{prop1}),
thus~(\ref{Reta}) follows.

\end{proof}

\subsection{Condition (\ref{criter}) is sufficient}

In this part we will prove the inverse sense of Theorem~\ref{equivbounded}.
We assume that the (bounded) set
$\{\eta_j\}_{j\geq1}$ satisfies~(\ref{criter}).
We begin with the following result that
is a little stronger consequence
(see Remark~\ref{criterbis}).

\begin{lemma}\label{criter+}

There is $R'_{\eta}\geq1$ such that
for all $p,q,s\geq0$ with
$0\leq s\leq q$, one has
\begin{eqnarray*}
\left|
\Delta_{p,(\eta_p,\ldots,\eta_1)}
\left[
\zeta\mapsto
\frac{\overline{\zeta}^s}
{\left(1+|\zeta|^2\right)^q}
\right](\eta_{p+1})
\right|
& \leq &
{R'_{\eta}}^{p+q}
\,.
\end{eqnarray*}

\end{lemma}

\begin{proof}

Set
\begin{eqnarray}\label{defRQS}
\begin{cases}
R\;=\;
\max\left(1,R_{\eta}\right)
\,,
\\
Q\;=\;
\max(3,R_{\eta})
\,,
\\
S\;=\;
3\max(1,\|\eta\|_{\infty})
\,.
\end{cases}
\end{eqnarray}
In order to prove the lemma, we want to prove
the following estimation:
$\forall\,p,q,s\geq0$ with
$s\leq q$, one has
\begin{eqnarray}\label{proofcriterbis}
\left|
\Delta_{p,(\eta_p,\ldots,\eta_1)}
\left[
\zeta\mapsto
\frac{\overline{\zeta}^s}
{\left(1+|\zeta|^2\right)^q}
\right](\eta_{p+1})
\right|
& \leq &
R^pQ^qS^{q-s}
\,.
\end{eqnarray}
This will be proved by induction on
$p+q-s\geq0$.
If
$p+q-s=0$ then since $p,q-s\geq0$, necessarily
$p=0$ and $s=q\geq0$. Thus
\begin{eqnarray*}
\left|
\Delta_0
\left(
\frac{\overline{\zeta}^q}{\left(1+|\zeta|^2\right)^q}
\right)(\eta_1)
\right|
& = &
\left|
\frac{\overline{\eta_1}^q}
{\left(1+|\eta_1|^2\right)^q}
\right|
\;=\;
\left(
\frac{0.1+1.|\eta_1|}
{1+|\eta_1|^2}
\right)^q
\\
& \leq &
\left(
\frac{1.\sqrt{1+|\eta_1|^2}}
{1+|\eta_1|^2}
\right)^q
\mbox{ (Cauchy-Schwarz inequality)}
\\
& = &
\frac{1}
{\left(\sqrt{1+|\eta_1|^2}\right)^q}
\;\leq\;1
\;\leq\;
R^0Q^qS^0
\,.
\end{eqnarray*}

If $p+q-s=1$, then
$p=1$ and $q=s\geq0$, or
$p=0$ and $0\leq s=q-1$.
In the first case, since
(\ref{criter})
is satisfied, one has
\begin{eqnarray*}
\left|
\Delta_{1,\eta_1}
\left[
\left(
\frac{\overline{\zeta}}{1+|\zeta|^2}
\right)^q
\right](\eta_2)
\right|
& \leq &
R_{\eta}^{1+q}
\;\leq\;
R^1Q^qS^0
\,.
\end{eqnarray*}
In the second case, one has for all
$q\geq1$
\begin{eqnarray*}
\left|
\Delta_0
\left(
\frac{\overline{\zeta}^{q-1}}
{\left(1+|\zeta|^2\right)^q}
\right)(\eta_1)
\right|
\;=\;
\left|
\frac{\overline{\eta_1}^{q-1}}
{\left(1+|\eta_1|^2\right)^q}
\right|
\;=\;
\frac{1}{1+|\eta_1|^2}
\left|
\frac{\overline{\eta_1}}{1+|\eta_1|^2}
\right|^{q-1}
\;\leq\;
1
\;\leq\;
R^0Q^qS^1
\,.
\end{eqnarray*}

Now let be $m\geq1$ and assume that
(\ref{proofcriterbis}) is true for all
$p,q,s\geq0$ with $s\leq q$ and such that
$p+q-s\leq m$. 
Now let be $p,q,s\geq0$ with $s\leq q$ and such that
$p+q-s=m+1$. One has different cases.

If $p=0$ then
\begin{eqnarray*}
\left|
\Delta_0
\left(
\frac{\overline{\zeta}^{s}}
{\left(1+|\zeta|^2\right)^q}
\right)(\eta_1)
\right|
& = &
\left|
\frac{\overline{\eta_1}^{s}}
{\left(1+|\eta_1|^2\right)^q}
\right|
\;=\;
\left|
\frac{\overline{\eta_1}}
{\left(1+|\eta_1|^2\right)}
\right|^s
\frac{1}
{\left(1+|\eta_1|^2\right)^{q-s}}
\\
& \leq &
1
\;\leq\;
R^0Q^{q}S^{q-s}
\,.
\end{eqnarray*}

If $s=q$ then by (\ref{criter})
\begin{eqnarray*}
\left|
\Delta_{p,(\eta_p,\ldots,\eta_1)}
\left[
\zeta\mapsto
\frac{\overline{\zeta}^q}
{\left(1+|\zeta|^2\right)^q}
\right](\eta_{p+1})
\right|
& \leq &
R_{\eta}^{p+q}
\;\leq\;
R^pQ^qS^{0}
\,.
\end{eqnarray*}

Otherwise
$p\geq1$ and $0\leq s\leq q-1$ (in particular $q\geq1$).
On one hand, one has by Lemmas~\ref{Deltaprod}
and~\ref{Deltanal}
\begin{eqnarray*}
\Delta_{p,(\eta_p,\ldots,\eta_1)}
\left(
\frac{\overline{\zeta}^{s+1}\zeta}
{(1+|\zeta|^2)^q}
\right)(\eta_{p+1})
& = &
\;\;\;\;\;\;\;\;\;\;\;\;\;\;\;\;\;\;\;\;\;\;\;\;\;\;\;\;\;\;\;\;\;\;\;\;\;\;\;\;\;\;
\;\;\;\;\;\;\;\;\;\;\;\;\;\;\;\;\;\;\;\;\;\;\;\;\;\;\;\;\;\;\;\;\;\;\;\;\;\;\;\;
\end{eqnarray*}
\begin{eqnarray*}
& = &
\sum_{r=0}^p
\Delta_{r,(\eta_r,\ldots,\eta_1)}
\left(
\frac{\overline{\zeta}^{s+1}}
{(1+|\zeta|^2)^q}
\right)(\eta_{r+1})
\Delta_{p-r,(\eta_p,\ldots,\eta_{r+1})}
\left(
\zeta\mapsto\zeta
\right)(\eta_{p+1})
\\
& = &
\Delta_{p-1,(\eta_{p-1},\ldots,\eta_1)}
\left(
\frac{\overline{\zeta}^{s+1}}
{(1+|\zeta|^2)^q}
\right)(\eta_{p})
\times1
+
\Delta_{p,(\eta_p,\ldots,\eta_1)}
\left(
\frac{\overline{\zeta}^{s+1}}
{(1+|\zeta|^2)^q}
\right)(\eta_{p+1})
\times\eta_{p+1}
\,.
\end{eqnarray*}
On the other hand,
\begin{eqnarray*}
\Delta_{p,(\eta_p,\ldots,\eta_1)}
\left(
\frac{\overline{\zeta}^{s+1}\zeta}
{(1+|\zeta|^2)^q}
\right)(\eta_{p+1})
& = &
\;\;\;\;\;\;\;\;\;\;\;\;\;\;\;\;\;\;\;\;\;\;\;\;\;\;\;\;\;\;\;\;\;\;\;\;\;\;\;\;\;\;
\;\;\;\;\;\;\;\;\;\;\;\;\;\;\;\;\;\;\;\;\;\;\;\;\;\;\;\;\;\;\;\;\;\;\;\;\;\;\;\;
\end{eqnarray*}
\begin{eqnarray*}
& = &
\Delta_{p,(\eta_p,\ldots,\eta_1)}
\left(
\frac{\overline{\zeta}^{s}(|\zeta|^2+1-1)}
{(1+|\zeta|^2)^q}
\right)(\eta_{p+1})
\\
& = &
\Delta_{p,(\eta_p,\ldots,\eta_1)}
\left(
\frac{\overline{\zeta}^{s}}
{(1+|\zeta|^2)^{q-1}}
\right)(\eta_{p+1})
-
\Delta_{p,(\eta_p,\ldots,\eta_1)}
\left(
\frac{\overline{\zeta}^{s}}
{(1+|\zeta|^2)^{q}}
\right)(\eta_{p+1})
\,.
\end{eqnarray*}
Thus
\begin{eqnarray*}
\Delta_{p,(\eta_p,\ldots,\eta_1)}
\left(
\frac{\overline{\zeta}^{s}}
{(1+|\zeta|^2)^{q}}
\right)(\eta_{p+1})
& = &
\;\;\;\;\;\;\;\;\;\;\;\;\;\;\;\;\;\;\;\;\;\;\;\;\;\;\;\;\;\;\;\;\;\;
\;\;\;\;\;\;\;\;\;\;\;\;\;\;\;\;\;\;\;\;\;\;\;\;\;\;\;\;\;\;\;\;\;\;
\end{eqnarray*}
\begin{eqnarray*}
& = &
\Delta_{p,(\eta_p,\ldots,\eta_1)}
\left(
\frac{\overline{\zeta}^{s}}
{(1+|\zeta|^2)^{q-1}}
\right)(\eta_{p+1})
\\
& &
-
\Delta_{p-1,(\eta_{p-1},\ldots,\eta_1)}
\left(
\frac{\overline{\zeta}^{s+1}}
{(1+|\zeta|^2)^q}
\right)(\eta_{p})
-
\eta_{p+1}
\Delta_{p,(\eta_p,\ldots,\eta_1)}
\left(
\frac{\overline{\zeta}^{s+1}}
{(1+|\zeta|^2)^q}
\right)(\eta_{p+1})
\,.
\end{eqnarray*}
Since
$s\leq q-1$, $s+1\leq q$ and 
$(p-1)+q-(s+1)\leq p+(q-1)-s=p+q-(s+1)=m$,
by induction and~(\ref{defRQS}) it follows that
\begin{eqnarray*}
\left|
\Delta_{p,(\eta_p,\ldots,\eta_1)}
\left(
\frac{\overline{\zeta}^{s}}
{(1+|\zeta|^2)^{q}}
\right)(\eta_{p+1})
\right|
& \leq &
\;\;\;\;\;\;\;\;\;\;\;\;\;\;\;\;\;\;\;\;\;\;\;\;\;\;\;\;\;\;\;\;
\;\;\;\;\;\;\;\;\;\;\;\;\;\;\;\;\;\;\;\;\;\;\;\;\;\;\;\;\;\;\;\;\;
\end{eqnarray*}
\begin{eqnarray*}
& \leq &
R^pQ^{q-1}S^{q-s}
+R^{p-1}Q^qS^{q-s-1}
+\|\eta\|_{\infty}
R^pQ^qS^{q-s-1}
\\
& \leq &
R^{p-1}Q^{q-1}S^{q-s-1}
\left(
RS+Q+\|\eta\|_{\infty}RQ
\right)
\\
& \leq &
R^{p-1}Q^{q-1}S^{q-s-1}
\left(
RS\frac{Q}{3}
+
Q\frac{RS}{3}
+
\frac{S}{3}
RQ
\right)
\;=\;
R^pQ^qS^{q-s}
\,,
\end{eqnarray*}
and this proves~(\ref{proofcriterbis}).

Finally, if we set
\begin{eqnarray}\label{defReta'}
R'_{\eta}
& := &
\left[
\max(R,Q,S)
\right]^2
\;=\;
\left[
\max\left(
3,3\|\eta\|_{\infty},R_{\eta}
\right)
\right]^2
\,,
\end{eqnarray}
(\ref{proofcriterbis}) yields to,
for all $p,q,s\geq0$ with $s\leq q$,
\begin{eqnarray*}
\left|
\Delta_{p,(\eta_p,\ldots,\eta_1)}
\left(
\frac{\overline{\zeta}^{s}}
{(1+|\zeta|^2)^{q}}
\right)(\eta_{p+1})
\right|
& \leq &
R^pQ^qS^{q}
\;\leq\;
{R'_{\eta}}^{p+q}
\,,
\end{eqnarray*}
and the proof is achieved.

\end{proof}

In the following the constant
$R_{\eta}$ will mean 
$R'_{\eta}$ from Lemma~\ref{criter+}.
One can deduce as a consequence the following result.

\begin{lemma}\label{aux}

For all $p,q\geq0$ and $z\in\C^2$,
\begin{eqnarray*}
\left|
\Delta_{p,(\eta_p,\ldots,\eta_1)}
\left[
\zeta\mapsto
\left(
\frac{z_2+\overline{\zeta}z_1}
{1+|\zeta|^2}
\right)^q
\right](\eta_{p+1})
\right|
& \leq &
R_{\eta}^p(2R_{\eta}\|z\|)^q
\,.
\end{eqnarray*}

\end{lemma}

\bigskip

\begin{proof}

Indeed,
Lemma~\ref{criter+} yields to
\begin{eqnarray*}
\left|
\Delta_{p,(\eta_p,\ldots,\eta_1)}
\left[
\left(
\frac{z_2+\overline{\zeta}z_1}
{1+|\zeta|^2}
\right)^q
\right](\eta_{p+1})
\right|
& = &
\;\;\;\;\;\;\;\;\;\;\;\;\;\;\;\;\;\;\;\;\;\;\;\;\;\;\;\;\;\;\;
\;\;\;\;\;\;\;\;\;\;\;\;\;\;\;\;\;\;\;\;\;\;\;\;\;\;\;\;\;\;\;
\end{eqnarray*}
\begin{eqnarray*}
& = &
\left|
\sum_{u=0}^q
\frac{q!}{u!\,(q-u)!}
z_2^{q-u}
z_1^{u}
\Delta_{p,(\eta_p,\ldots,\eta_1)}
\left(
\frac{\overline{\zeta}^{u}}
{(1+|\zeta|^2)^q}
\right)(\eta_{p+1})
\right|
\\
& \leq &
\sum_{u=0}^q
\frac{q!}{u!\,(q-u)!}
\|z\|^q
\left|
\Delta_{p,(\eta_p,\ldots,\eta_1)}
\left(
\frac{\overline{\zeta}^{u}}
{(1+|\zeta|^2)^q}
\right)(\eta_{p+1})
\right|
\\
& \leq &
\|z\|^q
\sum_{u=0}^q
\frac{q!}{u!\,(q-u)!}
R_{\eta}^{p+q}
\;=\;
\|z\|^q
2^q
R_{\eta}^{p+q}
\;=\;
R_{\eta}^p(2R_{\eta}\|z\|)^q
\,.
\end{eqnarray*}

\end{proof}

The next result will be usefull in order to prove the reciprocal
sense of Theorem~\ref{equivbounded}.

\begin{lemma}\label{combi}

For all $n,p\geq0$,
\begin{eqnarray*}
A_n^p & := &
\sum_{l_1=0}^{n}
\sum_{l_2=0}^{l_1}
\cdots
\sum_{l_p=0}^{l_{p-1}}
1
\;=\;
card
\left\{
(l_1,\ldots,l_p)\in\mathbb{N}^p,\,
n\geq l_1\geq l_2\geq\cdots\geq l_p\geq0
\right\}
\\
& = &
\frac{(n+p)!}{n!\,p!}
\,.
\end{eqnarray*}

\end{lemma}

\begin{proof}

First, we admit that if $p=0$ then $l_0=n$ and
\begin{eqnarray*}
A_n^0
& = &
\prod_{j=1}^0
\sum_{l_j=0}^{l_{j-1}}
1
\;=\;
1
\,.
\end{eqnarray*}
So one can assume that $p\geq1$ and we
prove this result by induction on
$n+p\geq1$.
If $n+p=1$ then $p=1$ and $n=0$ then
\begin{eqnarray*}
A_0^1
\;=\;
card\left\{
s_1\in\mathbb{N},\;
0\geq s_1\geq0
\right\}
\;=\;
1
\,.
\end{eqnarray*}
Now we assume that it is true for all
$p,\,n$ such that
$1\leq p+n\leq m$ with $p\geq1$ and $n\geq0$,
and let be $n\geq0$, $p\geq1$ such that $n+p=m+1$.

If $n=0$ then
\begin{eqnarray*}
card
\left\{
0\geq s_1\geq\cdots\geq s_{p}\geq0
\right\}
& = &
card\{(0,\ldots,0)\}
\;=\;
1
\,.
\end{eqnarray*}

If $p=1$ then
\begin{eqnarray*}
card\left\{n\geq s_1\geq0\right\}
& = &
\frac{(n+1)!}{n!\,1!}
\,.
\end{eqnarray*}

Otherwise $n\geq1$ and $p\geq2$. We claim that
\begin{eqnarray}\label{pascal}
A_n^p 
& = &
A_{n-1}^p\;+\;A_n^{p-1}
\,.
\end{eqnarray}
Indeed, for any element of
$\left\{
(s_1,\ldots,s_p)\in\mathbb{N}^p,\;
n\geq s_1\geq\cdots\geq s_p\geq0
\right\}$,
either $s_1=n$ or $s_1\leq n-1$. Then
\begin{eqnarray*}
card
\left\{
(s_1,\ldots,s_p)\in\mathbb{N}^p,\;
n\geq s_1\geq\cdots\geq s_p\geq0
\right\}
& = &
\;\;\;\;\;\;\;\;\;\;\;\;\;\;\;\;\;\;\;\;\;\;\;\;\;\;\;\;\;\;\;\;\;\;\;\;
\end{eqnarray*}
\begin{eqnarray*}
& = &
card
\left\{
(s_2,\ldots,s_p)\in\mathbb{N}^{p-1},\;
n\geq s_2\geq\cdots\geq s_p\geq0
\right\}
\\
& &
+\;
card
\left\{
(s_1,\ldots,s_p)\in\mathbb{N}^p,\;
n-1\geq s_1\geq\cdots\geq s_p\geq0
\right\}
\\
& = &
A_n^{p-1}+A_{n-1}^p\,.
\end{eqnarray*}
Since
$n+p-1=n-1+p=m$, one has by induction
\begin{eqnarray*}
A_n^p
& = &
\frac{(n+p-1)!}{n!\,(p-1)!}
+\frac{(n-1+p)!}{(n-1)!\,p!}
\;=\;
\frac{(n+p-1)!}{n!\,p!}(p+n)
\end{eqnarray*}
and the lemma is proved.

\end{proof}

Now Lemmas~\ref{aux} and~\ref{combi} yield to
the following result.

\begin{lemma}\label{aux1}

Let be
$f\in\mathcal{O}\left(\mathbb{C}^2\right)$.
For all $p\geq0$, $N\geq1$,
$z\in\mathbb{C}^2$
and
$R>2\|\eta\|_{\infty}R_{\eta}^2\|z\|$,
one has
\begin{eqnarray*}
\left|
\Delta_{p,(\eta_p,\ldots,\eta_1)}
\left[
\zeta\mapsto
\sum_{m\geq N}
\left(
\frac{z_2+\overline{\zeta}z_1}
{1+|\zeta|^2}
\right)^{m-N+1}
\sum_{k+l=m}a_{k,l}\zeta^k
\right](\eta_{p+1})
\right|
& \leq &
\;\;\;\;\;\;\;\;\;\;\;\;\;\;\;\;\;\;
\end{eqnarray*}
\begin{eqnarray*}
\;\;\;\;\;\;\;\;\;\;\;\;\;\;\;\;\;\;\;\;\;\;\;\;\;\;\;\;\;\;\;\;\;\;\;\;
& \leq &
\frac{8\|f\|_RR_{\eta}^2\|z\|}
{\|\eta\|_{\infty}(1-2\|\eta\|_{\infty}R_{\eta}^2\|z\|/R)}
\left(
\frac{\|\eta\|_{\infty}R_{\eta}}{R}
\right)^N
\left(\frac{R_{\eta}(1+\|\eta\|_{\infty})}{\|\eta\|_{\infty}}\right)^p
\,,
\end{eqnarray*}
with
\begin{eqnarray}
\|f\|_R
& := &
\sup_{|z_1|,|z_2|\leq R}|f(z_1,z_2)|
\,.
\end{eqnarray}

\end{lemma}

\begin{proof}

First, for all $m\geq N$, one has
\begin{eqnarray*}
\Delta_{p,(\eta_p,\ldots,\eta_1)}
\left[
\left(
\frac{z_2+\overline{\zeta}z_1}
{1+|\zeta|^2}
\right)^{m-N+1}
\sum_{k+l=m}a_{k,l}\zeta^k
\right](\eta_{p+1})
& = &
\;\;\;\;\;\;\;\;\;\;\;\;\;\;\;\;\;\;\;\;\;\;\;\;\;\;\;\;\;\;\;
\;\;\;\;\;\;
\end{eqnarray*}
\begin{eqnarray*}
& = &
\sum_{v=0}^{p}
\Delta_{v,(\eta_v,\ldots,\eta_1)}
\left[
\left(
\frac{z_2+\overline{\zeta}z_1}
{1+|\zeta|^2}
\right)^{m-N+1}
\right](\eta_{v+1})
\Delta_{p-v,(\eta_p,\ldots,\eta_{v+1})}
\left(
\sum_{k+l=m}a_{k,l}\zeta^k
\right)
(\eta_{p+1})
\,.
\end{eqnarray*}

Next, for all
$0\leq v\leq p$ and $m\geq N$ ($>p$),
on has by Lemmas~\ref{Deltanal}
and~\ref{combi}
\begin{eqnarray*}
\left|
\Delta_{p-v,(\eta_p,\ldots,\eta_{v+1})}
\left(
\sum_{k=0}^{m}a_{k,m-k}\zeta^k
\right)
(\eta_{p+1})
\right|
& \leq &
\;\;\;\;\;\;\;\;\;\;\;\;\;\;\;\;\;\;\;\;\;\;\;\;\;\;\;\;\;\;\;
\;\;\;\;\;\;\;\;\;\;\;\;\;\;\;\;\;\;\;\;\;\;\;\;\;\;\;\;\;\;\;
\end{eqnarray*}
\begin{eqnarray*}
& \leq &
\sum_{k=p-v}^{m}
|a_{k,m-k}|
\sum_{l_1=0}^{k-p+v}|\eta_{v+1}|^{k-p+v-l_1}
\sum_{l_2=0}^{l_1}
|\eta_{v+2}|^{l_1-l_2}
\cdots
\sum_{l_{p-v}=0}^{l_{p-v-1}}
|\eta_{p}|^{l_{p-v-1}-l_{p-v}}
|\eta_{p+1}|^{l_{p-v}}
\\
& \leq &
\sum_{k=p-v}^{m}
|a_{k,m-k}|
\|\eta\|_{\infty}^{k-p+v}
\sum_{l_1=0}^{k-p+v}
\cdots
\sum_{l_{p-v}=0}^{l_{p-v-1}}
1
=\sum_{k=p-v}^{m}
|a_{k,m-k}|
\|\eta\|_{\infty}^{k-p+v}
\frac{k!}{(p-v)!\,(k-p+v)!}
\,.
\end{eqnarray*}
On the other hand, since
$f\in\mathcal{O}\left(\mathbb{C}^2\right)$, one has
for all $R\geq1$
\begin{eqnarray*}
\left|a_{k,l}\right|
& = &
\left|
\frac{1}{(2i\pi)^2}
\int_{|\zeta_1|=|\zeta_2|=R}
\frac{f(\zeta_1,\zeta_2)\,d\zeta_1\wedge d\zeta_2}
{\zeta_1^{k+1}\zeta_2^{l+1}}
\right|
\;\leq\;
\frac{\|f\|_R}{R^{k+l}}
\,.
\end{eqnarray*}
Thus
\begin{eqnarray*}
\left|
\Delta_{p-v,(\eta_p,\ldots,\eta_{v+1})}
\left(
\sum_{k=0}^{m}a_{k,m-k}\zeta^k
\right)
(\eta_{p+1})
\right|
& \leq &
\;\;\;\;\;\;\;\;\;\;\;\;\;\;\;\;\;\;\;\;\;\;\;\;\;\;\;\;\;\;\;
\;\;\;\;\;\;\;\;\;\;\;\;\;\;\;\;\;\;\;\;\;\;\;\;\;\;\;\;\;\;\;
\end{eqnarray*}
\begin{eqnarray*}
& \leq &
\frac{\|f\|_R}{R^m}
\sum_{k=p-v}^{m}
\frac{k!}{(p-v)!\,(k-p+v)!}
\|\eta\|_{\infty}^{k-p+v}
\;=\;
\frac{\|f\|_R}{R^m}
\frac{1}{(p-v)!}
\frac{\partial^{p-v}}{\partial t^{p-v}}|_{t=\|\eta\|_{\infty}}
\left[
\sum_{k=0}^{m}
t^k
\right]
\\
& = &
\frac{\|f\|_R}{R^m}
\frac{1}{2i\pi}
\int_{|t|=R_{\eta}}
\sum_{k=0}^{m}
t^k
\frac{dt}{(t-\|\eta\|_{\infty})^{p-v+1}}
\\
& \leq &
\frac{\|f\|_R}{R^m}
\frac{R_{\eta}\sum_{k=0}^{m}R_{\eta}^k}{(R_{\eta}-\|\eta\|_{\infty})^{p-v+1}}
\;=\;
\frac{\|f\|_R}{R^m}
\frac{R_{\eta}^{m+1}-1}
{(R_{\eta}-\|\eta\|_{\infty})^{p-v+1}(1-1/R_{\eta})}
\\
& \leq &
\frac{2\|f\|_R}{R^m}
\frac{R_{\eta}^{m+1}}
{\|\eta\|_{\infty}^{p-v+1}}
\;=\;
2\|f\|_R
\frac{R_{\eta}}{\|\eta\|_{\infty}}
\left(
\frac{R_{\eta}}{R}
\right)^m
\frac{1}{\|\eta\|_{\infty}^{p-v}}
\,,
\end{eqnarray*}
since by~(\ref{defReta'}),
$R_{\eta}\geq2\|\eta\|_{\infty},2$.

It follows with Lemma~\ref{aux} that
\begin{eqnarray*}
\Delta_{p,(\eta_p,\ldots,\eta_1)}
\left[
\left(
\frac{z_2+\overline{\zeta}z_1}
{1+|\zeta|^2}
\right)^{m-N+1}
\sum_{k+l=m}a_{k,l}\zeta^k
\right](\eta_{p+1})
& \leq &
\;\;\;\;\;\;\;\;\;\;\;\;\;\;\;\;\;\;
\end{eqnarray*}
\begin{eqnarray*}
\;\;\;\;\;\;\;\;\;\;\;\;
& \leq &
\sum_{v=0}^p
R_{\eta}^v
(2R_{\eta}\|z\|)^{m-N+1}
\frac{2\|f\|_RR_{\eta}}{\|\eta\|_{\infty}}
\left(
\frac{R_{\eta}}{R}
\right)^m
\frac{1}{\|\eta\|_{\infty}^{p-v}}
\\
& = &
\frac{2\|f\|_RR_{\eta}}{\|\eta\|_{\infty}}
\left(
\frac{R_{\eta}}{R}
\right)^m
(2R_{\eta}\|z\|)^{m-N+1}
\frac{1}{\|\eta\|_{\infty}^p}
\sum_{v=0}^p
\left(R_{\eta}\|\eta\|_{\infty}\right)^v
\\
& \leq &
\frac{2\|f\|_RR_{\eta}}
{\|\eta\|_{\infty}(2R_{\eta}\|z\|)^{N-1}}
\left(
\frac{2\|\eta\|_{\infty}R_{\eta}^2\|z\|}{R}
\right)^m
\frac{1}{\|\eta\|_{\infty}^p}
\frac{(R_{\eta}(1+\|\eta\|_{\infty}))^{p+1}-1}
{R_{\eta}(1+\|\eta\|_{\infty})-1}
\\
& \leq &
\frac{4\|f\|_RR_{\eta}}
{\|\eta\|_{\infty}(2R_{\eta}\|z\|)^{N-1}}
\left(
\frac{2\|\eta\|_{\infty}R_{\eta}^2\|z\|}{R}
\right)^m
\left(\frac{R_{\eta}(1+\|\eta\|_{\infty})}{\|\eta\|_{\infty}}\right)^p
\,.
\end{eqnarray*}
One can deduce that, for all
$N\geq1$, $0\leq p\leq N-1$,
$z\in\mathbb{C}^2$ and
$R>2\|\eta\|_{\infty}R_{\eta}^2\|z\|$,
\begin{eqnarray*}
\left|
\Delta_{p,(\eta_p,\ldots,\eta_1)}
\left[
\sum_{m\geq N}
\left(
\frac{z_2+\overline{\zeta}z_1}
{1+|\zeta|^2}
\right)^{m-N+1}
\sum_{k+l=m}a_{k,l}\zeta^k
\right](\eta_{p+1})
\right|
& \leq &
\;\;\;\;\;\;\;\;\;\;\;\;\;\;\;\;\;\;
\end{eqnarray*}
\begin{eqnarray*}
& \leq &
\sum_{m\geq N}
\left|
\Delta_{p,(\eta_p,\ldots,\eta_1)}
\left[
\left(
\frac{z_2+\overline{\zeta}z_1}
{1+|\zeta|^2}
\right)^{m-N+1}
\sum_{k+l=m}a_{k,l}\zeta^k
\right](\eta_{p+1})
\right|
\\
& \leq &
\frac{4\|f\|_RR_{\eta}}
{\|\eta\|_{\infty}(2R_{\eta}\|z\|)^{N-1}}
\left(\frac{R_{\eta}(1+\|\eta\|_{\infty})}{\|\eta\|_{\infty}}\right)^p
\sum_{m\geq N}
\left(
\frac{2\|\eta\|_{\infty}R_{\eta}^2\|z\|}{R}
\right)^m
\\
& \leq &
\frac{4\|f\|_RR_{\eta}}
{\|\eta\|_{\infty}(2R_{\eta}\|z\|)^{N-1}}
\left(\frac{R_{\eta}(1+\|\eta\|_{\infty})}{\|\eta\|_{\infty}}\right)^p
\frac{\left(
2\|\eta\|_{\infty}R_{\eta}^2\|z\|/R
\right)^N}
{1-2\|\eta\|_{\infty}R_{\eta}^2\|z\|/R}
\\
& = &
\frac{8\|f\|_RR_{\eta}^2\|z\|}
{\|\eta\|_{\infty}(1-2\|\eta\|_{\infty}R_{\eta}^2\|z\|/R)}
\left(
\frac{\|\eta\|_{\infty}R_{\eta}}{R}
\right)^N
\left(\frac{R_{\eta}(1+\|\eta\|_{\infty})}{\|\eta\|_{\infty}}\right)^p
\,.
\end{eqnarray*}

\end{proof}

We can finally complete the proof of Theorem~\ref{equivbounded}.

\begin{proof}

$f\in\mathcal{O}\left(\C^2\right)$
and $K\subset\C^2$ compact subset being given,
it follows from Lemma~\ref{aux1} that,
for all
$N\geq1$, $z\in K$
and
$R>2\|\eta\|_{\infty}R_{\eta}^2\|z\|$,
\begin{eqnarray*}
|R_N(f;\eta)(z)|
& \leq &
\;\;\;\;\;\;\;\;\;\;\;\;\;\;\;\;\;\;\;\;\;\;\;\;\;\;\;\;\;\;\;\;\;\;\;\;
\;\;\;\;\;\;\;\;\;\;\;\;\;\;\;\;\;\;\;\;\;\;\;\;\;\;\;\;\;\;\;\;\;\;\;\;
\;\;\;\;\;\;\;\;\;\;\;\;\;\;\;\;\;\;\;\;\;\;\;\;
\end{eqnarray*}
\begin{eqnarray*}
& \leq &
\sum_{p=0}^{N-1}
|z_2|^{N-1-p}
\prod_{j=1}^p
|z_1-\eta_jz_2|
\left|
\Delta_{p,(\eta_p,\ldots,\eta_1)}
\left[
\sum_{k+l\geq N}
a_{k,l}\eta_p^k
\left(
\frac{z_2+\overline{\eta_p}z_1}{1+|\eta_p|^2}
\right)^{k+l-N+1}
\right](\eta_{p+1})
\right|
\\
& \leq &
\sum_{p=0}^{N-1}
\|z\|^{N-1-p}
\|z\|^p
\prod_{j=1}^p\sqrt{1+|\eta_j|^2}
\frac{8\|f\|_RR_{\eta}^2\|z\|}
{\|\eta\|_{\infty}(1-2\|\eta\|_{\infty}R_{\eta}^2\|z\|/R)}
\left(
\frac{\|\eta\|_{\infty}R_{\eta}}{R}
\right)^N
\left(\frac{R_{\eta}(1+\|\eta\|_{\infty})}{\|\eta\|_{\infty}}\right)^p
\\
& \leq &
\frac{8\|f\|_RR_{\eta}^2}
{\|\eta\|_{\infty}(1-2\|\eta\|_{\infty}R_{\eta}^2\|z\|/R)}
\left(
\frac{\|\eta\|_{\infty}R_{\eta}\|z\|}{R}
\right)^N
\sum_{p=0}^{N-1}
\left(\frac{R_{\eta}(1+\|\eta\|_{\infty})^2}{\|\eta\|_{\infty}}\right)^p
\\
& = &
\frac{8\|f\|_RR_{\eta}^2}
{\|\eta\|_{\infty}(1-2\|\eta\|_{\infty}R_{\eta}^2\|z\|/R)}
\left(
\frac{\|\eta\|_{\infty}R_{\eta}\|z\|}{R}
\right)^N
\frac{((1+\|\eta\|_{\infty})^2R_{\eta}/\|\eta\|_{\infty})^N-1}
{(1+\|\eta\|_{\infty})^2R_{\eta}/\|\eta\|_{\infty}-1}
\\
& \leq &
\frac{16\|f\|_RR_{\eta}}
{(1+\|\eta\|_{\infty})^2(1-2\|\eta\|_{\infty}R_{\eta}^2\|z\|/R)}
\left(
\frac{R_{\eta}^2(1+\|\eta\|_{\infty})^2\|z\|}{R}
\right)^N
.
\end{eqnarray*}
If we set
\begin{eqnarray*}
R\;=\;
R_{\eta,K} & := &
4\left(1+\|\eta\|_{\infty}\right)^2R_{\eta}^2
\sup_{z\in K}\|z\|
\,,
\end{eqnarray*}
in particular
$2\|\eta\|_{\infty}R_{\eta}^2\|z\|/R_{\eta,K}\leq1/2<1$ and
\begin{eqnarray}\label{plus}
\sup_{z\in K}|R_N(f;\eta)(z)|
& \leq &
\frac{32R_{\eta}\|f\|_{R_{\eta,K}}}
{(1+\|\eta\|_{\infty})^2}
\frac{1}{4^N}
\;\xrightarrow[N\rightarrow\infty]{}\;
0\,,
\end{eqnarray}
and the proof of Theorem~\ref{equivbounded} is achieved.

\end{proof}

Furthermore, (\ref{plus}) yields to a precision for the convergence of
$E_N(f;\eta)$.

\begin{corollary}\label{velocidad}

Assume that $\{\eta_j\}_{j\geq1}$ satisfies~(\ref{criter}).
Let be $\mathcal{K}\subset\mathcal{O}\left(\C^2\right)$
(resp. $K\subset\C^2$)
a compact subset. Then there exists
$C_{\mathcal{K},K}$ such that, for all
$N\geq1$,
\begin{eqnarray*}
\sup_{f\in\mathcal{K}}
\sup_{z\in K}
\left|
f(z)-E_N(f;\eta)(z)
\right|
& \leq &
\frac{C_{\mathcal{K},K}}{4^N}
\,.
\end{eqnarray*}

\end{corollary}

\begin{proof}

It follows from above and the fact that
$f(z)=E_N(f;\eta)(z)-R_N(f;\eta)(z)+\sum_{k+l\geq N}a_{k,l}z_1^kz_2^l$,
with
$|a_{k,l}|\leq\|f\|_R/R^{k+l}$, $\forall\,R\geq1$.

\end{proof}

\bigskip

\section{Proof of Theorems~\ref{equivnodense} and~\ref{superbe}}\label{proofs}

\subsection{Proof of Theorem \ref{superbe} when
$\{\eta_j\}_{j\geq1}$ is bounded}\label{proofsuperbounded}

\subsubsection{An equivalent condition for
$\{\eta\}_{j\geq1}$ to be real-analytically interpolated}

We begin with giving the following definition.

\begin{definition}\label{defunifDelta}

The (bounded) set
$\{\eta_j\}_{j\geq1}\subset\C$
will be said {\em of uniform exponential $\Delta$}
if there exist $C_{\eta},\,R_{\eta}$ such that,
for all subsequence $(j_k)_{k\geq1}$ and
for all $p\geq0$,
\begin{eqnarray*}
\left|
\Delta_{p,(\eta_{j_p},\ldots,\eta_{j_1})}
\left(
\zeta\mapsto\overline{\zeta}
\right)
(\eta_{j_{p+1}})
\right|
& \leq &
C_{\eta}\,R_{\eta}^p\,.
\end{eqnarray*}

\end{definition}

This condition looks like~(\ref{criter}) from Theorem~\ref{equivbounded},
with the difference that the constant $R_{\eta}$ does not
depend on the subsequence $(\eta_{j_k})_{k\geq1}$. This uniform condition
for $\{\eta_j\}_{j\geq1}$ seems stronger, in particular as it is specified
by the following result.

\begin{proposition}\label{unifDelta}

The bounded set $\{\eta_j\}_{j\geq1}\subset\C$
is real-analytically interpolated if and only if
it is of uniform exponential $\Delta$.

\end{proposition}

We begin with the reciprocal sense of the equivalence.

\begin{lemma}

If $\{\eta_j\}_{j\geq1}$ is of uniform exponential $\Delta$,
then it is real-analytically interpolated.

\end{lemma}

\begin{proof}

We want to prove that, for all
$\zeta_0\in\overline{\{\eta_j\}_{j\geq1}}$,
there exist $V\in\mathcal{V}\left(\zeta_0\right)$ and
$g\in\mathcal{O}(V)$ such that,
$\forall\,\eta_j\in V$,
$\overline{\eta_j}=g(\eta_j)$.

If $\zeta_0$ is isolated, then
$\zeta_0=\eta_{j_0}$. Let be
$V\in\mathcal{V}(\eta_{j_0})$ such that
$V\cap\{\eta_j\}_{j\geq1}=\{\eta_{j_0}\}$.
One can choose the constant function
$g(\zeta):=\overline{\eta_{j_0}}$.

Otherwise $\zeta_0$ is a limit point. 
Let be $V=D(\zeta_0,1/(4R_{\eta}))$ and
let be
$(\eta_{j_k})_{k\geq1}$ a sequence that converges to $\zeta_0$.
We can assume that $\{\eta_{j_k}\}_{k\geq1}\subset V$ (by removing
a finite number of points if necessary).
Consider the Lagrange interpolation polynomial that is also
by Lemma~\ref{lagrangebis}
\begin{eqnarray*}
P_N(\zeta) & = &
\sum_{p=1}^N
\prod_{k=1,k\neq p}^N
\frac{\zeta-\eta_{j_k}}{\eta_{j_p}-\eta_{j_k}}
\,\overline{\eta_{j_p}}
\\
& = &
\sum_{p=0}^{N-1}
\prod_{k=1}^p
(z-\eta_{j_k})\,
\Delta_{p,(\eta_{j_p},\ldots,\eta_{j_1})}
\left(\zeta\mapsto\overline{\zeta}\right)
(\eta_{j_{p+1}})\,.
\end{eqnarray*}
For all $\zeta\in V$ and $k\geq1$, one has
$|\zeta-\eta_{j_k}|\leq
|(\zeta-\zeta_0)-(\eta_{j_k}-\zeta_0)|
\leq
1/(2R_{\eta})$
then
\begin{eqnarray*}
\sum_{p\geq0}
\sup_{\zeta\in V}
\left|
\prod_{k=1}^p
(\zeta-\eta_{j_k})\,
\Delta_{p,(\eta_{j_p},\ldots,\eta_{j_1})}
\left(\overline{\zeta}\right)
(\eta_{j_{p+1}})
\right|
& \leq &
\sum_{p\geq0}
\left(\frac{1}{2R_{\eta}}\right)^p
\,C_{\eta}R_{\eta}^p
\\
& = &
C_{\eta}
\sum_{p\geq0}
\frac{1}{2^p}
\\
& = &
2C_{\eta}\,.
\end{eqnarray*}
The series
$\sum_{p\geq0}
\prod_{k=1}^p
(\zeta-\eta_{j_k})\,
\Delta_{p,(\eta_{j_p},\ldots,\eta_{j_1})}
\left(\overline{\zeta}\right)
(\eta_{j_{p+1}})$
is absolutely convergent on $V$. The sequence
$(P_N)_{N\geq1}$ uniformly converges on $V$ to a function
$g_1\in\mathcal{O}(V)$.
Moreover 
\begin{eqnarray*}
\sup_{\zeta\in V}|g_1(\zeta)|
& \leq &
2C_{\eta}
\end{eqnarray*}
and for all $k\geq1$,
\begin{eqnarray*}
g_1(\eta_{j_k})\;=\;
\lim_{N\rightarrow\infty}
P_N(\eta_{j_k})
\;=\;
\lim_{N\rightarrow\infty,\,N\geq k}
P_N(\eta_{j_k})
\;=\;
\overline{\eta_{j_k}}
\,,
\end{eqnarray*}
ie $g_1$ is a holomorphic and bounded function on $V$
that interpolates the values
$\overline{\eta_{j_k}}$ on the points
$\eta_{j_k},\,k\geq1$.

Now if
$V\cap\{\eta_j\}_{j\geq1}=\{\eta_{j_k}\}_{k\geq1}$,
the function $g_1$ satisfies the required conditions.
Otherwise we set
$S_1:=(\eta_{j_k})_{k\geq1}$,
choose an element
$\eta_{p_2}\in V\setminus S_1$ and set
$S_2:=(\eta_{p_2},S_1):=(\eta_{p_2},\eta_{j_1},\ldots,\eta_{j_k},\ldots)$.
Then $S_2\subset V$ is another subsequence of
$\{\eta_j\}_{j\geq1}$ that converges to $\zeta_0$. One can construct
the same sequence of Lagrange polynomials that converges to a
function $g_2\in\mathcal{O}(V)$ (since
$\{\eta_j\}_{j\geq1}$ is of uniform exponential $\Delta$). With the same argument
$g_2$ is bounded  (by $2C_{\eta}$) and interpolates the values
$\overline{\eta_j}$ on the points
$\eta_{p_2},\eta_{j_1},\ldots,\eta_{j_k},\ldots$.

We can follow this process as long as there is
$\eta_j\in V$ that is not reached. If there is
$r\geq1$ such that $S_r=V\cap\{\eta_j\}_{j\geq1}$, the associate function
$g_r$ will satisfy the required conditions. Otherwise we can construct
a sequence
$(S_r,g_r)_{r\geq1}$ with
$S_{r}=(\eta_{p_{r}},S_{r-1})$ and
$g_r\in\mathcal{O}(V)$, bounded that interpolates
the values $\overline{\eta_j}$ on $S_r$.
Since $\{\eta_j\}_{j\geq1}$ is countable,
for all $\eta_j\in V$, there exists $r\geq1$ such that
$\eta_j\in S_r$ and
$\forall\,s\geq r$,
$g_s(\eta_j)=g_r(\eta_j)=\overline{\eta_j}$.

On the other hand the sequence $(g_r)_{r\geq1}$ is uniformly
bounded on $V$ (by $2C_{\eta}$). By the
Stiltjes-Vitali-Montel Theorem, there is a subsequence
$(g_{r_l})_{l\geq1}$ that uniformly converges on $V$ to a function
$g_{\infty}\in\mathcal{O}(V)$. So $\forall\,\eta_j\in V$,
$\exists\,r\geq1$, $\eta_j\in S_r$ and
\begin{eqnarray*}
g_{\infty}(\eta_j)\;=\;
\lim_{l\rightarrow\infty}g_{r_l}(\eta_j)
\;=\;
\lim_{l\rightarrow\infty,\,r_l\geq r}g_{r_l}(\eta_j)
\;=\;
\lim_{l\rightarrow\infty,\,r_l\geq r}\overline{\eta_j}
\;=\;
\overline{\eta_j}\,,
\end{eqnarray*}
ie the function $g_{\infty}$ interpolates the values
$\overline{\eta_j}$ on all the points
$\eta_j\in V$.

\end{proof}

Now we prove the first sense of the equivalence.

\begin{lemma}

If $\{\eta_j\}_{j\geq1}$ is real-analytically interpolated,
then it is of uniform exponential $\Delta$.

\end{lemma}

\begin{proof}

For all
$\zeta\in\overline{\{\eta_j\}_{j\geq1}}$, 
reducing $V_{\zeta}$ if necessary, we can assume that
$V_{\zeta}=D(\zeta,3\varepsilon_{\zeta})$. Since
\begin{eqnarray*}
\overline{\{\eta_j\}_{j\geq1}}
& \subset &
\bigcup_{\zeta\in\overline{\{\eta_j\}_{j\geq1}}}D(\zeta,\varepsilon_{\zeta})
\end{eqnarray*}
and
$\overline{\{\eta_j\}_{j\geq1}}$ is a compact subset, there exists a
finite number $\zeta_1,\ldots,\zeta_L$ such that
\begin{eqnarray*}
\overline{\{\eta_j\}_{j\geq1}}
& \subset &
\bigcup_{l=1}^L
D(\zeta_l,\varepsilon_{\zeta_l})
\,.
\end{eqnarray*}
There also exists
$\varepsilon_0$ such that, for all 
$\zeta\in\overline{\{\eta_j\}_{j\geq1}}$,
$\exists\,l,\,1\leq l\leq L$,
$D(\zeta,\varepsilon_0)\subset D(\zeta_l,\varepsilon_{\zeta_l})$.

Now we begin with giving the proof in the following special case.

\begin{lemma}\label{aux2}

Let be $p\geq1$ and
$\eta_{j_1},\ldots,\eta_{j_{p+1}}$ such that,
for all $1\leq k<l\leq p+1$,
$|\eta_{j_k}-\eta_{j_l}|<\varepsilon_0$.
Then $\exists\,C_{\eta},\,\varepsilon_{\eta}$
(that do not depend on $p$), 
\begin{eqnarray*}
\left|
\Delta_{p,(\eta_{j_p},\ldots,\eta_{j_1})}
\left(\overline{\zeta}\right)
(\eta_{j_{p+1}})
\right|
& \leq &
\frac{C_{\eta}}{\varepsilon_{\eta}^p}\,.
\end{eqnarray*}

\end{lemma}

\begin{proof}

In particular,
$\eta_{j_2},\ldots,\eta_{j_{p+1}}\in D(\eta_{j_1},\varepsilon_0)$.
On the other hand,
$\exists\,l,\,1\leq l\leq L$,
$D(\eta_{j_1},\varepsilon_0)\subset D(\zeta_l,\varepsilon_{\zeta_l})$.
Since the function
$g_{\zeta_l}\in\mathcal{O}\left(D(\zeta_l,3\varepsilon_{\zeta_l})\right)$
interpolates the $\overline{\eta_j}$ on the points
$\eta_j\in D(\zeta_l,3\varepsilon_{\zeta_l})$,
one has
\begin{eqnarray*}
\Delta_{p,(\eta_{j_p},\ldots,\eta_{j_1})}
\left(\zeta\mapsto\overline{\zeta}\right)
(\eta_{j_{p+1}})
& = &
\Delta_{p,(\eta_{j_p},\ldots,\eta_{j_1})}
\left(\zeta\mapsto g_{\zeta_l}(\zeta)\right)
(\eta_{j_{p+1}})
\,.
\end{eqnarray*}
Let consider for all
$|\zeta-\zeta_l|<3\varepsilon_{\zeta_l}$,
\begin{eqnarray*}
g_{\zeta_l}(\zeta) & = &
\sum_{n\geq0}a_n(\zeta_l)(\zeta-\zeta_l)^n
\end{eqnarray*}
the Taylor expansion of $g_{\zeta_l}$ on $\zeta_l$. Since
$\eta_{j_1},\ldots,\eta_{j_p},\eta_{j_{p+1}}
\in D(\zeta_l,\varepsilon_{\zeta_l})$, it follows by
Lemmas~\ref{Deltanal} and~\ref{combi} that
\begin{eqnarray*}
\left|
\Delta_{p,(\eta_{j_p},\ldots,\eta_{j_1})}
\left(\zeta\mapsto g_{\zeta_l}(\zeta)\right)
(\eta_{j_{p+1}})
\right|
& \leq &
\;\;\;\;\;\;\;\;\;\;\;\;\;\;\;\;\;\;\;\;\;\;\;\;\;\;\;\;\;\;\;\;\;\;\;\;\;\;\;\;
\;\;\;\;\;\;\;\;\;\;\;\;\;\;\;\;\;\;\;\;\;\;\;\;\;\;
\end{eqnarray*}
\begin{eqnarray*}
& \leq &
\sum_{n\geq p}
|a_n(\zeta_l)|
\sum_{l_1=0}^{n-p}
|\eta_{j_1}-\zeta_l|^{n-p-l_1}
\cdots
\sum_{l_p=0}^{l_{p-1}}
|\eta_{j_p}-\zeta_l|^{l_{p-1}-l_p}
|\eta_{j_{p+1}}-\zeta_l|^{l_p}
\\
& \leq &
\sum_{n\geq p}
|a_n(\zeta_l)|
\sum_{l_1=0}^{n-p}
\cdots
\sum_{l_p=0}^{l_{p-1}}
\varepsilon_{\zeta_l}^{n-p-l_1+\cdots+l_{p-1}-l_p+l_p}
\\
& = &
\sum_{n\geq p}
|a_n(\zeta_l)|\,
\varepsilon_{\zeta_l}^{n-p}
\sum_{l_1=0}^{n-p}
\cdots
\sum_{l_p=0}^{l_{p-1}}1
\;=\;
\sum_{n\geq p}
|a_n(\zeta_l)|\,
\varepsilon_{\zeta_l}^{n-p}
\frac{n!}{p!\,(n-p)!}
\\
& = &
\frac{1}{p!}\frac{\partial^p}{\partial t^p}|_{t=\varepsilon_{\zeta_l}}
\left(
\sum_{n\geq0}|a_n(\zeta_l)|t^n
\right)
\;=\;
\frac{1}{2i\pi}
\int_{|t-\zeta_l|=2\varepsilon_{\zeta_l}}
\frac{\sum_{n\geq0}|a_n(\zeta_l)|t^n\;dt}{(t-\varepsilon_{\zeta_l})^{p+1}}
\\
& \leq &
2\varepsilon_{\zeta_l}
\sup_{|t|=2\varepsilon_{\zeta_l}}
\frac{\sum_{n\geq0}|a_n(\zeta_l)|\,|t|^n}{(|t|-\varepsilon_{\zeta_l})^{p+1}}
\;=\;
\frac{2}{\varepsilon_{\zeta_l}^p}
\sum_{n\geq0}
|a_n(\zeta_l)|(2\varepsilon_{\zeta_l})^n
\,.
\end{eqnarray*}

For all $l=1,\ldots,L$, we set
\begin{eqnarray*}
M_{\zeta_l}\;:=\;2
\sum_{n\geq0}
|a_n(\zeta_l)|(2\varepsilon_{\zeta_l})^n
\;<\;+\infty
\,,
\end{eqnarray*}
\begin{eqnarray*}
C_{\eta} & := &
\max_{1\leq l\leq L}M_{\zeta_l}
\;<\;+\infty
\end{eqnarray*}
and
\begin{eqnarray*}
\varepsilon_{\eta} & := &
\min\left\{
\varepsilon_0/2\,,\;
\min_{1\leq l\leq L}\varepsilon_{\zeta_l}
\right\}
\;>\;
0\,.
\end{eqnarray*}
Thus
\begin{eqnarray*}
\left|
\Delta_{p,(\eta_{j_p},\ldots,\eta_{j_1})}
\left(\overline{\zeta}\right)
(\eta_{j_{p+1}})
\right|
& \leq &
\frac{C_{\eta}}{\varepsilon_{\eta}^p}
\,.
\end{eqnarray*}

\end{proof}

Now we can give the proof in the general case by induction on
$p\geq0$ with the above choice of
$C_{\eta},\,\varepsilon_{\eta}$.

Let be $p=0$ and $j_1\geq1$. Then
$\exists\,l,\,1\leq l\leq L$,
$\eta_{j_1}\in D(\zeta_l,\varepsilon_{\zeta_l})$, thus
\begin{eqnarray*}
\left|
\Delta_0
\left(\overline{\zeta}\right)(\eta_{j_1})
\right|
& = &
\left|\overline{\eta_{j_1}}\right|
\;=\;
|g_{\zeta_l}(\eta_{j_1})|
\;\leq\;
\sum_{n\geq0}|a_n(\zeta_l)|\,|\eta_{j_1}-\zeta_l|^n
\\
& \leq &
\sum_{n\geq0}|a_n(\zeta_l)|\,\varepsilon_{\zeta_l}^n
\;\leq\;
M_{\zeta_l}
\;\leq\;
C_{\eta}
\,.
\end{eqnarray*}

Now if it is true for $p-1\geq0$, let be
$p\geq1$ and
$\eta_{j_1},\ldots,\eta_{j_{p+1}}
\in\{\eta_j\}_{j\geq1}$. If for all
$1\leq k<l\leq L$, 
$|\eta_{j_k}-\eta_{j_l}|<\varepsilon_0$,
then it is still true by Lemma~\ref{aux2}. 
Otherwise, $\exists\,k,l$ with
$1\leq k<l\leq L$ such that
$|\eta_{j_k}-\eta_{j_l}|\geq\varepsilon_0$, then
by Lemma~\ref{Deltapermut}
\begin{eqnarray*}
\left|
\Delta_{p,(\eta_{j_p},\ldots,\eta_{j_1})}
\left(\overline{\zeta}\right)
(\eta_{j_{p+1}})
\right|
& = &
\left|
\Delta_{p,(\eta_{j_k},\eta_{j_r},\ldots,\eta_{j_s})}
\left(\overline{\zeta}\right)
(\eta_{j_{l}})
\right|
\\
& = &
\left|
\frac{
\Delta_{p-1,(\eta_{j_r},\ldots,\eta_{j_s})}
\left(\overline{\zeta}\right)
(\eta_{j_{l}})
-
\Delta_{p-1,(\eta_{j_r},\ldots,\eta_{j_s})}
\left(\overline{\zeta}\right)
(\eta_{j_{k}})
}
{\eta_{j_l}-\eta_{j_k}}
\right|
\\
& \leq &
\frac{
\left|
\Delta_{p-1,(\eta_{j_r},\ldots,\eta_{j_s})}
\left(\overline{\zeta}\right)
(\eta_{j_{l}})
\right|
+
\left|
\Delta_{p-1,(\eta_{j_r},\ldots,\eta_{j_s})}
\left(\overline{\zeta}\right)
(\eta_{j_{k}})
\right|
}
{\varepsilon_0}
\\
& \leq &
\frac{2C_{\eta}/\varepsilon_{\eta}^{p-1}}{\varepsilon_0}
\;\leq\;
\frac{C_{\eta}}{\varepsilon_{\eta}^p}
\,.
\end{eqnarray*}

\end{proof}

\subsubsection{Proof of Theorem \ref{superbe} when
$\|\eta\|_{\infty}<+\infty$}\label{proofsuperbebounded}

Now we will give the proof of Theorem~\ref{superbe} in the
special case when
$\{\eta_j\}_{j\geq1}$ is bounded.

\begin{proof}

One has by~(\ref{lagrangebisRN}) from the proof of
Lemma~\ref{condnec1},
\begin{eqnarray*}
R_N(f;\eta)(z) & = &
\sum_{p=0}^{N-1}
z_2^{N-1-p}
\prod_{j=1}^p(z_1-\eta_jz_2)
\Delta_{p,(\eta_p,\ldots,\eta_1)}
\left(
\zeta\mapsto r_{N}(\zeta,z)
\right)(\eta_{p+1})
\,.
\end{eqnarray*}
We know that
$\forall\,\zeta_0\in\overline{\{\eta_j\}_{j\geq1}}$,
$\exists\,V_{\zeta_0}\in\mathcal{V}(\zeta_0)$,
$g_{\zeta_0}\in\mathcal{O}\left(V_{\zeta_0}\right)$,
such that
$\forall\,\eta_j\in V_{\zeta_0}$,
$\overline{\eta_j}=g_{\zeta_0}(\eta_j)$.
In particular,
$\overline{\zeta_0}=g_{\zeta_0}(\zeta_0)$. Indeed, if
$\zeta_0$ is isolated, 
$\zeta_0=\eta_{j_0}$. Reducing
$V_{\zeta_0}$ if necessary, one has
$\{\eta_j\}_{j\geq1}\cap V_{\zeta_0}=\{\eta_{j_0}\}$
then
$\overline{\zeta_0}=\overline{\eta_{j_0}}
=g_{\zeta_0}(\eta_{j_0})=g_{\zeta_0}(\zeta_0)$.
Otherwise $\zeta_0$ is a limit point so there is a
subsequence
$(\eta_{j_k})_{k\geq1}$ that converges to
$\zeta_0$, then
$\overline{\zeta_0}=\lim_{k\rightarrow\infty}\overline{\eta_{j_k}}
=\lim_{k\rightarrow\infty}g_{\zeta_0}(\eta_{j_k})
=g_{\zeta_0}(\zeta_0)$.

In particular,
$\zeta_0g_{\zeta_0}(\zeta_0)=|\zeta_0|^2\geq0$
then reducing $V_{\zeta_0}$ if necessary,
$\forall\,\zeta\in V_{\zeta_0}$,
$\Re\left(\zeta g_{\zeta_0}(\zeta)\right)>-1/2$.
Finally, reducing $V_{\zeta_0}$ again if necessary, one can choose
$V_{\zeta_0}=D(\zeta_0,4\varepsilon_{\zeta_0})$ such that
\begin{eqnarray}
\begin{cases}
\forall\,\eta_j\in V_{\zeta_0},\;
\overline{\eta_j}=g_{\zeta_0}(\eta_j)\,,
\\
\forall\,\zeta\in V_{\zeta_0},\;
\Re\left(\zeta g_{\zeta_0}(\zeta)\right)>-1/2\,,
\\
\|g_{\zeta_0}\|_{V_{\zeta_0}}:=
\sup_{\zeta\in V_{\zeta_0}}
\left|
g_{\zeta_0}(\zeta)
\right|
<+\infty\,.
\end{cases}
\end{eqnarray}

Now let consider for all
$\zeta\in V_{\zeta_0}$
and $z\in\C^2$,
\begin{eqnarray}
\widetilde{r}_{N,\zeta_0}(\zeta,z)
& := &
\sum_{m\geq N}
\left(
\frac{z_2+z_1g_{\zeta_0}(\zeta)}
{1+\zeta g_{\zeta_0}(\zeta)}
\right)^{m-N+1}
\sum_{k+l=m}a_{k,l}\zeta^k
\,.
\end{eqnarray}
This function is well-defined for all
$N\geq1$,
$\zeta\in V_{\zeta_0}$ and $z\in\C^2$
since
\begin{eqnarray*}
\sum_{m\geq N}
\left|
\left(
\frac{z_2+z_1g_{\zeta_0}(\zeta)}
{1+\zeta g_{\zeta_0}(\zeta)}
\right)^{m-N+1}
\sum_{k+l=m}a_{k,l}\zeta^k
\right|
& \leq &
\;\;\;\;\;\;\;\;\;\;\;\;\;\;\;\;\;\;\;\;\;\;\;\;\;\;\;\;\;\;\;\;\;\;
\;\;\;\;\;\;\;\;\;\;\;\;\;\;\;\;\;\;\;\;\;
\end{eqnarray*}
\begin{eqnarray*}
& \leq &
\sum_{m\geq N}
\left(
\frac{\|z\|\sqrt{1+|g_{\zeta_0}(\zeta)|^2}}
{|1+\zeta g_{\zeta_0}(\zeta)|}
\right)^{m-N+1}
\sum_{k+l=m}|a_{k,l}||\zeta|^k
\\
& \leq &
\sum_{m\geq N}
\left(
\frac{\|z\|(1+\|g_{\zeta_0}\|_{V_{\zeta_0}})}
{|\Re(1+\zeta g_{\zeta_0}(\zeta))|}
\right)^{m-N+1}
\sum_{k+l=m}
\frac{\|f\|_R}{R^{k+l}}
\|\eta\|_{\infty}^k
\\
& \leq &
\|f\|_R
\sum_{m\geq N}
\left(
2\|z\|(1+\|g_{\zeta_0}\|_{V_{\zeta_0}})
\right)^{m-N+1}
\frac{1}{R^m}
\frac{(2+\|\eta\|_{\infty})^{m+1}-1}{(2+\|\eta\|_{\infty})-1}
\\
& \leq &
\frac{2\|f\|_R}{(2\|z\|(1+\|g_{\zeta_0}\|_{V_{\zeta_0}}))^{N-1}}
\sum_{m\geq N}
\left(
\frac{2\|z\|(2+\|\eta\|_{\infty})
(1+\|g_{\zeta_0}\|_{V_{\zeta_0}})}{R}
\right)^m
\\
& = &
\frac{2\|f\|_R}{(2\|z\|(1+\|g_{\zeta_0}\|_{V_{\zeta_0}}))^{N-1}}
\;
\frac{
\left(
2(2+\|\eta\|_{\infty})
\|z\|(1+\|g_{\zeta_0}\|_{V_{\zeta_0}})/R
\right)^N}
{1-2(2+\|\eta\|_{\infty})
\|z\|(1+\|g_{\zeta_0}\|_{V_{\zeta_0}})/R}
\,,
\end{eqnarray*}
for all
$R>2(2+\|\eta\|_{\infty})\|z\|(1+\|g_{\zeta_0}\|_{V_{\zeta_0}})$.
Moreover
$\widetilde{r}_{N,\zeta_0}\in\mathcal{O}\left(V_{\zeta_0}\times\C^2\right)$
and for any compact subset
$K\in\C^2$ and all
$R\geq4(2+\|\eta\|_{\infty})\|z\|_K(1+\|g_{\zeta_0}\|_{V_{\zeta_0}})$,
\begin{eqnarray}\nonumber
\sup_{(\zeta,z)\in V_{\zeta_0}\times K}
\left|
\widetilde{r}_{N,\zeta_0}(\zeta,z)
\right|
& \leq &
\frac{4\|f\|_R\|z\|_K(1+\|g_{\zeta_0}\|_{V_{\zeta_0}})}
{1-2(2+\|\eta\|_{\infty})\|z\|_K(1+\|g_{\zeta_0}\|_{V_{\zeta_0}})/R}
\left(
\frac{2+\|\eta\|_{\infty}}{R}
\right)^N
\\\label{estimrtildeNzeta}
& \leq &
8\|f\|_R\|z\|_K(1+\|g_{\zeta_0}\|_{V_{\zeta_0}})
\left(
\frac{2+\|\eta\|_{\infty}}{R}
\right)^N
\,.
\end{eqnarray}
Moreover, $\forall\,\eta_j\in V_{\zeta_0}$,
\begin{eqnarray}\label{rNtildeinterpol}
r_N(\eta_j,z)
& = &
\widetilde{r}_{N,\zeta_0}(\eta_j,z)
\,.
\end{eqnarray}

Now let consider the Taylor expansion of 
$\widetilde{r}_{N,\zeta_0}(\cdot,z)$
on $V_{\zeta_0}$,
\begin{eqnarray*}
\widetilde{r}_{N,\zeta_0}(\zeta,z)
& = &
\sum_{n\geq0}
a_n(N,\zeta_0,z)(\zeta-\zeta_0)^n
\,.
\end{eqnarray*}
One has, for all $n\geq0$,
\begin{eqnarray*}
|a_n(N,\zeta_0,z)|\;=\;
\left|
\frac{1}{2i\pi}
\int_{|\zeta-\zeta_0|=3\varepsilon_{\zeta_0}}
\frac{\widetilde{r}_{N,\zeta_0}(\zeta,z)\,d\zeta}
{(\zeta-\zeta_0)^{n+1}}
\right|
\;\leq\;
\frac{\|\widetilde{r}_{N,\zeta_0}(\cdot,z)\|_{V_{\zeta_0}}}
{(3\varepsilon_{\zeta_0})^n}
\,,
\end{eqnarray*}
then
\begin{eqnarray*}
\sup_{z\in K}
|a_n(N,\zeta_0,z)|
& \leq &
\frac{\|\widetilde{r}_{N,\zeta_0}\|_{V_{\zeta_0}\times K}}
{(3\varepsilon_{\zeta_0})^n}
\,.
\end{eqnarray*}
It follows that
\begin{eqnarray}\nonumber
\sup_{z\in K}
\sum_{n\geq0}
|a(N,\zeta_0,z)|(2\varepsilon_{\zeta_0})^n
& \leq &
\sum_{n\geq0}
\|\widetilde{r}_{N,\zeta_0}\|_{V_{\zeta_0}\times K}
\left(\frac{2}{3}\right)^n
\\\label{MNzetaK}
& \leq &
3\|\widetilde{r}_{N,\zeta_0}\|_{V_{\zeta_0}\times K}
\;=:\;
M_{N,\zeta_0,K}
\,.
\end{eqnarray}

Since it is true for any
$\zeta_0\in\overline{\{\eta_j\}_{j\geq1}}$
and
$\overline{\{\eta_j\}_{j\geq1}}$ is compact,
there are
$\zeta_1,\ldots,\zeta_L$ such that
\begin{eqnarray*}
\overline{\{\eta_j\}_{j\geq1}}
& \subset &
\bigcup_{l=1}^L
D(\zeta_l,\varepsilon_{\zeta_l})
\,.
\end{eqnarray*}
Moreover, there is $\varepsilon_0>0$ such that,
for all 
$\zeta\in\overline{\{\eta_j\}_{j\geq1}}$,
$\exists\,l,\,1\leq l\leq L$,
$D(\zeta,\varepsilon_0)\subset
D(\zeta_l,\varepsilon_{\zeta_l})$.
Set
\begin{eqnarray}\label{MNK}
M_{N,K}
& := &
2\max_{1\leq l\leq L}
M_{N,\zeta_l,K}
\end{eqnarray}
and
\begin{eqnarray}\label{epsilon}
\varepsilon
& := &
\min
\left\{1/2,
\varepsilon_0/2,
\min_{1\leq l\leq L}\varepsilon_{\zeta_l}
\right\}
\,.
\end{eqnarray}
For all
\begin{eqnarray}\label{RK1}
R & \geq & 
4(2+\|\eta\|_{\infty})\|z\|_K
\left(1+\max_{1\leq l\leq L}
\|g_{\zeta_l}\|_{V_{\zeta_l}}\right)
\,,
\end{eqnarray}
one has by~(\ref{estimrtildeNzeta}),
(\ref{MNzetaK}) and~(\ref{MNK}),
\begin{eqnarray}\nonumber
M_{N,K}
& = &
6\max_{1\leq l\leq L}
\|\widetilde{r}_{N,\zeta_l}\|_{V_{\zeta_l}\times K}
\\\label{estimMNK}
& \leq &
48\|f\|_R\|z\|_K
\left(
1+\max_{1\leq l\leq L}\|g_{\zeta_l}\|_{V_{\zeta_l}}
\right)
\left(
\frac{2+\|\eta\|_{\infty}}{R}
\right)^N
\;<\;
+\infty
\,.
\end{eqnarray}

Now we want to prove by induction on $p\geq0$ that
\begin{eqnarray}\label{estimDeltarN}
\sup_{z\in K}
\left|
\Delta_{p,(\eta_p,\ldots,\eta_1)}
(\zeta\mapsto r_N(\zeta,z))(\eta_{p+1})
\right|
& \leq &
\frac{M_{N,K}}{\varepsilon^p}
\,.
\end{eqnarray}

If $p=0$, let be $\zeta_l$ such that
$\eta_1\in D(\zeta_l,\varepsilon_{\zeta_l})$.
One has
\begin{eqnarray*}
\sup_{z\in K}
\left|
\Delta_{0}
(\zeta\mapsto r_N(\zeta,z))
(\eta_1)
\right|
& = &
\sup_{z\in K}|r_N(\eta_1,z)|
\;=\;
\sup_{z\in K}
\left|\widetilde{r}_{N,\zeta_l}(\eta_1,z)\right|
\\
& \leq &
\|\widetilde{r}_{N,\zeta_l}\|_{V_{\zeta_l}\times K}
\;\leq\;
M_{N,\zeta_l,K}
\;\leq\;
M_{N,K}
\,.
\end{eqnarray*}

Now if it is true for $p\geq0$, let be
$\eta_1,\ldots,\eta_{p+1},\eta_{p+2}$.
If there exist
$1\leq k<l\leq p+2$ such that
$|\eta_k-\eta_l|\geq\varepsilon_0$, then by
permuting if necessary the $\eta_j,\,1\leq j\leq p+2$,
(that does not change $\Delta_{p+1}$ by Lemma~\ref{Deltapermut}),
we can assume that 
$|\eta_{p+2}-\eta_{p+1}|\geq\varepsilon_0$. Then
\begin{eqnarray*}
\sup_{z\in K}
\left|
\Delta_{p+1,(\eta_{p+1},\ldots,\eta_1)}
(\zeta\mapsto r_N(\zeta,z))(\eta_{p+2})
\right|
& \leq &
\;\;\;\;\;\;\;\;\;\;\;\;\;\;\;\;\;\;\;\;\;\;\;\;\;\;\;\;\;\;\;\;\;\;\;
\;\;\;\;\;\;\;\;\;\;\;\;\;\;\;\;\;\;\;\;\;
\end{eqnarray*}
\begin{eqnarray*}
& \leq &
\frac{
\sup_{z\in K}
\left|
\Delta_{p,(\eta_{p},\ldots,\eta_1)}
(r_N(\zeta,z))(\eta_{p+2})
\right|
+
\sup_{z\in K}
\left|
\Delta_{p,(\eta_{p},\ldots,\eta_1)}
(r_N(\zeta,z))(\eta_{p+1})
\right|
}
{|\eta_{p+2}-\eta_{p+1}|}
\\
& \leq &
\frac{M_{N,K}/\varepsilon^p+M_{N,K}/\varepsilon^p}
{\varepsilon_0}
\;=\;
\frac{M_{N,K}/\varepsilon^p}{\varepsilon_0/2}
\;\leq\;
\frac{M_{N,K}}{\varepsilon^{p+1}}
\,.
\end{eqnarray*}
Otherwise, for all
$1\leq l<k\leq p+2$,
$|\eta_k-\eta_l|<\varepsilon_0$.
In particular,
$\eta_2,\ldots,\eta_{p+1}
\in D(\eta_1,\varepsilon_0)$.
On the other hand, $\exists\,\zeta_l$,
$D(\eta_1,\varepsilon_0)\subset
D(\zeta_l,\varepsilon_{\zeta_l})$.
Then $\forall\,z\in K$, one has by~(\ref{rNtildeinterpol}),
Lemmas~\ref{Deltanal} and~\ref{combi}
\begin{eqnarray*}
\left|
\Delta_{p+1,(\eta_{p+1},\ldots,\eta_1)}
(r_N(\zeta,z))(\eta_{p+2})
\right|
& = &
\;\;\;\;\;\;\;\;\;\;\;\;\;\;\;\;\;\;\;\;\;\;\;\;\;\;\;\;\;\;\;\;
\;\;\;\;\;\;\;\;\;\;\;\;\;\;\;\;\;\;\;\;\;\;\;\;\;\;\;\;\;\;\;\;
\end{eqnarray*}
\begin{eqnarray*}
& = &
\left|
\Delta_{p+1,(\eta_{p+1},\ldots,\eta_1)}
(\widetilde{r}_{N,\zeta_l}(\zeta,z))(\eta_{p+2})
\right|
\\
& \leq &
\sum_{n\geq p+1}
|a_n(N,\zeta_l,z)|
\sum_{s_1=0}^{n-p-1}
|\eta_1-\zeta_l|^{n-p-1-s_1}
\cdots
\sum_{s_{p+1}=0}^{s_{p}}
|\eta_{p+1}-\zeta_l|^{s_p-s_{p+1}}
|\eta_{p+2}-\zeta_l|^{s_{p+1}}
\\
& \leq &
\sum_{n\geq p+1}
|a_n(N,\zeta_l,z)|
\varepsilon_{\zeta_l}^{n-p-1}
\sum_{s_1=0}^{n-p-1}
\cdots
\sum_{s_{p+1}=0}^{s_{p+1}}
1
\\
& = &
\sum_{n\geq p+1}
|a_n(N,\zeta_l,z)|
\varepsilon_{\zeta_l}^{n-p-1}
\frac{n!}{(p+1)!\,(n-p-1)!}
\\
& = &
\frac{1}{(p+1)!}
\frac{\partial^{p+1}}{\partial t^{p+1}}|_{t=\varepsilon_{\zeta_l}}
\left[
\sum_{n\geq0}
|a_n(N,\zeta_l,z)|\,t^n
\right]
\;=\;
\frac{1}{2i\pi}
\int_{|t|=2\varepsilon_{\zeta_l}}
\frac{\sum_{n\geq0}|a_n(N,\zeta_l,z)|\,t^n\,dt}
{(t-\varepsilon_{\zeta_l})^{p+2}}
\\
& \leq &
2\varepsilon_{\zeta_l}
\sup_{|t|=2\varepsilon_{\zeta_l}}
\frac{\sum_{n\geq0}|a_n(N,\zeta_l,z)|\,|t|^n\,dt}
{(|t|-\varepsilon_{\zeta_l})^{p+2}}
\;=\;
\frac{2}{\varepsilon_{\zeta_l}^{p+1}}
\sum_{n\geq0}
|a_n(N,\zeta_l,z)|
(2\varepsilon_{\zeta_l})^n
.
\end{eqnarray*}
It follows by~(\ref{MNzetaK}), (\ref{MNK})
and~(\ref{epsilon}) that
\begin{eqnarray*}
\sup_{z\in K}
\left|
\Delta_{p+1,(\eta_{p+1},\ldots,\eta_1)}
(r_N(\zeta,z))(\eta_{p+2})
\right|
& \leq &
\frac{2}{\varepsilon_{\zeta_l}^{p+1}}
\sum_{n\geq0}
\sup_{z\in K}
|a_n(N,\zeta_l,z)|
(2\varepsilon_{\zeta_l})^n
\\
& \leq &
\frac{2M_{N,\zeta_l,K}}{\varepsilon_{\zeta_l}^{p+1}}
\;\leq\;
\frac{M_{N,K}}{\varepsilon^{p+1}}
\,,
\end{eqnarray*}
and this proves~(\ref{estimDeltarN}).

Then one has by~(\ref{epsilon})
and~(\ref{estimMNK}), for all
$N\geq1$,
\begin{eqnarray*}
\sup_{z\in K}
|R_N(f;z)(z)|
& \leq &
\sup_{z\in K}
\sum_{p=0}^{N-1}
|z_2|^{N-1-p}
\prod_{j=1}^p
|z_1-\eta_jz_2|
\left|
\Delta_{p,(\eta_p,\ldots,\eta_1)}
(\zeta\mapsto r_N(\zeta,z))(\eta_{p+1})
\right|
\\
& \leq &
\sum_{p=0}^{N-1}
\|z\|_K^{N-1-p}
\prod_{j=1}^p
\left(\|z\|_K
\sqrt{1+|\eta_j|^2}
\right)
\sup_{z\in K}
\left|
\Delta_p
(r_N(\zeta,z))(\eta_{p+1})
\right|
\\
& \leq &
\|z\|_K^{N-1}
\sum_{p=0}^{N-1}
(1+\|\eta\|_{\infty})^p
\frac{M_{N,K}}{\varepsilon^p}
\;=\;
M_{N,K}
\|z\|_K^{N-1}
\frac{((1+\|\eta\|_{\infty})/\varepsilon)^N-1}
{(1+\|\eta\|_{\infty})/\varepsilon-1}
\\
& \leq &
M_{N,K}
\|z\|_K^{N-1}
((1+\|\eta\|_{\infty})/\varepsilon)^N
\\
& \leq &
48\|f\|_R
\left(
1+\max_{1\leq l\leq L}
\|g_{\zeta_l}\|_{V_{\zeta_l}}
\right)
\left(
\frac{\|z\|_K(2+\|\eta\|_{\infty})
(1+\|\eta\|_{\infty})/\varepsilon}
{R}
\right)^N
.
\end{eqnarray*}
It follows by~(\ref{RK1}) that if we fix
\begin{eqnarray*}
R\;=\;
R_K
& := &
4\|z\|_K
(2+\|\eta\|_{\infty})^2
\left(1+\max_{1\leq l\leq L}
\|g_{\zeta_l}\|_{V_{\zeta_l}}\right)
/\varepsilon
\,,
\end{eqnarray*}
we get
\begin{eqnarray*}
\sup_{z\in K}
|R_N(f;z)(z)|
\;\leq\;
48\|f\|_{R_K}
\left(
1+\max_{1\leq l\leq L}
\|g_{\zeta_l}\|_{V_{\zeta_l}}
\right)
\frac{1}{4^N}
\;\xrightarrow[N\rightarrow\infty]{}\;
0\,,
\end{eqnarray*}
and the proof is achieved.

\end{proof}

\subsection{Proof of Theorems \ref{equivnodense} and \ref{superbe}}\label{endproofs}

In this part we do not assume any more that
$\{\eta_j\}_{j\geq1}$ is bounded, else it is not dense in
$\C$. Then
$\exists\,\eta_{\infty}\in\overline{\C},\,
\exists\,V\in\mathcal{V}(\eta_{\infty})$,
$\forall\,j\geq1,\,\eta_j\notin V$. We can assume that
$\eta_{\infty}\neq\infty$ because otherwise it means
that $\{\eta_j\}_{j\geq1}$ is bounded and the results we want to prove
have already been proved in Section~\ref{proofequivbounded}
and Subsection~\ref{proofsuperbounded}. Then
\begin{eqnarray}\label{nondense}
\exists\,\varepsilon_{\infty}>0,\;
\forall\,j\geq1,\;|\eta_j-\eta_{\infty}|\geq\varepsilon_{\infty}
\,.
\end{eqnarray}

\subsubsection{Proof of Theorem \ref{equivnodense}}

Let consider
\begin{eqnarray}
U_{\eta_{\infty}} & := &
\frac{1}{\sqrt{1+|\eta_{\infty}^2|}}
\bordermatrix{ & & \cr
& \overline{\eta_{\infty}}  & 1  \cr
& 1  & -\eta_{\infty}  \cr}
\,.
\end{eqnarray}
$U_{\eta_{\infty}}\in\mathcal{U}(2,\C)$, ie
\begin{eqnarray}
U_{\eta_{\infty}}^{\star}
\;=\;
U_{\eta_{\infty}}^{-1}
& = &
\frac{1}{\sqrt{1+|\eta_{\infty}|^2}}
\bordermatrix{ & & \cr
& \eta_{\infty}  & 1  \cr
& 1  & -\overline{\eta_{\infty}}  \cr}
\end{eqnarray}
and
\begin{eqnarray}
\begin{cases}
U_{\eta_{\infty}}(\{z_1-\eta_{\infty}z_2=0\})
\;=\;\{z_2=0\}\,,\\
U_{\eta_{\infty}}^{\star}(\{z_2=0\})
\;=\;\{z_1-\eta_{\infty}z_2=0\}\,.
\end{cases}
\end{eqnarray}

We remind the definition of
$\theta_j$ associate to $\eta_{\infty}$
(Introduction, (\ref{deftheta})),
\begin{eqnarray*}
\forall\,j\geq1,\;
\theta_j & = &
\frac{1+\overline{\eta_{\infty}}\eta_j}
{\eta_j-\eta_{\infty}}
\,,
\end{eqnarray*}
and we begin with this preliminar result.

\begin{lemma}\label{lemmextensio}

The following assertions are equivalent:

\begin{enumerate}

\item

the formula $R_N(f;\eta)$ converges for every
function $f\in\mathcal{O}\left(\C^2\right)$

\item
 
$\exists\,\eta_{\infty}\notin\overline{\{\eta_j\}_{j\geq1}}\cup\{\infty\}$
such that the formula
$R_N(f;\theta)$ (constructed with the associate $\theta_j$)
converges for every
$f\in\mathcal{O}\left(\C^2\right)$

\item

$\forall\,\eta_{\infty}\notin
\overline{\{\eta_j\}_{j\geq1}}\cup\{\infty\}$,
the formula
$R_N(f;\theta)$ converges for every
$f\in\mathcal{O}\left(\C^2\right)$.

\end{enumerate}

\end{lemma}

\begin{proof}

We begin with reminding this equality (Introduction,
Corollary~\ref{equivcv0}):
$\forall\,f\in\mathcal{O}(\C^2)$,
$\forall\,N\geq1$ and $\forall\,z\in\mathbb{B}_2$,
\begin{eqnarray}\label{prelimrest}
& &
\prod_{j=1}^N(z_1-\eta_jz_2)
\lim_{\varepsilon\rightarrow0}
\frac{1}{(2i\pi)^2}
\int_{\Omega_{\varepsilon}}
\frac{f(\zeta)\,\omega'(\overline{\zeta})\wedge\omega(\zeta)}
{\prod_{j=1}^N(\zeta_1-\eta_j\zeta_2)
\,
(1-<\overline{\zeta},z>)^2}
\;=\;
\end{eqnarray}
\begin{eqnarray*}
\;\;\;\;\;\;\;\;\;\;\;\;\;\;\;\;\;\;\;\;\;\;\;\;\;\;\;\;\;\;\;\;\;\;\;\;\;\;
\;\;\;\;\;\;\;\;\;\;\;\;\;\;\;\;\;\;\;\;\;
& = &
R_N(f;\eta)(z)
-\sum_{k+l\geq N}
a_{k,l}z_1^kz_2^l
\,,
\end{eqnarray*}
with
\begin{eqnarray}\label{defOmega}
\Omega_{\varepsilon}
& = &
\left\{
\zeta\in\mathbb{S}_2,\,
\left|\prod_{j=1}^N(\zeta_1-\eta_j\zeta_2)\right|>\varepsilon
\right\}
\,.
\end{eqnarray}
On the other hand,
\begin{eqnarray*}
\prod_{j=1}^N(z_1-\eta_jz_2)
\lim_{\varepsilon\rightarrow0}
\frac{1}{(2i\pi)^2}
\int_{\Omega_{\varepsilon}}
\frac{f(\zeta)\,\omega'(\overline{\zeta})\wedge\omega(\zeta)}
{\prod_{j=1}^N(\zeta_1-\eta_j\zeta_2)
\,
(1-<\overline{\zeta},z>)^2}
& = &
\;\;\;\;\;\;\;\;\;\;\;\;\;\;\;\;\;\;\;\;\;\;\;
\end{eqnarray*}
\begin{eqnarray*}
& = &
\prod_{j=1}^N(z_1-\eta_jz_2)
\times
\\
& &
\times
\lim_{\varepsilon\rightarrow0}
\frac{1}{(2i\pi)^2}
\int_{\Omega_{\varepsilon}}
\frac{f(U_{\eta_{\infty}}^{\star}U_{\eta_{\infty}}\zeta)
\,
\omega'(\overline{U_{\eta_{\infty}}^{\star}U_{\eta_{\infty}}\zeta})
\wedge\omega(U_{\eta_{\infty}}^{\star}U_{\eta_{\infty}}\zeta)}
{\prod_{j=1}^N((U_{\eta_{\infty}}^{\star}U_{\eta_{\infty}}\zeta)_1
-\eta_j
(U_{\eta_{\infty}}^{\star}U_{\eta_{\infty}}\zeta)_2)
\,
(1-<\overline{U_{\eta_{\infty}}^{\star}U_{\eta_{\infty}}\zeta},z>)^2}
\\
& = &
\prod_{j=1}^N(z_1-\eta_jz_2)
\times
\\
& &
\times
\lim_{\varepsilon\rightarrow0}
\frac{1}{(2i\pi)^2}
\int_{U_{\eta_{\infty}}(\Omega_{\varepsilon})}
\frac{f(U_{\eta_{\infty}}^{\star}\zeta)
\,
\omega'(\overline{U_{\eta_{\infty}}^{\star}\zeta})
\wedge\omega(U_{\eta_{\infty}}^{\star}\zeta)}
{\prod_{j=1}^N((U_{\eta_{\infty}}^{\star}\zeta)_1
-\eta_j
(U_{\eta_{\infty}}^{\star}\zeta)_2)
\,
(1-<\overline{U_{\eta_{\infty}}^{\star}\zeta},z>)^2}
\,.
\end{eqnarray*}
Since
\begin{eqnarray*}
\begin{cases}
(U_{\eta_{\infty}}^{\star}\zeta)_1
=
\dfrac{\eta_{\infty}\zeta_1+\zeta_2}{\sqrt{1+|\eta_{\infty}|^2}}
\\
\\
(U_{\eta_{\infty}}^{\star}\zeta)_2
=
\dfrac{\zeta_1-\overline{\eta_{\infty}}\zeta_2}{\sqrt{1+|\eta_{\infty}|^2}}
\end{cases}
\,,
\end{eqnarray*}
then
\begin{eqnarray*}
\omega'\left(\overline{U_{\eta_{\infty}}^{\star}\zeta}\right)
& = &
\frac{1}{1+|\eta_{\infty}|^2}
\left[
(\overline{\eta_{\infty}}\overline{\zeta_1}+\overline{\zeta_2})
(d\overline{\zeta_1}-\eta_{\infty}d\overline{\zeta_2})
-
(\overline{\zeta_1}-\eta_{\infty}\overline{\zeta_2})
(\overline{\eta_{\infty}}d\overline{\zeta_1}+d\overline{\zeta_2})
\right]
\\
& = &
\frac{1}{1+|\eta_{\infty}|^2}
\left[
(1+|\eta_{\infty}|^2)\overline{\zeta_2}d\overline{\zeta_1}
+
(-1-|\eta_{\infty}|^2)\overline{\zeta_1}d\overline{\zeta_2}
\right]
\\
& = &
-\omega'\left(\overline{\zeta}\right)
\,,
\end{eqnarray*}
and
\begin{eqnarray*}
\omega\left(U_{\eta_{\infty}}^{\star}\zeta\right)
\;=\;
\frac{1}{2}\,
\overline{
d\left(
\omega'\left(\overline{U_{\eta_{\infty}}^{\star}\zeta}\right)
\right)
}
\;=\;
-\frac{1}{2}\,
\overline{d\omega'\left(\overline{\zeta}\right)}
\;=\;
-\overline{\omega\left(\overline{\zeta}\right)}
\;=\;
-\omega(\zeta)
\,.
\end{eqnarray*}
One has also, $\forall\,j=1,\ldots,N$,
\begin{eqnarray*}
(U_{\eta_{\infty}}^{\star}\zeta)_1
-\eta_j
(U_{\eta_{\infty}}^{\star}\zeta)_2
& = &
\frac{1}{\sqrt{1+|\eta_{\infty}|^2}}
\left(
(\eta_{\infty}-\eta_j)\zeta_1
+
(1+\overline{\eta_{\infty}}\eta_j)\zeta_2
\right)
\\
& = &
\frac{\eta_{\infty}-\eta_j}{\sqrt{1+|\eta_{\infty}|^2}}
\left(
\zeta_1-\theta_j\zeta_2
\right)
\,.
\end{eqnarray*}
Moreover, since
$U_{\eta_{\infty}}^{\star}\in\mathcal{U}(2,\C)$,
\begin{eqnarray*}
<\overline{U_{\eta_{\infty}}^{\star}\zeta},z>
& = &
<\overline{U_{\eta_{\infty}}^{\star}}\overline{\zeta}\,,
\,U_{\eta_{\infty}}^{\star}U_{\eta_{\infty}}z>
\\
& = &
<\overline{\zeta}\,,\,U_{\eta_{\infty}}z>
\end{eqnarray*}
and
$\forall\,j=1,\ldots,N$,
\begin{eqnarray*}
z_1-\eta_jz_2 
& = &
(U_{\eta_{\infty}}^{\star}U_{\eta_{\infty}}z)_1
-
(U_{\eta_{\infty}}^{\star}U_{\eta_{\infty}}z)_2
\\
& = &
\frac{\eta_{\infty}-\eta_j}{\sqrt{1+|\eta_{\infty}|^2}}
\left(
(U_{\eta_{\infty}}z)_1-\theta_j(U_{\eta_{\infty}}z)_2
\right)
\,.
\end{eqnarray*}
Finally, since
$U_{\eta_{\infty}}^{\star}\left(\mathbb{S}_2\right)
=\mathbb{S}_2$, one has
\begin{eqnarray*}
U_{\eta_{\infty}}(\Omega_{\varepsilon})
& = &
\left(
U_{\eta_{\infty}}^{\star}
\right)^{-1}
(\Omega_{\varepsilon})
\\
& = &
\left\{
\zeta\in\mathbb{S}_2,\;
\left|
\prod_{j=1}^N
\left(
(U_{\eta_{\infty}}^{\star}\zeta)_1-\eta_j(U_{\eta_{\infty}}^{\star}\zeta)_2
\right)
\right|
>\varepsilon
\right\}
\\
& = &
\left\{
\left(
\prod_{j=1}^N
\frac{|\eta_{\infty}-\eta_j|}{\sqrt{1+|\eta_{\infty}|^2}}
\right)
\left|
\prod_{j=1}^N
\left(
\zeta_1-\theta_j\zeta_2
\right)
\right|
>\varepsilon
\right\}
\\
& = &
\left\{
\zeta\in\mathbb{S}_2,\;
\left|
\prod_{j=1}^N
\left(
\zeta_1-\theta_j\zeta_2
\right)
\right|
>\varepsilon/C_N
\right\}
\\
& = &
\Omega'_{\varepsilon/C_N}
\,.
\end{eqnarray*}

It follows from (\ref{prelimrest}) that
\begin{eqnarray*}
R_N(f;\eta)(z)
-\sum_{k+l\geq N}
\frac{1}{k!\,l!}
\frac{\partial^{k+l}f}{\partial z_1^k\partial z_2^l}(0)
z_1^kz_2^l
& = &
\\
& &
\;\;\;\;\;\;\;\;\;\;\;\;\;\;\;\;\;\;\;\;\;\;\;\;\;\;\;\;\;\;\;\;\;\;
\;\;\;\;\;\;\;\;\;\;\;\;\;\;\;\;\;\;\;\;\;
\end{eqnarray*}
\begin{eqnarray*}
& = &
\prod_{j=1}^N(z_1-\eta_jz_2)
\lim_{\varepsilon\rightarrow0}
\frac{1}{(2i\pi)^2}
\int_{\Omega_{\varepsilon}}
\frac{f(\zeta)\,\omega'(\overline{\zeta})\wedge\omega(\zeta)}
{\prod_{j=1}^N(\zeta_1-\eta_j\zeta_2)
\,
(1-<\overline{\zeta},z>)^2}
\\
& = &
C_N
\prod_{j=1}^N
\left(
(U_{\eta_{\infty}}z)_1-\theta_j(U_{\eta_{\infty}}z)_2
\right)
\times\\
& &
\times
\lim_{\varepsilon\rightarrow0}
\frac{1}{(2i\pi)^2}
\int_{\Omega'_{\varepsilon/C_N}}
\frac{f(U_{\eta_{\infty}}^{\star}\zeta)
(-\omega'(\zeta))\wedge(-\omega(\zeta))}
{C_N
\prod_{j=1}^N(\zeta_1-\theta_j\zeta_2)
(1-<\overline{\zeta},U_{\eta_{\infty}}z>)^2}
\\
\\
& = &
\prod_{j=1}^N
\left(
(U_{\eta_{\infty}}z)_1-\theta_j(U_{\eta_{\infty}}z)_2
\right)
\lim_{\varepsilon\rightarrow0}
\frac{1}{(2i\pi)^2}
\int_{\Omega'_{\varepsilon}}
\frac{\left(f\circ U_{\eta_{\infty}}^{\star}\right)(\zeta)\,
\omega'(\zeta)\wedge\omega(\zeta)}
{\prod_{j=1}^N(\zeta_1-\theta_j\zeta_2)
(1-<\overline{\zeta},U_{\eta_{\infty}}z>)^2}
\\
& = &
R_N\left(
f\circ U_{\eta_{\infty}}^{\star};
\theta
\right)
\left(
U_{\eta_{\infty}}z
\right)
-\sum_{k+l\geq N}
\frac{1}{k!\,l!}
\frac{\partial^{k+l}\left(f\circ U_{\eta_{\infty}}^{\star}\right)}
{\partial z_1^k\partial z_2^l}(0)
\left(U_{\eta_{\infty}}z\right)_1^k
\left(U_{\eta_{\infty}}z\right)_2^l
\,.
\end{eqnarray*}

Since
$f,\,f\circ U_{\eta_{\infty}}^{\star}\in\mathcal{O}(\C^2)$
and for any compact subset $K\subset\C^2$,
$U_{\eta_{\infty}}(K)$ is still compact, one has
\begin{eqnarray*}
\sup_{z\in K}
\left|
\sum_{k+l\geq N}
\frac{1}{k!\,l!}
\frac{\partial^{k+l}f}{\partial z_1^k\partial z_2^l}(0)
z_1^kz_2^l
\right|
\xrightarrow[N\rightarrow\infty]{}0\,,
\end{eqnarray*}
and
\begin{eqnarray*}
\sup_{z\in K}
\left|
\sum_{k+l\geq N}
\frac{1}{k!\,l!}
\frac{\partial^{k+l}\left(f\circ U_{\eta_{\infty}}^{\star}\right)}
{\partial z_1^k\partial z_2^l}(0)
\left(U_{\eta_{\infty}}z\right)_1^k
\left(U_{\eta_{\infty}}z\right)_2^l
\right|
\xrightarrow[N\rightarrow\infty]{}0
\,.
\end{eqnarray*}

Therefore
$R_N(f;\eta)$ converges for every function $f\in\mathcal{O}(\C^2)$
if and only if
$R_N\left(
f\circ U_{\eta_{\infty}}^{\star};
\theta
\right)
\left(
U_{\eta_{\infty}}(\cdot)
\right)$ converges for every
$f\in\mathcal{O}(\C^2)$,
then if and only if
$R_N\left(
f;\theta
\right)$
converges for every function
$f\in\mathcal{O}(\C^2)$.
Moreover, this equivalence is true for all
$\eta_{\infty}\notin\overline{\{\eta_j\}_{j\geq1}}\cup\{\infty\}$.
This yields to
(2)$\Rightarrow$(1)$\Rightarrow$(3)
(and (3)$\Rightarrow$(2) since
$\C\setminus\overline{\{\eta_j\}_{j\geq1}}\neq\emptyset$).

\end{proof}

Now we see that if
$|\eta_j|\leq2|\eta_{\infty}|$,
then by~(\ref{nondense})
\begin{eqnarray*}
|\theta_j|
\;=\;
\frac{|1+\overline{\eta_{\infty}}\eta_j|}{|\eta_j-\eta_{\infty}|}
& \leq &
\frac{1+2|\eta_{\infty}|^2}{\varepsilon_{\infty}}
\;<\;
+\infty
\,.
\end{eqnarray*}
Otherwise
$|\eta_j|>2|\eta_{\infty}|$
then
\begin{eqnarray*}
|\theta_j|
\;=\;
\frac{|\overline{\eta_{\infty}}+1/\eta_j|}{|1-\eta_{\infty}/\eta_j|}
& \leq &
\frac{|\eta_{\infty}|+1/(2|\eta_{\infty}|)}{1-1/2}
\;=\;
2\left(
|\eta_{\infty}|+1/(2|\eta_{\infty}|)
\right)
\;<\;
+\infty\,,
\end{eqnarray*}
as long as $\eta_{\infty}\neq0$;
if it is not the case, one has
$|\eta_j|\geq\varepsilon_{\infty}$,
$\forall\,j\geq1$, then
\begin{eqnarray*}
|\theta_j|
\;=\;
1/|\eta_j|
& \leq &
1/\varepsilon_{\infty}
\;<\;
+\infty
\,.
\end{eqnarray*}
It follows that
\begin{eqnarray}\label{thetabounded}
\|\theta\|_{\infty}
\;=\;
\sup_{j\geq1}|\theta_j|
& < &
+\infty
\,.
\end{eqnarray}
On the other hand, if we admit that for
$\eta_{\infty}=\infty$
(ie
$\|\eta\|_{\infty}<+\infty$) one has
\begin{eqnarray*}
\theta_j & := &
\lim_{x\rightarrow\infty,x>0}
\frac{1+\overline{ix}\eta_j}{\eta_j-ix}
\,,
\end{eqnarray*}
then $\theta_j=\eta_j$.
Moreover, if we more generally choose
\begin{eqnarray*}
\theta_j & := &
\lim_{x\rightarrow\infty,x>0}
\frac{1+\overline{e^{i\varphi}x}\eta_j}{\eta_j-e^{i\varphi}x}
\,,
\end{eqnarray*}
then
$\theta_j=-e^{-2i\varphi}\eta_j$
and one still has
\begin{eqnarray}\nonumber
\left|
\Delta_{p,(\theta_p,\ldots,\theta_1)}
\left[
\left(
\frac{\overline{\zeta}}{1+|\zeta|^2}
\right)^q
\right](\theta_{p+1})
\right|
& = &
\left|
\Delta_{p,(-e^{-2i\varphi}\eta_p,\ldots,-e^{-2i\varphi}\eta_1)}
\left[
\left(
\frac{\overline{\zeta}}{1+|\zeta|^2}
\right)^q
\right]
\left(-e^{-2i\varphi}\eta_{p+1}\right)
\right|
\\\label{convetainfini}
& = &
\left|
\Delta_{p,(\eta_p,\ldots,\eta_1)}
\left[
\left(
\frac{\overline{\zeta}}{1+|\zeta|^2}
\right)^q
\right](\eta_{p+1})
\right|
\,.
\end{eqnarray}
This allows us to give the proof of Theorem~\ref{equivnodense}.

\begin{proof}

Let be 
$\eta_{\infty}\notin
\overline{\{\eta_j\}_{j\geq1}}\cup\{\infty\}$
and the associate
$\{\theta_j\}_{j\geq1}$. We know by
Lemma~\ref{lemmextensio} that
$R_N(f;\eta)$ converges for every
$f\in\mathcal{O}\left(\C^2\right)$
if and only if so does
$R_N(f;\theta)$,
$\forall\,f\in\mathcal{O}\left(\C^2\right)$.
On the other hand,
$\{\theta_j\}_{j\geq1}$ being bounded by~(\ref{thetabounded}),
it follows by Theorem~\ref{equivbounded}
that it is true if and only if
$\exists\,R_{\theta}$,
$\forall\,p,q\geq0$,
\begin{eqnarray*}
\left|
\Delta_{p,(\theta_p,\ldots,\theta_1)}
\left[
\left(
\frac{\overline{\zeta}}{1+|\zeta|^2}
\right)^q
\right](\theta_{p+1})
\right|
& \leq &
R_{\theta}^{p+q}
\,.
\end{eqnarray*}

\end{proof}

\begin{remark}

It follows from~(\ref{convetainfini}) that
Theorem~\ref{equivnodense} can be extended in the
case when $\eta_{\infty}=\infty$, so it is
an extension of Theorem~\ref{equivbounded}.

\end{remark}

\subsubsection{Proof of Theorem \ref{superbe}}

We begin with specifying a point about
a real-analytically interpolated set
$\{\eta_j\}_{j\geq1}$.

\begin{remark}

If $\{\eta_j\}_{j\geq1}$ is not bounded,
then $\zeta_0=\infty$ is a limit point. Let be
$(\eta_{j_k})_{k\geq1}$ a subsequence that converges
to $\infty$, then from Definition~\ref{anal},
$g_{\infty}(\infty)=\lim_{k\rightarrow\infty}g_{\infty}(\eta_{j_k})
=\lim_{k\rightarrow\infty}\overline{\eta_{j_k}}=\infty$.
It follows that the associate function
$g_{\infty}$ is holomorphic from a neighborhood of
$\infty$ to a neighborhood of $\infty$, ie the function
\begin{eqnarray*}
\widetilde{g_{\infty}}\;:V_0
& \rightarrow &
V'_0
\\
\zeta & \mapsto &
\begin{cases}
\dfrac{1}{g_{\infty}(1/\zeta)}
\;\mbox{if}\;\zeta\neq0\,,
\\
0
\;\mbox{if}\;\zeta=0
\end{cases}
\end{eqnarray*}
(with $V_0,\,V'_0\in\mathcal{V}(0)$),
is holomorphic.

\end{remark}

In Definition~\ref{anal}, we do not need to assume that
$\{\eta_j\}_{j\geq1}$
is not dense. The following result specifies that
it cannot be the case.

\begin{lemma}\label{analnodense}

Let $\{\eta_j\}_{j\geq1}\subset\C$ be any subset.
If it is real-analytically interpolated, then
it is not dense.

\end{lemma}

\begin{proof}

On the contrary, assume that
$\{\eta_j\}_{j\geq1}$ is dense. Then in particular
$0$ is limit point. Let be
the associate $V\in\mathcal{V}(0)$ and
$g\in\mathcal{O}(V)$. 
Then 
$V\subset\overline{\{\eta_j\}_{j\geq1}}$
and for all $\eta_j$ close to $0$ one has
\begin{eqnarray*}
\frac{\overline{\eta_j}-0}{\eta_j-0}
\;=\;
\frac{\overline{\eta_j}}{\eta_j}
& = &
\frac{g(\eta_j)-g(0)}{\eta_j-0}
\;\xrightarrow[\eta_j\rightarrow0]{}\;
\frac{\partial g}{\partial\zeta}(0)
\,.
\end{eqnarray*}
In particular
$\left|
\dfrac{\partial g}{\partial\zeta}(0)
\right|=1$ then one has
$\dfrac{\partial g}{\partial\zeta}(0)
=e^{i\theta}$.
We set
\begin{eqnarray*}
w_p & = &
\frac{1}{p}\,i\,e^{-i\theta/2}
\end{eqnarray*}
with $p$ large enough such that
$w_p\in V$. Let 
$(\eta_{j_p})_{p\geq p_0}$
be a subsequence of
$\{\eta_j\}_{j\geq1}$ such that,
for all $p\geq p_0$,
$w_p\in V$ and
\begin{eqnarray*}
\left|
\eta_{j_p}-w_p
\right|
& \leq &
\frac{1}{2p^2}
\,.
\end{eqnarray*}
Then (since
$(\eta_{j_p})_{p\geq p_0}$ converges to $0$)
\begin{eqnarray*}
e^{i\theta}
\;=\;
\frac{\partial g}{\partial\zeta}(0)
& = &
\lim_{p\rightarrow\infty}
\frac{\overline{\eta_{j_p}}}{\eta_{j_p}}
\\
& = &
\lim_{p\rightarrow\infty}
\frac{\overline{w_p}+O(1/p^2)}{w_p+O(1/p^2)}
\;=\;
\lim_{p\rightarrow\infty}
\frac{-ie^{i\theta/2}+O(1/p)}{ie^{-i\theta/2}+O(1/p)}
\;=\;
-e^{i\theta}
\,,
\end{eqnarray*}
and that is impossible.

\end{proof}

The following result will be usefull for the proof of
Theorem~\ref{superbe}.

\begin{lemma}\label{equivanal}

Assume that
$\overline{\{\eta_j\}_{j\geq1}}$ is
not bounded and
real-analytically interpolated. Then
for all
$\eta_{\infty}\notin\overline{\{\eta_j\}_{j\geq1}}$
($\neq\C$),
the associate bounded subset
$\{\theta_j\}_{j\geq1}$ is real-analytically interpolated.

\end{lemma}

\begin{proof}

Let be any
$\eta_{\infty}\notin\overline{\{\eta_j\}_{j\geq1}}$
(such an element exists by Lemma~\ref{analnodense}).
Necessarily $\eta_{\infty}\neq\infty$.
By definition,
$\theta_j=h(\eta_j)$, where
\begin{eqnarray}\label{defh}
h:\;\overline{\C} & \rightarrow & \overline{\C}
\\\nonumber
z & \mapsto & 
\frac{\overline{\eta_{\infty}}z+1}{z-\eta_{\infty}}
\,,
\end{eqnarray}
is homographic. In particular, $h$ is bicontinuous
from $\overline{\C}$ to $\overline{\C}$ then
there is a correspondence between the limit
(resp. isolated) points of
$\{\eta_j\}_{j\geq1}$ and the ones of
$\{\theta_j\}_{j\geq1}$.

Now let be
$w_0\in\overline{\{\theta_j\}_{j\geq1}}$.
$\exists\,!\,\zeta_0\in\overline{\{\eta_j\}_{j\geq1}}$,
$w_0=h(\zeta_0)$. One has
$\zeta_0=\infty$ if and only if
$w_0=\overline{\eta_{\infty}}$.
First, assume that
$\zeta_0\neq\infty$ and let be the associate
$V_{\zeta_0}\in\mathcal{V}(\zeta_0)$,
$g_{\zeta_0}\in\mathcal{O}(V_{\zeta_0})$.
In particular,
$g_{\zeta_0}(\zeta_0)=\overline{\zeta_0}$.
One has, for all $\eta_j\in V_{\zeta_0}$,
\begin{eqnarray*}
\overline{\theta_j}
\;=\;
\overline{
\left(
\frac{1+\overline{\eta_{\infty}}\eta_j}
{\eta_j-\eta_{\infty}}
\right)
}
\;=\;
\frac{1+\eta_{\infty}\overline{\eta_j}}
{\overline{\eta_j}-\overline{\eta_{\infty}}}
\;=\;
\frac{1+\eta_{\infty}g_{\zeta_0}(\eta_j)}
{g_{\zeta_0}(\eta_j)-\overline{\eta_{\infty}}}
\,.
\end{eqnarray*}
On the other hand, one has
$\eta_j=h^{-1}(\theta_j)$, with
\begin{eqnarray*}
h^{-1}:\;\C
& \rightarrow &
\C
\\
w & \mapsto &
\frac{\eta_{\infty}w+1}{w-\overline{\eta_{\infty}}}
\,.
\end{eqnarray*}
Notice that
$h^{-1}(w)=\overline{h\left(\overline{w}\right)}$,
then
\begin{eqnarray*}
\overline{\theta_j}
& = &
h^{-1}
\left(
g_{\zeta_0}(\eta_j)
\right)
\,.
\end{eqnarray*}
Finally, for all
$\theta_j\in\left(h^{-1}\right)^{-1}
(V_{\zeta_0})=h(V_{\zeta_0})$,
\begin{eqnarray}\label{interthetafini}
\overline{\theta_j} & = &
h^{-1}
\left[
g_{\zeta_0}
\left(h^{-1}(\theta_j)\right)
\right]
\,.
\end{eqnarray}
Since
$dist\left(
\overline{\{\eta_j\}_{j\geq1}},\eta_{\infty}
\right)\geq\varepsilon_{\infty}$
and
$w_0\neq\overline{\eta_{\infty}}$,
then
$\exists\,V_{w_0}\in\mathcal{V}(w_0)$,
$\forall\,w\in V_{w_0}$,
$w\neq\overline{\eta_{\infty}}$.
One the other hand,
$g_{\zeta_0}(\zeta_0)=\overline{\zeta_0}\neq\overline{\eta_{\infty}}$,
then by reducing $V_{\zeta_0}$ if necessary, one can assume that,
$\forall\,\zeta\in V_{\zeta_0}$,
$g_{\zeta_0}(\zeta)\neq\overline{\eta_{\infty}}$.
Finally, $\exists\,W_{w_0}$,
$\forall\,w\in W_{w_0}$,
$h^{-1}(w)\in V_{\zeta_0}$.
This allows us to define 
\begin{eqnarray*}\label{defgw0fini}
g_{w_0}\;:\;W_{w_0}
& \rightarrow & \C
\\\nonumber
w & \mapsto &
h^{-1}
\left[
g_{\zeta_0}
\left(
h^{-1}(w)
\right)
\right]
\,,
\end{eqnarray*}
the composed function of
\begin{eqnarray*}
\bordermatrix{ & & & \cr
& W_{w_0}  & \rightarrow & V_{\zeta_0} \cr
& w  & \mapsto & h^{-1}(w) \cr}
,
\bordermatrix{ & & & \cr
& V_{\zeta_0}  & \rightarrow & 
\C\setminus\{\overline{\eta_{\infty}}\} \cr
& \zeta  & \mapsto & g_{\zeta_0}(\zeta) \cr}
\mbox{ and}
\bordermatrix{ & & & \cr
& \C\setminus\{\overline{\eta_{\infty}}\} 
& \rightarrow & \C \cr
& u  & \mapsto & h^{-1}(u) \cr}
.
\end{eqnarray*}
It follows that $g_{w_0}$ is holomorphic and
satisfies by~(\ref{interthetafini}): 
$\forall\,\theta_j\in W_{w_0}$,
$g_{w_0}(\theta_j)=\overline{\theta_j}$.
\bigskip

Now assume that
$\zeta_0=\infty$ (then
$w_0=\overline{\eta_{\infty}}$) and let be the associate
$V_{\infty}\in\mathcal{V}(\infty)$,
$g_{\infty}:V_{\infty}\rightarrow\overline{\C}$,
holomorphic that satisfies:
$\forall\,\eta_j\in V_{\infty}$,
$\overline{\eta_j}=g_{\infty}(\eta_j)$.
On the other hand, let be
$W_0,V_0\in\mathcal{V}(0)$ such that
\begin{eqnarray}
\widetilde{g_{\infty}}
\;:\;W_{0}
& \rightarrow & V_{0}
\\\nonumber
\zeta & \mapsto &
\begin{cases}
\dfrac{1}{g_{\infty}(1/\zeta)}
\mbox{ if }
\zeta\neq0,
\\
0\mbox{ if }
\zeta=0,
\end{cases}
\end{eqnarray}
is holomorphic. By reducing
$V_0$ and $W_0$ if necessary,
one can assume that,
$\forall\,\zeta\in W_0\setminus\{0\}$,
$1/\zeta\in V_{\infty}$ and
$\forall\,w\in V_0$,
$1-\overline{\eta_{\infty}}\,w\neq0$.

Now let be
$W_{\infty}\in\mathcal{V}(\infty)$
such that
$W_{\infty}\subset V_{\infty}$,
$0\notin W_{\infty}$ and $\forall\,\zeta\in W_{\infty}$,
$1/\zeta\in W_0$.
Since we still have,
$\forall\,\eta_j\in W_{\infty}\subset V_{\infty}$,
\begin{eqnarray*}
\overline{\theta_j}
\;=\;
\frac{1+\eta_{\infty}g_{\infty}(\eta_j)}
{g_{\infty}(\eta_j)-\overline{\eta_{\infty}}}
\;=\;
\frac{\eta_{\infty}+1/g_{\infty}(\eta_j)}
{1-\overline{\eta_{\infty}}/g_{\infty}(\eta_j)}
\;=\;
\frac{\eta_{\infty}+\widetilde{g_{\infty}}(1/\eta_j)}
{1-\overline{\eta_{\infty}}\,\widetilde{g_{\infty}}(1/\eta_j)}
\,,
\end{eqnarray*}
and
\begin{eqnarray*}
\eta_j & = &
h^{-1}(\theta_j)
\;=\;
\frac{\eta_{\infty}\theta_j+1}
{\theta_j-\overline{\eta_{\infty}}}
\,,
\end{eqnarray*}
we get, for all
$\theta_j\in h(V_{\infty})$,
\begin{eqnarray}\label{interthetainfini}
\overline{\theta_j} & = &
h_{\infty}
\left[
\widetilde{g_{\infty}}
\left(
1/h^{-1}(\theta_j)
\right)
\right]
\,,
\end{eqnarray}
with
\begin{eqnarray}
h_{\infty}\;:\;
\overline{\C}
& \rightarrow & 
\overline{\C}
\\\nonumber
w & \mapsto &
\frac{\eta_{\infty}+w}
{1-\overline{\eta_{\infty}}\,w}
\,.
\end{eqnarray}
In particular, the restriction
$h_{\infty}:V_0\rightarrow\C$,
is still holomorphic.
If we choose 
$V_{\overline{\eta_{\infty}}}
\in\mathcal{V}(\overline{\eta_{\infty}})$
such that
$V_{\overline{\eta_{\infty}}}\subset h(V_{\infty})$
and,
$\forall\,w\in V_{\overline{\eta_{\infty}}}$ ,
one has
$1/h^{-1}(w)=
\dfrac{w-\overline{\eta_{\infty}}}{\eta_{\infty}w+1}\in W_0$
and
$h(w)=
\dfrac{\eta_{\infty}w+1}{w-\overline{\eta_{\infty}}}\in W_{\infty}$,
the function
\begin{eqnarray}
g_{\overline{\eta_{\infty}}}
\;:\;
V_{\overline{\eta_{\infty}}}
& \rightarrow &
\C
\\\nonumber
w & \mapsto &
\begin{cases}
h_{\infty}
\left[
\widetilde{g_{\infty}}
\left(
1/h^{-1}(w)
\right)
\right]
\mbox{ if }
w\neq\overline{\eta_{\infty}},
\\
h_{\infty}(\widetilde{g_{\infty}}(0))=\eta_{\infty}
\mbox{ if }
w=\overline{\eta_{\infty}},
\end{cases}
\end{eqnarray}
is holomorphic and satisfies by~(\ref{interthetainfini}):
$\forall\,\theta_j\in V_{\overline{\eta_{\infty}}}$,
$g_{\overline{\eta_{\infty}}}(\theta_j)=\overline{\theta_j}$.

\end{proof}

\begin{remark}

We could also prove that the assertion is reciprocal but
it will not be usefull for the result that we want to prove.

\end{remark}

Now we can give the proof of Theorem~\ref{superbe}.

\begin{proof}

Let
$\{\eta_j\}_{j\geq1}$ be a real-analytically interpolated
subset. If it is bounded, then Theorem~\ref{superbe} follows
by Section~\ref{proofs}, Subsection~\ref{proofsuperbounded}.
Otherwise, we know by Lemma~\ref{analnodense} that
$\{\eta_j\}_{j\geq1}$ cannot be dense, then there is
$\eta_{\infty}\notin\overline{\{\eta_j\}_{j\geq1}}
\cup\{\infty\}$.
Let consider the associate bounded subset
$\{\theta_j\}_{j\geq1}$. Thus by Lemma~\ref{equivanal},
$\{\theta_j\}_{j\geq1}$ is bounded and real-analytically interpolated.
It follows by Subsection~\ref{proofsuperbounded} that
$R_N(f;\theta)$ converges for every function
$f\in\mathcal{O}\left(\C^2\right)$.
Finally, by Lemma~\ref{lemmextensio},
$R_N(f;\eta)$ (then $E_N(f;\eta)$) converges for every
$f\in\mathcal{O}\left(\C^2\right)$,
and the theorem is proved.

\end{proof}

\bigskip

\section{A counterexample}\label{ctrex}

Let consider the subset
$\R\cup i\R$. It is a union of real-analytical manifolds of $\C$
($\mathbb{R}$ and $i\mathbb{R}$), but is not a manifold
(problem on $0$).
Here we will deal with which can happen when the sequence
$(\eta_j)_{j\geq1}$ converges to $0$
without staying in one or the other
line (see Introduction)
and show the following result.

\begin{proposition}\label{ctrexample}

There exists a (bounded) subset
$\{\eta_j\}_{j\geq1}\subset\R\cup i\R$
that does not satisfy the condition~(\ref{criter}).
It follows by Theorem~\ref{equivbounded} that
there is (at least) a function
$f\in\mathcal{O}\left(\C^2\right)$ such that
the formula $E_N(f;\eta)$ cannot converge.

\end{proposition}

We begin with the following result that gives the
construction in a more general case.

\begin{lemma}\label{constructctrex}

Let $f$ be a function of class
$C^2$ in a neighborhood $V$ of $0$ and that is not
$\C$-differentiable on $0$, ie
\begin{eqnarray*}
\frac{\partial f}{\partial\overline{\zeta}}(0)
& \neq &
0\,.
\end{eqnarray*}
Then there exists a bounded sequence
$(\eta_j)_{j\geq1}\subset\R\cup i\R$ such that,
for all $p\geq1$,
\begin{eqnarray*}
\left|
\Delta_{3p-1,(\eta_{3p-1},\ldots,\eta_1)}
(\zeta\mapsto f(\zeta))
(\eta_{3p})
\right|
& \geq &
p^p\,.
\end{eqnarray*}

\end{lemma}

\begin{proof}

For all differents and nonzero
$\eta_1,\ldots,\eta_{p}\in\R\cup i\R$, the function
$\zeta\mapsto\Delta_{p,(\eta_p,\ldots,\eta_1)}(f)(\zeta)$
is still of class $C^2$ on
$V\setminus\{\eta_1,\ldots,\eta_p\}$.
Moreover,
\begin{eqnarray*}
\frac{\partial}{\partial\overline{\zeta}}
\left[\Delta_p(f)\right](\zeta)
& = &
\frac{\partial}{\partial\overline{\zeta}}
\left[
\frac{
\Delta_{p-1}(f)(\zeta)-\Delta_{p-1}(f)(\eta_p)}
{\zeta-\eta_p}
\right]
\\
& = &
\frac{1}{\zeta-\eta_p}
\frac{\partial}{\partial\overline{\zeta}}
\left[\Delta_{p-1}(f)\right](\zeta)
\;=\;
\cdots
\\
& = &
\frac{1}{(\zeta-\eta_p)(\zeta-\eta_{p-1})
\cdots
(\zeta-\eta_1)}
\frac{\partial}{\partial\overline{\zeta}}
\left[\Delta_{0}(f)\right](\zeta)
\\
& = &
\frac{
\dfrac{\partial f}{\partial\overline{\zeta}}
(\zeta)}
{(\zeta-\eta_p)
\cdots
(\zeta-\eta_1)}
\,.
\end{eqnarray*}
Now let be $\eta_1,\ldots,\eta_{3p-1},\eta_{3p}$
all different and nonzero, and let be
$\eta_{3p+1}\in(\R\cup i\R)\cap(V\setminus\{0,\eta_1,\ldots,\eta_{3p}\})$.
Since
$\dfrac{\partial f}{\partial\overline{\zeta}}(0)\neq0$,
one has
\begin{eqnarray*}
\frac{\partial}{\partial\overline{\zeta}}
\left[\Delta_{3p+1}(f)\right](0)
& = &
\frac{
\dfrac{\partial f}{\partial\overline{\zeta}}
(0)}
{(-\eta_{3p+1})(-\eta_{3p})
\cdots
(-\eta_1)}
\\
& = &
\frac{-1}{\eta_1\cdots\eta_{3p}\,\eta_{3p+1}}
\;
\frac{\partial f}{\partial\overline{\zeta}}
(0)
\;\xrightarrow[\eta_{3p+1}\rightarrow0]{}\;
\infty
\,.
\end{eqnarray*}
Let fix
$\eta_{3p+1}\in\R\cup i\R$ with 
$0<|\eta_{3p+1}|<\min_{1\leq i\leq3p}|\eta_i|$
(then 
$\eta_{3p+1}\in V\setminus\{0,\eta_1,\ldots,\eta_{3p}\}$)
and such that
\begin{eqnarray}\label{grandbarre}
\left|
\frac{\partial}{\partial\overline{\zeta}}
\left[\Delta_{3p+1}(f)\right](0)
\right|
& \geq &
(p+1)^{p+1}+1\,.
\end{eqnarray}
Now for all
$\eta_{3p+2},\eta_{3p+3}\in\R\cup i\R$,
different and such that
$0<|\eta_{3p+2}|,|\eta_{3p+2}|<\min_{1\leq i\leq3p+1}|\eta_i|$,
one has
\begin{eqnarray*}
\Delta_{3p+2,(\eta_{3p+2},\eta_{3p+1},\ldots,\eta_1)}
(f)(\eta_{3p+3})
& = &
\;\;\;\;\;\;\;\;\;\;\;\;\;\;\;\;\;\;\;\;\;\;\;\;\;\;\;\;\;\;\;\;\;\;
\;\;\;\;\;\;\;\;\;\;\;\;\;\;\;\;\;\;\;\;\;\;\;\;\;\;\;\;\;\;\;\;\;\;\;
\end{eqnarray*}
\begin{eqnarray*}
& = &
\frac{
\Delta_{3p+1,(\eta_{3p+1},\ldots,\eta_1)}
(f)(\eta_{3p+3})
-
\Delta_{3p+1,(\eta_{3p+1},\ldots,\eta_1)}
(f)(\eta_{3p+2})
\pm\Delta_{3p+1,(\eta_{3p+1},\ldots,\eta_1)}
(f)(0)
}
{\eta_{3p+3}-\eta_{3p+2}}
\\
& = &
\frac{1}{\eta_{3p+3}-\eta_{3p+2}}
\left[
\eta_{3p+3}
\frac{\partial}{\partial\zeta}
\Delta_{3p+1}(f)(0)
+
\overline{\eta_{3p+3}}
\frac{\partial}{\partial\overline{\zeta}}
\Delta_{3p+1}(f)(0)
+
O(\eta_{3p+3}^2)
\right]
\\
& &
-\;
\frac{1}{\eta_{3p+3}-\eta_{3p+2}}
\left[
\eta_{3p+2}
\frac{\partial}{\partial\zeta}
\Delta_{3p+1}(f)(0)
+
\overline{\eta_{3p+2}}
\frac{\partial}{\partial\overline{\zeta}}
\Delta_{3p+1}(f)(0)
+
O(\eta_{3p+2}^2)
\right]
\\
& = &
\frac{\partial}{\partial\zeta}
\Delta_{3p+1}(f)(0)
+
\frac{\overline{\eta_{3p+3}}-\overline{\eta_{3p+2}}}
{\eta_{3p+3}-\eta_{3p+2}}
\frac{\partial}{\partial\overline{\zeta}}
\Delta_{3p+1}(f)(0)
+\frac{O(\eta_{3p+3}^2)+O(\eta_{3p+2}^2)}{\eta_{3p+3}-\eta_{3p+2}}
\,.
\end{eqnarray*}
On the other hand, since
$\Delta_{3p+1}(f)(\zeta)$ is of class $C^2$ on
$V\setminus\{\eta_1,\ldots,\eta_{3p+1}\}$, there is
$\varepsilon_{p+1}>0$ small enough with
$0<\varepsilon_{p+1}<\min_{1\leq i\leq3p+1}|\eta_i|$,
$\overline{D(0,\varepsilon_{p+1})}\subset V$,
such that
\begin{eqnarray*}
M_{p+1}
& := &
\sup_{0<|\zeta|\leq\varepsilon_{p+1}}
\left|
\frac{O\left(\zeta^2\right)}
{\zeta^2}
\right|
\;<\;
+\infty
\end{eqnarray*}
($\varepsilon_{p+1}$ and $M_{p+1}$ will only depend on
$f$ and $\eta_1,\ldots,\eta_{3p+1}$).

Now we assume that 
$\dfrac{\partial}{\partial\zeta}
\Delta_{3p+1}(f)(0)\neq0$.
Then
\begin{eqnarray*}
\frac{\dfrac{\partial}{\partial\zeta}
\Delta_{3p+1}(f)(0)}
{\dfrac{\partial}{\partial\overline{\zeta}}
\Delta_{3p+1}(f)(0)}
\;=\;
\frac{
\left|
\dfrac{\partial}{\partial\zeta}
\Delta_{3p+1}(f)(0)
\right|
e^{i\varphi_1}
}
{
\left|
\dfrac{\partial}{\partial\overline{\zeta}}
\Delta_{3p+1}(f)(0)
\right|
e^{i\varphi_2}
}
\;=\;
\left|
\frac{
\dfrac{\partial}{\partial\zeta}
\Delta_{3p+1}(f)(0)
}
{
\dfrac{\partial}{\partial\overline{\zeta}}
\Delta_{3p+1}(f)(0)
}
\right|
e^{i\theta}
\,.
\end{eqnarray*}
Three cases can happen.

\begin{enumerate}

\item

If $e^{i\theta}=1$,
we choose
$\eta_{3p+2},\eta_{3p+3}\in\R$ with
$\eta_{3p+3}=-\eta_{3p+2}$
and
$0<|\eta_{3p+2}|=|\eta_{3p+3}|\leq
\min(\varepsilon_{p+1},1/M_{p+1})$.
One has
\begin{eqnarray*}
\left|
\frac{O(\eta_{3p+3}^2)+O(\eta_{3p+2}^2)}
{\eta_{3p+3}-\eta_{3p+2}}
\right|
& \leq &
\frac{2M_{p+1}\,\eta_{3p+2}^2}{2|\eta_{3p+2}|}
\;=\;
M_{p+1}|\eta_{3p+2}|
\;\leq\;1
\,,
\end{eqnarray*}
then by~(\ref{grandbarre})
\begin{eqnarray*}
\left|
\Delta_{3p+2}(f)(\eta_{3p+3})
\right|
& \geq &
\left|
\frac{\partial}{\partial\zeta}
\Delta_{3p+1}(f)(0)
+
\frac{\overline{\eta_{3p+3}}-\overline{\eta_{3p+2}}}
{\eta_{3p+3}-\eta_{3p+2}}
\frac{\partial}{\partial\overline{\zeta}}
\Delta_{3p+1}(f)(0)
\right|
\\
& &
\;\;\;\;\;\;\;\;\;\;\;\;\;\;\;\;\;\;\;\;\;\;\;\;\;\;\;\;\;\;\;
\;\;\;\;\;\;\;\;\;\;\;\;\;\;\;\;
-\;
\left|
\frac{O(\eta_{3p+3}^2)+O(\eta_{3p+2}^2)}{\eta_{3p+3}-\eta_{3p+2}}
\right|
\\
& \geq &
\left|
\frac{\partial}{\partial\zeta}
\Delta_{3p+1}(f)(0)
+
\frac{\partial}{\partial\overline{\zeta}}
\Delta_{3p+1}(f)(0)
\right|
-1
\\
& = &
\left|
\left|
\frac{
\dfrac{\partial}{\partial\zeta}
\Delta_{3p+1}(f)(0)
}
{
\dfrac{\partial}{\partial\overline{\zeta}}
\Delta_{3p+1}(f)(0)
}
\right|
\frac{\partial}{\partial\overline{\zeta}}
\Delta_{3p+1}(f)(0)
+
\frac{\partial}{\partial\overline{\zeta}}
\Delta_{3p+1}(f)(0)
\right|
-1
\\
& \geq &
\left|
\frac{\partial}{\partial\overline{\zeta}}
\Delta_{3p+1}(f)(0)
\right|
-1
\;\geq\;
(p+1)^{p+1}
\,.
\end{eqnarray*}

\item

If $e^{i\theta}=-1$,
we choose
$\eta_{3p+2},\eta_{3p+3}\in\,i\R$ with
$\eta_{3p+3}=-\eta_{3p+2}$,
$0<|\eta_{3p+2}|=|\eta_{3p+3}|\leq
\min(\varepsilon_{p+1},1/M_{p+1})$,
such that
\begin{eqnarray*}
\left|
\frac{O(\eta_{3p+3}^2)+O(\eta_{3p+2}^2)}
{\eta_{3p+3}-\eta_{3p+2}}
\right|
& \leq &
M_{p+1}|\eta_{3p+2}|
\;\leq\;1
\,,
\end{eqnarray*}
then by~(\ref{grandbarre})
\begin{eqnarray*}
\left|
\Delta_{3p+2}(f)(\eta_{3p+3})
\right|
& \geq &
\left|
\frac{\partial}{\partial\zeta}
\Delta_{3p+1}(f)(0)
+
\frac{\overline{\eta_{3p+3}}-\overline{\eta_{3p+2}}}
{\eta_{3p+3}-\eta_{3p+2}}
\frac{\partial}{\partial\overline{\zeta}}
\Delta_{3p+1}(f)(0)
\right|
-1
\\
& = &
\left|
\frac{\partial}{\partial\zeta}
\Delta_{3p+1}(f)(0)
-
\frac{\partial}{\partial\overline{\zeta}}
\Delta_{3p+1}(f)(0)
\right|
-1
\\
& = &
\left|
-
\left|
\frac{
\dfrac{\partial}{\partial\zeta}
\Delta_{3p+1}(f)(0)
}
{
\dfrac{\partial}{\partial\overline{\zeta}}
\Delta_{3p+1}(f)(0)
}
\right|
\frac{\partial}{\partial\overline{\zeta}}
\Delta_{3p+1}(f)(0)
-
\frac{\partial}{\partial\overline{\zeta}}
\Delta_{3p+1}(f)(0)
\right|
-1
\\
& \geq &
\left|
\frac{\partial}{\partial\overline{\zeta}}
\Delta_{3p+1}(f)(0)
\right|
-1
\;\geq\;
(p+1)^{p+1}
\,.
\end{eqnarray*}

\item

Otherwise
$e^{i\theta}\neq\pm1$,
ie
$e^{i\theta}=\cos\theta+i\sin\theta$
with
$\sin\theta\neq0$. We choose
\begin{eqnarray*}
\begin{cases}
\eta_{3p+2}:=r_{p+1}\cos(\theta/2)
\\
\eta_{3p+3}:=ir_{p+1}\sin(\theta/2)
\end{cases}
,
\end{eqnarray*}
with 
$0<r_{p+1}\leq\min(\varepsilon_{p+1},1/M_{p+1})$.
Since
$\theta/2\neq0\;(\mbox{mod}\;\pi/2)$
then
$\eta_{3p+2},\eta_{3p+3}$ are nonzero and different.
One has
\begin{eqnarray*}
\left|
\frac{O(\eta_{3p+3}^2)+O(\eta_{3p+2}^2)}
{\eta_{3p+3}-\eta_{3p+2}}
\right|
& \leq &
\frac{
M_{p+1}(r_{p+1}\cos(\theta/2))^2
+
M_{p+1}(r_{p+1}\sin(\theta/2))^2
}
{|r_{p+1}\cos(\theta/2)-ir_{p+1}\sin(\theta/2)|}
\\
& = &
\frac{M_{p+1}r_{p+1}^2}{r_{p+1}}
\;=\;
M_{p+1}r_{p+1}
\;\leq\;
1\,.
\end{eqnarray*}
On the other hand,
\begin{eqnarray*}
\frac{
\overline{\eta_{3p+3}}
-
\overline{\eta_{3p+2}}
}
{\eta_{3p+3}-\eta_{3p+2}}
& = &
\frac{
-ir_{p+1}\sin(\theta/2)
-
r_{p+1}\cos(\theta/2)
}
{
ir_{p+1}\sin(\theta/2)
-r_{p+1}\cos(\theta/2)
}
\\
& = &
\frac{\cos(\theta/2)+i\sin(\theta/2)}
{\cos(\theta/2)-i\sin(\theta/2)}
\;=\;
e^{i\theta}
\,.
\end{eqnarray*}
It follows by~(\ref{grandbarre}) that
\begin{eqnarray*}
\left|
\Delta_{3p+2}(f)(\eta_{3p+3})
\right|
& \geq &
\left|
\frac{\partial}{\partial\zeta}
\Delta_{3p+1}(f)(0)
+
\frac{\overline{\eta_{3p+3}}-\overline{\eta_{3p+2}}}
{\eta_{3p+3}-\eta_{3p+2}}
\frac{\partial}{\partial\overline{\zeta}}
\Delta_{3p+1}(f)(0)
\right|
-1
\\
& = &
\left|
\left|
\frac{
\dfrac{\partial}{\partial\zeta}
\Delta_{3p+1}(f)(0)
}
{
\dfrac{\partial}{\partial\overline{\zeta}}
\Delta_{3p+1}(f)(0)
}
\right|
e^{i\theta}
\frac{\partial}{\partial\overline{\zeta}}
\Delta_{3p+1}(f)(0)
+
e^{i\theta}
\frac{\partial}{\partial\overline{\zeta}}
\Delta_{3p+1}(f)(0)
\right|
-1
\\
& \geq &
\left|
\frac{\partial}{\partial\overline{\zeta}}
\Delta_{3p+1}(f)(0)
\right|
-1
\;\geq\;
(p+1)^{p+1}
\,.
\end{eqnarray*}

\end{enumerate}

Now if
$\dfrac{\partial}{\partial\zeta}
\Delta_{3p+1}(f)(0)=0$,
one still has for example
\begin{eqnarray*}
\dfrac{\partial}{\partial\zeta}
\Delta_{3p+1}(f)(0)
& = &
\left|
\frac{
\dfrac{\partial}{\partial\zeta}
\Delta_{3p+1}(f)(0)
}
{
\dfrac{\partial}{\partial\overline{\zeta}}
\Delta_{3p+1}(f)(0)
}
\right|
\frac{\partial}{\partial\overline{\zeta}}
\Delta_{3p+1}(f)(0)
\,,
\end{eqnarray*}
and we go back to the first case that is already done.

Finally, there are
$\eta_{3p+1},\,\eta_{3p+2},\,\eta_{3p+3}\in\R\cup i\R$,
different with
$0<|\eta_{3p+1}|,\,|\eta_{3p+2}|,\,|\eta_{3p+3}|
<\min_{1\leq i\leq3p}|\eta_i|$ and such that
\begin{eqnarray}\label{induction}
\left|
\Delta_{3p+2,(\eta_{3p+2},\ldots,\eta_1)}
(f)(\eta_{3p+3})
\right|
& \geq &
(p+1)^{p+1}
\,.
\end{eqnarray}
\bigskip

This allows us to construct the sequence
$(\eta_j)_{j\geq1}=(\eta_{3p+1},\eta_{3p+2},\eta_{3p+3})_{p\geq0}$
by induction on $p\geq0$. We begin with
$\eta_1,\eta_2,\eta_3\in(\R\cup i\R)\cap V$, different
and such that
\begin{eqnarray*}
\left|
\Delta_{2,(\eta_{2},\eta_1)}
(f)(\eta_{3})
\right|
& \geq &
1
\,.
\end{eqnarray*}
Now if we assume having constructed
$\eta_1,\ldots,\eta_{3p-1},\eta_{3p}$ such that,
$\forall\,j=1,\ldots,p$,
\begin{eqnarray*}
\left|
\Delta_{3j-1,(\eta_{3j-1},\ldots,\eta_1)}
(f)(\eta_{3j})
\right|
& \geq &
j^j
\,,
\end{eqnarray*}
we can construct
$\eta_{3p+1},\eta_{3p+2}$ and $\eta_{3p+3}$, different,
distinct from
$\eta_1,\ldots,\eta_{3p}$ and that satisfy
by~(\ref{induction})
\begin{eqnarray*}
\left|
\Delta_{3p+2,(\eta_{3p+2},\ldots,\eta_1)}
(f)(\eta_{3p+3})
\right|
& \geq &
(p+1)^{p+1}
\,,
\end{eqnarray*}
and this proves the induction. On the other hand,
the sequence $(\eta_j)_{j\geq1}$ is bounded,
thus the proof of the lemma is achieved.

\end{proof}

\begin{remark}\label{ctrexcv}

We can also construct $(\eta_j)_{j\geq1}$ such that it converges
to $0$. Indeed, in the proof, we can also assume that
$0<|\eta_{3p+1}|<
\min\left(\min_{1\leq i\leq3p}|\eta_i|,\,1/(3p+1)\right)$. Since
by construction
$0<|\eta_{3p+2}|,|\eta_{3p+3}|<|\eta_{3p+1}|$, then
$\lim_{j\rightarrow\infty}\eta_j=0$.

\end{remark}

Now we can give the proof of Proposition~\ref{ctrexample}.

\begin{proof}

We consider
\begin{eqnarray*}
f(\zeta) & = &
\frac{\overline{\zeta}}{1+|\zeta|^2}
\;=\;
\frac{\overline{\zeta}}{1+\zeta\overline{\zeta}}
\,.
\end{eqnarray*}
Then
$f\in C^{\infty}(\C)$ and
$\forall\,\zeta\in\C$,
\begin{eqnarray*}
\frac{\partial f}{\partial\overline{\zeta}}
(\zeta)
& = &
\frac{1+|\zeta|^2-\overline{\zeta}\,\zeta}
{(1+|\zeta|^2)^2}
\;=\;
\frac{1}{(1+|\zeta|^2)^2}
\,.
\end{eqnarray*}
In particular,
$\dfrac{\partial f}{\partial\overline{\zeta}}
(0)\neq0$,
then $f$ satisfies the conditions of Lemma~\ref{constructctrex}.
It follows that there is a (bounded) sequence
$(\eta_j)_{j\geq1}\subset\R\cup i\R$ that satisfies:
$\forall\,p\geq1$,
\begin{eqnarray*}
\left|
\Delta_{3p-1,(\eta_{3p-1},\ldots,\eta_1)}
\left(
\zeta\mapsto
\frac{\overline{\zeta}}{1+|\zeta|^2}
\right)
(\eta_{3p})
\right|
& \geq &
p^p
\,.
\end{eqnarray*}
Therefore the condition~(\ref{criter})
of Theorem~\ref{equivbounded} is not satisfied
with $\{\eta_j\}_{j\geq1}$ since
there cannot exist
$R\geq1$ such that, for all
$p\geq1$, one has
\begin{eqnarray*}
p^p & \leq &
R^{3p-1+1}
\;=\;
\left(R^3\right)^p
\,.
\end{eqnarray*}

\end{proof}


\begin{thebibliography}{00}




\bibitem{alabut} C. Alabiso, P. Butera, 
$N$-variable rational approximants and method of moments,
{\em J. Mathematical Phys.} {\bf 16} (1975), 840--845. 



\bibitem{berndtsson} B. Berndtsson,
A formula for interpolation and division in $\mathbb{C}^n$,
{\em Math. Ann.} {\bf 263} (1983), 399--418. 




\bibitem{colherr} N. Coleff, M. Herrera,
Les courants r{\'e}siduels associ{\'e}s {\`a} une forme m{\'e}romorphe (French),
{\em Lecture Notes in Mathematics}, {\bf 633}, Springer, Berlin (1978).



\bibitem{hensha1} G.M. Henkin, A.A. Shananin,
Bernstein theorems and Radon transform. Application to the theory
of production functions, {\em Transl. Math. Monogr.} {\bf 81}
(1990), 189--223.




\bibitem{hensha2} G.M. Henkin, A.A. Shananin,
$\mathbb{C}^n$-capacity and multidimensional moment problem, 
{\em Notre Dame Math. Lectures} {\bf 12} (1992), 69--85.




\bibitem{herrlieb} M. Herrera, M. Lieberman,
Residues and principal values on complex spaces, 
{\em Math. Ann.} {\bf 194} (1971), 259--294.



\bibitem{irigoyen3} A. irigoyen, An approximation formula for holomorphic functions
by interpolation on the ball (2008), http://arxiv.org/abs/0803.4178



\bibitem{loganshepp} B.F. Logan, L.A. Shepp, 
Optimal reconstruction of a function from its projections,
{\em Duke Math. J.} {\bf 42} (1975), 645--659.




\bibitem{scv} {\em Several Complex Variables I:
Introduction to Complex Analysis},
A.G. Vitushkin (ed.), Berlin: Springer (1990).




\end{thebibliography}
\end{document}